\documentclass[11pt]{article}
\usepackage[english]{babel}
\usepackage{titlesec}
\usepackage{geometry}
\usepackage{enumitem}
\usepackage[utf8]{inputenc}
\usepackage{amsmath}
\usepackage{amssymb}
\usepackage{amsthm}
\usepackage[mathcal]{euscript}
\usepackage{color}
\usepackage{tikz-cd}
\usepackage{mathabx}
\usepackage{physics}
\usepackage{MnSymbol}
\usepackage{csquotes}
\usepackage{url}
\usepackage{caption}
\usepackage{hyperref}
\allowdisplaybreaks

\usepackage{xparse}
\let\realItem\item 
\makeatletter
\NewDocumentCommand\myItem{ o }{%
   \IfNoValueTF{#1}%
      {\realItem}
      {\realItem[#1]\def\@currentlabel{#1}}
}
\makeatother

\usepackage{enumitem}    
\setlist[enumerate]{
    before=\let\item\myItem,       
    label=\textnormal{(\arabic*)}, 
}

\numberwithin{equation}{section}

\title{Bi-infinite systems of singularly interacting Brownian particles and the KPZ equation}

\author{Sayan Banerjee, Amarjit Budhiraja, Peter Rudzis}
\date{}

\newtheorem{theorem}{Theorem}[section]
\newtheorem{lemma}[theorem]{Lemma}
\newtheorem{corollary}[theorem]{Corollary}
\newtheorem{proposition}[theorem]{Proposition}
\newtheorem{remark}[theorem]{Remark}
\newtheorem{definition}[theorem]{Definition}
\newtheorem{assumption}{Assumption}[section]

\newtheorem{manualtheoreminner}{Claim}
\newenvironment{claim}[1]{%
  \renewcommand\themanualtheoreminner{#1}%
  \manualtheoreminner
}{\endmanualtheoreminner}

\newcommand{\hX}{\widehat{\mathit{X}}}
\newcommand{\cX}{\widecheck{\mathit{X}}}
\newcommand{\hL}{\widehat{\mathit{L}}}
\newcommand{\cL}{\widecheck{\mathit{L}}}
\newcommand{\hx}{\widehat{\mathit{x}}}
\newcommand{\cx}{\widecheck{\mathit{x}}}
\newcommand{\hB}{\widehat{\mathit{B}}}
\newcommand{\cB}{\widecheck{\mathit{B}}}
\newcommand{\hV}{\widehat{\mathit{V}}}
\newcommand{\cV}{\widecheck{\mathit{V}}}

\newcommand{\hKs}{\widehat{\mathit{K}}^*}
\newcommand{\cKs}{\widecheck{\mathit{K}}^*}

\newcommand{\clu}{\mathcal{U}}
\newcommand{\clc}{\mathcal{C}}
\newcommand{\cli}{\mathcal{I}}
\newcommand{\cle}{\mathcal{E}}
\newcommand{\clz}{\mathcal{Z}}
\newcommand{\clb}{\mathcal{B}}
\newcommand{\clW}{\mathcal{W}}
\newcommand{\clQ}{\mathcal{Q}}
\newcommand{\clD}{\mathcal{D}}
\newcommand{\vp}{\varphi}
\newcommand{\EE}{\mathbb{E}}
\newcommand{\eps}{\varepsilon}
\newcommand{\veps}{\varepsilon}
\newcommand{\clG}{\mathcal{G}}
\newcommand{\clM}{\mathcal{M}}
\newcommand{\clE}{\mathcal{E}}

\newcommand{\clL}{\mathcal{L}}

\newcommand{\clf}{\mathcal{F}}
\newcommand{\clg}{\mathcal{G}}
\newcommand{\PP}{\mathbb{P}}
\newcommand{\ZZ}{\mathbb{Z}}
\newcommand{\half}{[0,\infty)}
\newcommand{\clp}{\mathcal{P}}
\newcommand{\clv}{\mathcal{V}}

\newcommand{\RR}{\mathbb{R}}
\newcommand{\NN}{\mathbb{N}}
\newcommand{\Ksu}{\mathit{K}_\nearrow^*}
\newcommand{\Ksd}{\mathit{K}_\searrow^*}
\newcommand{\Ks}{\mathit{K}^*}
\newcommand{\clB}{\mathcal{B}}

\newcommand{\ohalf}{(0,\infty)}

\newcommand{\osc}{osc}

\newcommand{\Npe}{\mathcal{N}^+_{\varepsilon}}
\newcommand{\Nme}{\mathcal{N}^-_{\varepsilon}}
\newcommand{\tlE}{\tilde{\mathcal{E}}}

\newcommand{\sle}{\prec_{\mbox{\tiny{d}}}}
\newcommand{\clN}{\mathcal{N}}
\newcommand{\clH}{\mathcal{H}}
\newcommand{\init}{\gamma}

\definecolor{purp2}{rgb}{0.6,0.0,0.9}
\definecolor{ab}{rgb}{0.9,0.0,0.6}
\definecolor{sb}{rgb}{0.8,0.2,0.2}
\definecolor{pr}{rgb}{0.6,0.0,0.8}

\begin{document}

\maketitle

\begin{abstract}
We study a bi-infinite system of interacting Brownian particles on the real line with singular asymmetric interactions mediated by the collision local times. Particles perform Brownian motions, and when neighboring particles collide, the associated local time is split in proportions $p$ and $q=1-p$. We first develop well-posedness theory for the particle system, proving pathwise uniqueness and strong existence under natural growth assumptions on the initial configuration and local times.  We also identify a family of stationary distributions for the infinite-dimensional process of gaps between successive particles: for every $\lambda>0$, the product measure with i.i.d.\ Exp$(\lambda)$ gaps is invariant.

Our main result concerns the equilibrium fluctuations of the associated particle-count (height) function in a weakly asymmetric regime.  Taking $p=p_\varepsilon$ so that $p_\varepsilon^{-1}-1=\exp\{\sigma \varepsilon^{1/4}\}$, initializing the interparticle gaps with i.i.d. Exp$(1)$ random variables, and applying a microscopic Hopf--Cole transform to the diffusively rescaled count function, we prove convergence, as $\varepsilon \rightarrow 0$, to the multiplicative stochastic heat equation (SHE) with Brownian exponential initial data.  Equivalently, the logarithm of the limit is the Hopf--Cole solution of the KPZ equation with two-sided Brownian initial data.  The proof combines localization to finite particle subsystems via chains of collisions, Brownian last-passage percolation estimates, a key local time cancellation property, and a martingale problem for a scale-adapted mollification of the Hopf--Cole field, whose space-time regularity is tuned to match that of the limiting SHE. The resulting fluctuation theorem places these Brownian particle systems with asymmetric singular collision dynamics within the KPZ universality class.
\\


\end{abstract}
\section{Introduction}

The Kardar--Parisi--Zhang (KPZ) universality class, which has been conjectured to govern fluctuations and large-scale behavior of an extensive class of physical systems, has been an active subject of research in recent years. At the heart of it lies
the \emph{KPZ universality conjecture}, which asserts that
one-dimensional stochastic systems with local interactions, a conserved density or height-gradient field,
and a weak asymmetric perturbation of a reversible dynamics should have
large-scale fluctuations governed by the KPZ equation
\begin{equation}\label{eq:KPZdef}
\partial_t h = \frac{1}{2} \Delta h + \frac{1}{2}(\partial_x h)^2 - \sigma \dot W,
\end{equation}
where $W$ is the space-time white noise.
 As the KPZ equation by itself is ill-posed, a central
mechanism behind several rigorous instances of this principle is the
microscopic Hopf--Cole transform: after exponentiating a suitably centered and
scaled height function, one obtains convergence to the multiplicative
stochastic heat equation (SHE), whose logarithm defines the Hopf--Cole solution
of the KPZ equation.  This program is well understood for certain weakly asymmetric lattice
models \cite{BertiniGiacomin,DemboTsai2016,Labbe2018,CorwinShenTsai2018,Yang2023}. Other approaches to KPZ limits include energy-solution methods
\cite{goncalves2014,GJS, GubinelliPerkowski2017,DiehlGubinelliPerkowski2017},
which study the stochastic Burgers equation for the slope field and define
its singular nonlinear term through renormalization techniques. Separately, integrable
methods for exactly solvable models
\cite{SS,ACQ2011,corwin2012kardar,WFSbook}
exploit explicit kernels and Fredholm determinant formulas to derive detailed fluctuation results, including 
Tracy--Widom and Airy-process asymptotics. 
Two major relatively recent breakthroughs in the field have been the introduction of regularity structures by Hairer \cite{hairer2013solving} and of paracontrolled distributions by Gubinelli, Imkeller and Perkowski \cite{GIP}. Together these works have established a general framework for the well-posedness and renormalization of singular stochastic PDEs, including the KPZ equation.

In this article, we consider a bi-infinite system of Brownian particles on the real line,
indexed by $\ZZ$, with singular interactions which occur only at collision instances. The particle positions
$\{X_i(t), t\in \half, i\in \ZZ\}$ are ordered in $\RR$ so that $X_i(t) \leq X_{i+1}(t)$ for all $t \in \half, i \in \ZZ$ and satisfy the infinite system of SDEs
\begin{equation} \label{eq:asym_eq0}
dX_i(t)=dB_i(t)+pdL_{(i-1,i)}(t)-qdL_{(i,i+1)}(t),
\qquad t\in\half,\, i\in\ZZ,
\end{equation}
where $p=1-q \in [0,1]$, $\{B_i\}_{i\in\ZZ}$ are independent standard Brownian
motions, and $L_{(i,i+1)}$ is the local time at zero of the gap
$Z_i:=X_{i+1}-X_i$.  {The local time $L_{(i,i+1)}$ repulses the 
$i$-th and $(i+1)$-st particles at collision instances, preventing them from crossing.}  The `mass ratio' parameter $p$ describes
how this {local time-mediated repulsive interaction} at collisions is split between the two particles:
when $p=q=1/2$ the split is symmetric, whereas when $p\ne q$ the split is asymmetric (meaning adjacent particles have `equal masses' and `unequal masses', respectively).  The precise meaning of a
solution of \eqref{eq:asym_eq0} is given in Definition~\ref{def:biinf}. More broadly, see Section \ref{sec:modset} which handles well-posedness of the system \eqref{eq:asym_eq0}.

The main goal of this paper is to connect
the equilibrium fluctuations of an appropriate counting (height) function associated with
\eqref{eq:asym_eq0}, in a weakly asymmetric regime, to the KPZ equation.  Fix $\sigma>0$.  For
$\eps>0$, we choose $p=p_\eps$ so that
\begin{equation}\label{eq:weak-asym-intro}
        p_\eps^{-1}-1=e^{\sigma\eps^{1/4}}.
\end{equation}
Let $\{\tilde X_i^\eps, \tilde L_{(i,i+1)}^\eps\}$ denote the solution to \eqref{eq:asym_eq0} with this choice of $p$ and initial data $\{x_i = \tilde X_i^\eps(0)\}$ distributed according to some $\init \in \clp_0(\RR^\ZZ)$, where $\clp_0(\RR^\ZZ)$ is the space of probability measures on non-decreasing $\ZZ$-indexed sequences in $\RR$. (We will specify the properties that $\init$ needs to satisfy later.) Define the space-time diffusion-scaled processes, for $i \in \ZZ, t \in \half$,
\begin{equation} \label{eq:rescaled_sys}
B_i^\eps(t):=\eps^{1/2}B_i(\eps^{-1}t),\qquad
X_i^\eps(t):=\eps^{1/2}\tilde X_i^\eps(\eps^{-1}t),\qquad
L_{(i,i+1)}^\eps(t):=\eps^{1/2}\tilde L_{(i,i+1)}^\eps(\eps^{-1}t),
\end{equation}
and the rescaled counting function
\begin{equation} \label{eq:rescaled_count}
        N_\eps(t,x)=\max\{i\in\ZZ:X_i^\eps(t)\le x\},
        \qquad (t,x)\in\half\times\RR .
\end{equation}
The microscopic Hopf--Cole transform studied here is
\begin{equation} \label{eq:Gdef0}
\clG^\eps(t,x)=
\exp\left\{
\sigma\eps^{1/4}N_\eps(t,x-\sigma\eps^{-1/4}t)
-\sigma\eps^{-1/4}x+\frac12\sigma^2\eps^{-1/2}t
\right\}.
\end{equation}
We also introduce the spatially mollified field
\begin{equation} \label{eq:tEdef0}
        \tlE_\theta^\eps(t,x)
        =\int_\RR \clG^\eps(t,z)p_{\eps^{1/(2\theta)}}(x-z)\,dz,
        \qquad \theta\in(0,1),
\end{equation}
where $p_s(z)=(2\pi s)^{-1/2}\exp\{-z^2/(2s)\}$ is the Gaussian heat kernel.  

Our main result identifies the limit of \eqref{eq:Gdef0} as $\eps \to 0$ with the solution
of the multiplicative stochastic heat equation (SHE)
\begin{equation}\label{eq:she-intro}
\begin{cases}
&\partial_t u = \frac{1}{2} \Delta u - \sigma u \dot W, \\
&u_0(x) = \exp\{\sigma B(x)\}, \; x \in \RR,
\end{cases}
\end{equation}
where $W$ is space-time white noise and $B$ is a two-sided Brownian motion,
independent of $W$.  Equivalently, $h:= \log u$ is the Cole--Hopf solution of the
KPZ equation \eqref{eq:KPZdef} with
$h_0(x) = \sigma B(x)$.
More precisely, Theorem~\ref{thm:maincgce} proves that, when the initial distribution
$\init$ is chosen such that the gaps between adjacent particles are i.i.d. Exp$(1)$,
$\tlE_\theta^\eps$ converges in distribution to $u$ in
$C([0,\infty);\clc(\RR))$, and the difference between
$\tlE_\theta^\eps$ and $\clG^\eps$ vanishes locally uniformly, in time and space, in $L^r$ for any $r \ge 1$ (and hence in probability).
Consequently the original, non-mollified microscopic Hopf--Cole transform
$\clG^\eps$ converges to the same SHE limit, uniformly on compact subsets of
$\half\times\RR$ in distribution. 
A key feature of our construction is that the mollification $\tlE_\theta^\eps$ is matched to the fluctuation scale of the problem. As discussed in 
Remark \ref{rem:Ereg}, $\tlE_\theta^\eps$  has exactly the same H\"{o}lder regularity as the limiting stochastic heat equation for every $\eps \in (0,1]$, so that the regularization preserves the microscopic roughness expected in the limit.

The proof of this fluctuation theorem requires us to develop a well-posedness theory for the bi-infinite system of Brownian particles with asymmetric collision dynamics, allowing the fluctuation analysis to be carried out directly at the level of the infinite system.
We first prove pathwise uniqueness
and strong existence for the bi-infinite system \eqref{eq:asym_eq0} under
natural one-sided growth assumptions on the initial configuration and a mild
condition on the growth of local times at $-\infty$; see
Theorems~\ref{thm:unique} and~\ref{thm:exist}.  These results are obtained by
combining finite-particle approximations with control of \emph{chains of collisions}. The latter are related to Brownian last-passage percolation quantities whose control allows one to localize the trajectory of any fixed finite block of particles to a finite system over compact time intervals; see Lemma \ref{lem:finsys}. 

Another preliminary step, of interest in its own right, is to identify a family of stationary measures for the corresponding gap
process $Z_i(t)=X_{i+1}(t)-X_i(t), \, i \in \ZZ, t \ge 0$. We show that, for each $\lambda>0$, the
bi-infinite product measure under which the $Z_i$ are i.i.d. Exp$(\lambda)$
is stationary; see Theorem~\ref{thm:stationary}.  In the fluctuation theorem, 
for simplicity, we use
$\lambda=1$.  Under this equilibrium initialization, the spatial fluctuations
of the initial counting function are Brownian, which explains the initial
condition $u_0(x)=\exp\{\sigma B(x)\}$ in \eqref{eq:she-intro}.


\subsection{Background and comparison with related work}
\label{ssec:intro-background}

Systems of singularly interacting Brownian particles and related rank-based Brownian
particle systems have been studied extensively in connection with stochastic
portfolio theory, reflected Brownian motion, queueing models, and interacting
particle systems.  The Atlas model and its variants were introduced in the
context of equity markets in \cite{banner2005atlas,PP}; see also
\cite{fernholz2002stochastic,fernholz2009stochastic}.  Finite systems with
rank-dependent coefficients and asymmetric collisions were analyzed in
\cite{karatzas2016systems,sarantsev2017infinite,AS2,BB}.  At the level of gaps, these
models are semimartingale reflected Brownian motions in orthants, and the
existence of product-form exponential stationary distributions is closely
related to the skew-symmetry condition of  \cite{HR,harrison1987multidimensional,harrison1987brownian,williams1995semimartingale} which identifies a class of reflected Brownian motions with product form invariant distributions.

Infinite systems present additional difficulties; even well-posedness of such systems is unresolved in full generality. 
One-sided infinite systems (i.e. particles indexed by $\mathbb{N}_0$) of singularly interacting
Brownian particles were studied in \cite{PP,sarantsev2017infinite,sarantsev2017stationary,tsai_stat}. In particular, it was shown that the associated
gap process for such systems has a one-parameter family of stationary distributions. Domains of attraction and extremality of these invariant measures, in the symmetric case $p=q=1/2$, were further investigated in
\cite{DJO,banerjee2022domains,banerjee2024extremal}.  Two-sided infinite systems with $p=q=1/2$, arising from order statistics of rank-based diffusions, 
were considered in \cite{sarantsev2017two}, where one- and two- parameter families of stationary gap
distributions were constructed for models with general drift and diffusion coefficients.
For zero drift and unit variance, these coincide with the family of i.i.d. exponential gap distributions appearing in the present paper. The present work thus shows that this product-form structure persists in the asymmetric regime $p\neq q$.
{Solutions in the $p = 0$ (totally asymmetric) case were constructed in \cite[Chapter 2]{WFSbook}.}
For general $p > 0$, until very recently, existence of solutions was established only when $p \in [1/2,1)$, in the one-sided case, via so-called `approximative versions', {i.e.\ solutions obtained through approximation by finite-dimensional systems, exploiting certain monotonicity properties \cite{sarantsev2017infinite}. }
Existence was still open {when $p \in (0,1/2)$, while} uniqueness was unresolved for any $p \in [0,1)$. Recently, in \cite{banbudrud2025wp}, we resolved the problems of existence and uniqueness, for the one-sided system, for all $p \in [0,1)$, under mild assumptions on the initial data and local time growth. The key idea was to quantify the influence of far-away particles on the lowest few particles by analyzing `chains of collisions', combining comparison techniques for reflected systems with ideas from Brownian last-passage percolation and random matrix theory. Our well-posedness results for the two-sided system, given in Theorems~\ref{thm:unique} and~\ref{thm:exist}, build  on this prior work, although additional care needs to be taken for two-sided systems, where one needs to control chains of collisions in both down-right and up-right directions (see Section \ref{sec:wellp}).

The local time splitting in \eqref{eq:asym_eq0} is the 
singular Brownian 
analogue of asymmetric local interactions in one-dimensional growth models. In the world of  Brownian interacting particle systems, our results are closest in spirit to the works \cite{SS,WFSbook,DiehlGubinelliPerkowski2017}. 
In \cite{SS}, the one-sided version of our model was analyzed for a fixed value of $p$ and half-Poisson initial
data. Using exact calculations
based on Bethe ansatz, self-duality, and Fredholm determinant formulas, the
authors obtained a formula for the generating function of a particle-counting
observable and  connected its asymptotics to the KPZ universality class.
In comparison to the present article, their work exploits the integrable structure of the model; 
the focus is not weak-asymmetry
convergence to the stochastic heat equation or KPZ equation.  
The totally asymmetric cases $p \in \{0,1\}$ were analyzed in \cite{WFSbook} for the two-sided system, again using exact formulas via last-passage percolation and determinantal techniques, to obtain Tracy--Widom and Airy-process fluctuation limits. As in the earlier works on integrable systems, the emphasis in \cite{WFSbook} is on exact distributional asymptotics rather than on random-field convergence to the stochastic heat equation or KPZ equation.
The paper \cite{DiehlGubinelliPerkowski2017} establishes a stationary stochastic Burgers limit for stationary  weakly asymmetric
interacting Brownian motions with \emph{smooth} nearest-neighbor gradient
interactions, using the energy-solution theory \cite{goncalves2014,GJS, GubinelliPerkowski2017}.  Our
model is singular rather than smooth-gradient: the interaction occurs only
through weakly asymmetric local time reflection at collisions. Also, we do not use energy-solution methods.
A further distinction from \cite{DiehlGubinelliPerkowski2017} is the fluctuation observable. Their result concerns the fluctuation field of the local gradient (or gap) variables and yields the stochastic Burgers equation in the scaling limit. By contrast, our result is formulated at the level of the particle counting (height) function itself and, through a microscopic Hopf--Cole transform, yields convergence directly to the stochastic heat equation.

The convergence of microscopic Hopf--Cole transforms to the SHE goes back to the work of Bertini and Giacomin \cite{BertiniGiacomin} on weakly asymmetric exclusion processes (WASEP). This seminal result established convergence of a microscopic particle system to the stochastic heat equation and thereby connected weakly asymmetric particle dynamics with the KPZ universality class. It has since motivated KPZ/SHE scaling limits for a broad range of discrete and continuous models; see, for example, the survey \cite{corwin2012kardar}. Many such results
concern exclusion, zero-range, or related lattice systems, while considerably
fewer address Brownian particle systems or singular interactions.

Although the limiting object in our proof is obtained through a suitable
microscopic Hopf--Cole transform, the microscopic objects and estimates
involved here are substantially different from those in Bertini and
Giacomin~\cite{BertiniGiacomin}. In WASEP, the height function is
obtained by integrating the stationary occupation variables over physical
space; after fixing one reference height, it is completely determined by the
stationary microscopic field. The Hopf--Cole transform is therefore naturally
attached to a fixed spatial frame. In the present
continuum model, the corresponding object is the counting function
$N^\eps(t,x)$, whose distributional gradient is the random point measure of
particle locations. Consequently, the natural stationary description is instead given by the
labelled gap process. These gaps do not determine the counting function unless a
reference particle location is also specified. Thus the Hopf--Cole transform is
not itself built from a stationary field. The stationarity available to us lives in a ‘random moving
frame’, namely the frame viewed from the tagged particle initially placed at the origin; the
motion of this tagged particle is itself correlated with the surrounding
stationary gap process.
Before fluctuations become visible,
one must first remove both the characteristic drift and the background
particle density, leading to the moving-frame and renormalization corrections
appearing in \eqref{eq:Gdef0}.

Moreover, in models such as WASEP \cite{BertiniGiacomin} the weak asymmetry enters through a regular bias in the jump rates, and from the properties of the corresponding discrete generator it follows that the microscopic Hopf--Cole transform $\xi_t(x)$ satisfies an almost closed discrete stochastic heat
equation,
\[
  d\xi_t(x)
  =
  \frac12 e^{\gamma_\eps}\Delta \xi_t(x)\,dt
  + dM_t(x).
\]
Here, $\{M_t(x) : t \ge 0\}$ is a martingale for each fixed $x \in \ZZ$, having an explicit representation of the associated bracket process $\{\langle M(x), M(y)\rangle_t : t \ge 0\}$  for $x,y \in \ZZ$ in terms of $\xi_t$ and its discrete spatial derivatives. This provides a direct starting point for moment bounds, tightness, and
identification of the limit.
In contrast, the asymmetry in our
model is generated through local time accumulation at collision
times, a set of Lebesgue measure zero, rather than
through regular jump-rate biases. To handle this difficulty, our proofs proceed through the carefully constructed $\eps$-dependent mollified observables introduced in \eqref{eq:tEdef0} (in contrast with linearly interpolated fields appearing in \cite{BertiniGiacomin}), which satisfy an evolution equation (see \eqref{eq:prelim1}) that involves both the Hopf--Cole field $\clG^\eps$ and its mollified approximation $\tlE_\theta^\eps$. This introduces substantial challenges in obtaining the required moment and space-time continuity estimates.



The stationary regime considered here is also significant.  Under i.i.d.
exponential gaps, the initial counting function has Brownian fluctuations,
which lead to the Brownian exponential initial condition in \eqref{eq:she-intro}.
This is the equilibrium, or stationary, KPZ/SHE initial data.  The resulting
limit theorem is therefore
a stationary KPZ fluctuation theorem
for asymmetric singularly interacting Brownian particles. Although the results of this paper are established under equilibrium initial conditions, several components of the analysis do not use stationarity. We remark at appropriate places which parts of the proof continue to hold for more general initial conditions.

We briefly note that the totally asymmetric (i.e. $p=0$) reflected Brownian system has also been connected directly to the KPZ fixed point. Building on its integrable structure, \cite{NicaQuastelRemenik2020} obtained Fredholm determinant formulas for the multi-point distributions of the process and proved that its 1:2:3-scaled fluctuations converge to the KPZ fixed point. Thus the reflected Brownian model provides a continuous-space analogue of the exactly solvable lattice models whose large-scale behavior is governed by the KPZ fixed point. See also \cite{DauvergneOrtmannVirag2022} for the subsequent directed-landscape perspective.
Exact distributional formulas for the stationary KPZ equation with two-sided Brownian initial data were obtained in \cite{BorodinCorwinFerrariVeto2015}, including the identification of the stationary (Baik--Rains) fluctuation regime. More recently, \cite{DasDrillickParekh2024} established KPZ-type scaling
limits for systems of sticky Brownian motions, providing another example of
KPZ behavior arising from local time interactions. While local times
play a central role in both our work and in \cite{DasDrillickParekh2024}, the mechanism in the latter work is based on a
Girsanov representation of sticky Brownian flows, whereas our analysis
depends crucially on delicate cancellations among collision local times and
the resulting structure of the counting-function Hopf--Cole transform.


For the one-sided symmetric case, corresponding to
$p=q=1/2$, equilibrium fluctuations exhibit very different limiting behavior. For
homogeneous Poissonian initial profiles, the diffusively scaled fluctuations
were shown in \cite{dembo2017equilibrium} to converge to the stochastic heat
equation with additive white noise.  For inhomogeneous Poissonian initial
profiles, \cite{BBR2026AtlasFluctuations} showed that the limiting
fluctuations are governed by a linear stochastic partial differential equation (SPDE) driven by additive noise that is
white in time and inhomogeneous in space, and the linear operator governing the evolution is the infinitesimal generator of a geometric Brownian motion. 

\subsection{Proof of SHE limit: key steps}

{The proof of the convergence to SHE is carried out in Section \ref{sec:tight-conv}. The argument has three parts: establishing tightness of the collection $\{\tlE_\theta^\eps\}_{\eps > 0}$ in $C(\half \times \RR : \RR)$; showing that the microscopic Hopf--Cole field $\clG^\eps$ is approximated uniformly on compacts in probability by its mollification $\tlE_\theta^\eps$, as $\eps \to 0$; and showing that any limit point of $\{\tlE_\theta^\eps\}_{\eps > 0}$ satisfies a certain martingale characterization of solutions to SHE.}


Among the benefits of working with the prelimit field $\tlE_\theta^\eps$ is the observation
that this
functional admits a representation as a certain weighted infinite sum (see \eqref{eq:Edef}, \eqref{eq:te-def}, and Lemma~\ref{lem:Erep}).  The semimartingale decomposition of this sum has a \emph{local time cancellation
property}, which yields the martingale representation in
\eqref{eq:mart_rep}. The weighted sum in \eqref{eq:Edef} further contains a
tunable parameter $\delta$, which is chosen as suitable functions of $\eps$ in key
moment and tightness estimates in order to control the discontinuities coming from the collision local times. {This martingale representation and other fundamental properties of the mollified field are presented in Subsection \ref{ssec:weighted}. The moment bounds, continuity estimates, and tightness argument appear in Subsections \ref{ssec:Gmoments}-\ref{ssec:conttight}.} 

{In Subsection \ref{ssec:unifclose}, we establish that $\tlE_\theta$ approximates $\clG^\eps$ in an appropriate sense, ensuring that $\tlE_\theta^\eps$ and $\clG^\eps$ will exhibit like asymptotics.} 
The justification of this approximation
relies on a detailed analysis of the fluctuations of the rescaled counting
function \eqref{eq:rescaled_count}, carried out in Section
\ref{ssec:countfluc}, together with continuity estimates for $\clG^\eps$
obtained in Lemmas~\ref{lem:G} and~\ref{lem:RR}. Finally, in Subsection \ref{ssec:martprob}, using the weighted sum representation and local-time cancellation property, we identify the asymptotics of $\tlE_\theta^\eps$ {as given by SHE} through an associated martingale problem.

{Although our main result Theorem \ref{thm:maincgce} is stated for initial data such that the gaps between adjacent particles are i.i.d. Exp(1), in many of our lemmas this assumption can be dispensed with in place of weaker conditions on the initial data. We comment on possible extensions to more general initial data in Remarks \ref{rem:gentightness}, \ref{rem:gencont}, and \ref{rem:genconvE}.}


\subsection{Notation}
\label{sec:not}
The following notation will be used. {For real numbers $x,y$, $x \wedge y = \min\{x,y\}$, $x \vee y = \max\{x,y\}$, and $y^+ = \max\{0,y\}$.} We will denote the statement that a random variable $X$ has distribution $\pi$ as $X \sim \pi$. We denote by $\clc(\RR)$ (resp. $C([0,T) \times \RR)$) the space of continuous real functions on $\RR$ (resp. $[0,T) \times \RR$) equipped with the local uniform topology. $C^{0,1}([0,T] \times \RR)$ will denote the space of functions that are continuous in the first coordinate (time) and once continuously differentiable in the second coordinate (space).
For a Polish space $S$, $C([0,\infty); S)$ will denote the space of continuous functions from $[0,\infty)$ to $S$ equipped with the local uniform topology. The Borel $\sigma$-field on a topological space $S$ will be denoted as $\clb(S)$. For $p>0$, the $L^p$ norm of a random variable $X$ on a probability space $(\Omega, \clf, \PP)$ will be denoted as $\|X\|_p$, namely $\|X\|_p := \left(\EE(|X|^p)\right)^{1/p}$. For real valued random variables $X,Y$, we write $X\sle Y$ if $\PP(X \le z) \ge \PP(Y\le z)$ for all $z \in \RR$. Convolution of functions $f$ and $g$ on the real line will be denoted as $f*g$, i.e. $f*g(x) = \int_{\RR}f(x-y)g(y) dy$. We denote $p_t(x):= \frac{1}{(2\pi t)^{1/2}} e^{-x^2/2t}$ and let $p_t(x,y):= p_t(x-y)$ for $t>0$ and $x,y \in \RR$.

\section{Model and Main Results} \label{sec:modset}

\subsection{Bi-infinite systems of singularly interacting Brownian particles} \label{ssec:biinfinite}

Throughout this section, we fix $p = 1 - q \in (0,1/2)$. Let $(\Omega, \clf, \{\clf_t\}_{t \in \half}, \PP)$ be a filtered probability space, and let $\{B_i\}$ be a collection of i.i.d. standard $\{\clf_t\}$-Brownian motions. That is, $B_i$ are independent standard Brownian motions adapted to $\{\clf_t\}$ and satisfy the property that $\{B_i(t+\cdot)- B_i(t), i \in \ZZ\}$ is independent of $\clf_t$ for all $t\ge 0$. 
Let $x = \{x_i\}_{i \in \ZZ}$ be a non-decreasing sequence of $\clf_0$-measurable random variables.

\begin{definition}\normalfont
\label{def:biinf}
Let $X = \{X_i(t), t \in \half\}_{i \in \ZZ}$, $L = \{L_{(i,i+1)}(t), t \in \half\}_{i \in \ZZ}$ be collections of continuous, real-valued $\{\clf_t\}$-adapted processes having the properties 
\begin{enumerate}[label = \arabic*.]

\item \label{p:1} For all $t \in \half, i \in \ZZ$,
\begin{equation} \label{eq:asym_eq}
X_i(t) = x_i + B_i(t) + pL_{(i-1,i)}(t) - qL_{(i,i+1)}(t).
\end{equation} 

\item \label{p:2} For all $t \in \half$, the sequence $\{X_i(t)\}_{i \in \ZZ}$ is non-decreasing.

\item \label{p:3} For all $i \in \ZZ$, $L_{(i,i+1)}$ is a continuous, non-decreasing process such that $L_{(i,i+1)}(0) = 0$ and $L_{(i,i+1)}$ can only increase at times $t \in \half$ such that $X_{i+1}(t) = X_i(t)$, i.e. 
\[
\int_0^\infty (X_{i+1}(t) - X_i(t))dL_{(i,i+1)}(t) = 0.
\]
\end{enumerate}
We refer to $X = \{X_i\}$ as a \textit{bi-infinite system of singularly interacting Brownian particles with parameter $p$, initial condition $x$, and driving Brownian motions $\{B_i\}$.} The processes $\{L_{(i,i+1)}\}$ are called the \textit{local times} associated with the system $X$.
\end{definition}
Informally, by a \textit{solution to \eqref{eq:asym_eq}} we mean a process $X$ and associated local times $L$, given on some filtered probability space, and $x$, $\{B_i\}_{i \in \ZZ}$, as above, such that the properties in Definition \ref{def:biinf} are satisfied.

Let $\clp_0(\RR^{\ZZ})$ denote the space of probability measures on $\RR^{\ZZ}$, under which, for a.e. $z = \{z_i\} \in \RR^{\ZZ}$, 
$z_i \leq z_{i+1}$ for all $i \in \ZZ$. We say there \textit{exists a strong solution to \eqref{eq:asym_eq} for  initial distribution $\gamma \in \clp_0(\RR^{\ZZ})$} if for any choice of a filtered probability space, Brownian motions $\{B_i\}_{i \in \ZZ}$, and initial $x$ as above, with $x$ distributed as $\gamma$, there exist $X$ and $L$ as in Definition \ref{def:biinf}. We say that solutions to \eqref{eq:asym_eq} are \textit{pathwise unique for an initial distribution $\gamma$} if for any filtered probability space, and Brownian motions $\{B_i\}_{i \in \ZZ}$, and initial date $x$ distributed as $\gamma$, as above, if $(X,L)$ and $(X',L')$ are two solutions to \eqref{eq:asym_eq} such that $\{X_i(0)\} = \{X_i'(0)\} = x$, then almost surely $X = X'$ and $L = L'$. We will also talk about  \textit{pathwise uniqueness among solutions satisfying some property $(\clp)$}, by which we mean that if, whenever $(X,L)$ and $(X',L')$ are two  solutions to \eqref{eq:asym_eq} as above, satisfying property $(\clp)$,  then, almost surely, $(X,L) = (X',L')$.

In \cite{banbudrud2025wp} (see Theorems 2.3 and 2.7 therein), it is shown that the analogous one-sided (indexed by $\NN$) infinite system of singularly interacting Brownian particles with parameter $p \in [0,1]$ is well-posed (namely, strong existence and pathwise uniqueness holds) when the initial data and  the local times satisfy certain growth assumptions (the latter are needed only for the case $p \geq 1/2$). 

We now present a well-posedness result for the bi-infinite system that uses similar methods as in \cite{banbudrud2025wp}. 
{Many proof ideas are the same, and in such cases we only provided sketches of the proofs.}
These sketches of pathwise uniqueness and strong existence, are provided in Sections \ref{sec:pathuniq} and \ref{sec:strexi}, respectively.

\begin{theorem}[Pathwise uniqueness] \label{thm:unique}
Fix $\gamma \in \clp_0(\RR^{\ZZ})$, and let $x = \{x_i\}_{i \in \ZZ} \sim \gamma$. Assume that, $\gamma$-almost surely,  
\begin{equation} \label{eq:ainit}
\limsup_{M \to \infty} \frac{-x_{-M}}{\sqrt{M}} = \infty.
\end{equation}
Consider the following condition on the local times for some solution $(X,L)$ of \eqref{eq:asym_eq}:
\begin{equation} \label{eq:aloc}
\forall T \in \half, \quad \limsup_{M \to \infty} \left(\frac{q}{p}\right)^{-M} L_{(-M,-M+1)}(T) = 0 \text{ a.s.}
\end{equation}
Then, solutions to \eqref{eq:asym_eq} are pathwise unique for parameter $p$ and initial distribution $\gamma$, among all solutions whose associated local times satisfy \eqref{eq:aloc}.
\end{theorem}

\begin{theorem}[Strong existence] \label{thm:exist}
Fix  $\gamma \in \clp_0(\RR^{\ZZ})$. Assume that for some $\chi > 1/2$, 
\begin{equation} \label{eq:ainit2}
\sum_{k = 1}^\infty \gamma(-x_{-k} \leq k^\chi) < \infty.
\end{equation}
There exists a strong solution $(X,L)$ to \eqref{eq:asym_eq},  with parameter $p$ and initial distribution $\gamma$,
which satisfies \eqref{eq:aloc}.
\end{theorem}

Next we introduce stationary distributions associated with the gaps between the neighboring particles in the above infinite particle system.

Let $(X,L) = \{(X_j,L_{(j,j+1)})\}_{j \in \ZZ}$ be a solution to \eqref{eq:asym_eq}. Let $Z = \{Z_j\}_{j \in \ZZ}$ be the process taking values in the orthant $\half^{\ZZ}$, defined by 
\begin{equation}\label{eq:Zdefn}
Z_j(t) = X_{j+1}(t) - X_j(t), \quad t \in \half, \quad j \in \ZZ.
\end{equation}
We refer to $Z$ as the \textit{gap process} associated with $(X,L)$.

Fix $\lambda>0$, and let $z = \{z_i\}_{i \in \ZZ}$ be a bi-infinite sequence of i.i.d. Exp($\lambda$) random variables. Let $\pi_\lambda$ denote the distribution of $z$.   Let $\gamma_{\lambda} \in \clp_0(\RR^{\ZZ})$ be such that if $x = \{x_i\}_{i\in \ZZ} \sim \gamma_{\lambda}$, then 
\begin{equation} \label{eq:hominit}
  x_0 = 0, \quad \text{ and } \quad \{x_{i+1}-x_{i}\}_{i\in \ZZ}\sim \pi_\lambda.
\end{equation}
Note that $\gamma_{\lambda}$ satisfies \eqref{eq:ainit2} for any $1/2 < \chi <1$ and consequently $x$ satisfies \eqref{eq:ainit}. Let $X$ be the unique strong solution of \eqref{eq:asym_eq}
that satisfies \eqref{eq:aloc}, which is guaranteed by Theorems \ref{thm:unique} and
\ref{thm:exist}. Let $Z = \{Z_j\}_{j \in \ZZ}$ be defined by \eqref{eq:Zdefn}. 

\begin{theorem} \label{thm:stationary}
For each $\lambda>0$, the measure $\pi_{\lambda}$ is a stationary distribution for the process $Z$. That is, if $Z(0) \sim \pi_\lambda$, then  $Z(t) \sim \pi_{\lambda}$ for all $t \in \half$.
\end{theorem}

Our  main goal is to study fluctuations of the infinite particle system, in a weakly asymmetric regime,  by considering the associated counting process as defined in \eqref{eq:rescaled_count}, when the initial gaps are given by the above stationary distribution.
For simplicity, we  consider the stationary distribution $\gamma_1$ (i.e. $\lambda=1$). The general case can be treated similarly.

By a weakly asymmetric regime we mean the following. Fix $\sigma> 0$. For a given $\eps>0$ we consider the infinite particle system in  Definition \ref{def:biinf} with asymmetry parameter $p= p_{\eps} \in (0,1)$  such that  $p^{-1}-1 = e^{\sigma\eps^{1/4}}$. Thus, we will consider a collection of infinite particle systems, parametrized by the scaling parameter $\eps$, where each system is initialized with the  distribution  $\gamma_1$ which is stationary for the associated gaps.

Our main result will give a limit theorem for the exponential transform of the suitably scaled and translated counting process associated with the infinite particle system, as defined in \eqref{eq:Gdef0}. The limit will be characterized by a solution of the {\em stochastic heat equation} (SHE). We now give a precise description of this limit object.

Let $W$ and $B$ be independent white noise measures on $\RR_+\times \RR$ and $\RR$, respectively.  
The stochastic heat equation that arises from the distributional limit of the field $\clg^{\eps}$ in \eqref{eq:Gdef0}, is as follows.
\begin{equation}\label{eq:mshe}
\begin{cases}
&\partial_t u = \frac{1}{2} \Delta u - \sigma u \dot W, \\
&u_0(x) = \exp\{\sigma B(x)\}, \; x \in \RR.
\end{cases}
\end{equation}
We note that $u_0$ satisfies the following integrability property: For each $p>0$, there is an $a>0$ such that
\begin{equation*}
\sup_{x\in \RR} e^{-a|x|} \EE|u_0(x)|^p <\infty.
\end{equation*}
Using this fact we see from \cite[Theorem 3.1]{BertiniGiacomin} that there is a unique $u \in C([0, \infty); \clc(\RR))$  which solves \eqref{eq:mshe}
in the mild sense, namely
\begin{equation*}
u_t(x) = (p_t*u_0)(x) - \sigma\int_{(0,t)\times \RR} p_{t-s}(y-x) u_s(y) W(ds\, dy), \; t>0, x \in \RR.
\end{equation*}
Furthermore, $u_t(x)>0$ for every $(t,x) \in \RR_+\times \RR$ and
for each $t\ge 0$, $u_t(\cdot)$ is 
$$\clf_t = \sigma\{B, W([0,s]\times (A\cap [-n,n])), 0\le s \le t, n \in \NN, A \in \clb(\RR)\}$$ measurable.
Denote the measure induced by $u$ on $C([0, \infty); \clc(\RR))$ by $\clQ$.

We now state the precise sense in which $\clg^{\eps}$ converges to $u$.

We will say a random field $\{\clz^{\veps}(t,x), (t,x) \in \RR_+\times \RR\}$ converges, uniformly on compact subsets of $\RR_+ \times \RR$, in distribution to a continuous random field
$\{\clz(t,x), (t,x) \in \RR_+ \times \RR\}$, if there exists a sequence $\{\tilde \clz^{\veps}(t,x), (t,x) \in \RR_+ \times \RR\}$ of continuous random fields such that $\tilde \clz^{\veps} \to \clz$ in distribution in $C([0, \infty); \clc(\RR))$, as $\veps \to 0$, where the latter space is equipped with the topology specified in Section \ref{sec:not}, and 
$$\sup_{0\le t \le T} \sup_{x \in K} |\tilde \clz^{\veps}(t,x) - \clz^{\veps}(t,x)| \to 0 \mbox{ in probability, as } \veps\to 0,$$
for every $T \in (0,\infty)$ and compact $K \subset \RR$.

In our case, the role of $\clz^{\veps}$ (resp. $\tilde \clz^{\veps}$, $\clz$) will be played by
$\clg^{\eps}$ (resp. $\tilde \cle^{\eps}_{\theta}$, $u$).  Namely we prove the following result. The proof appears in Section \ref{ssec:martprob}.
\begin{theorem}
\label{thm:maincgce}
As $\eps \to 0$, $\tilde \cle^{\eps}_{\theta}$ converges to $u$, in distribution, in $C([0, \infty); \clc(\RR))$.
Furthermore, for any $r \ge 1$, $T<\infty$ and compact $K\subset \RR$,
\begin{equation}\label{eq:uniapprox}
   \left\|\sup_{t \in [0,T], x \in K} \left|\tilde \clE_{\theta}^\eps(t,x) - \clG^\eps(t,x) \right| \right\|_r \to 0 \text{ as } \eps \to 0.
\end{equation}
Consequently, $\{\clg^{\veps}(t,x), (t,x) \in \RR_+\times \RR\}$ converges, uniformly on compact subsets of $\RR_+ \times \RR$, in distribution to the random field
$\{u(t,x), (t,x) \in \RR_+ \times \RR\}$.
\end{theorem}

\begin{remark}[Tightness and convergence of $\tlE_\theta^\eps$ with general initial data] \label{rem:gentightness} \normalfont
{
While our main result requires the initial data to be distributed as $\gamma_1$, this assumption enters  primarily in proving the uniform limit \eqref{eq:uniapprox}. 
Tightness of the collection $\{\tlE_\theta^\eps,\eps > 0\}$ in $C(\half \times \RR : \RR)$ can be seen to hold for  more general initial data. 
This is described in detail in Remark \ref{rem:gencont}.

In addition to proving tightness, it should be possible to establish convergence of $\tlE_\theta^\eps$ to the stochastic heat equation under more general assumptions on the initial data. However, this does not follow immediately from our arguments. What is needed additionally to establish such convergence is described in Remark \ref{rem:genconvE}. 

To establish convergence of the unmollified field $\clG^\eps$ to SHE for more general initial data, one would need to prove the convergence \eqref{eq:uniapprox} without appealing to properties of the distribution $\gamma_1$, a task which lies beyond the scope of the methods of this paper.
}
\end{remark}

\section{Well-posedness} \label{sec:wellp}
In this section we prove Theorems \ref{thm:unique} and \ref{thm:exist}.
The main ingredient in the proof of the first result on pathwise uniqueness is to understand certain events associated with chains of collisions and analyze their probabilities using Brownian last-passage percolation estimates.

\subsection{Chains of collisions and Brownian last-passage percolation}
The idea of controlling the lengths of chains of collisions to establish pathwise uniqueness of solutions was introduced in \cite{banbudrud2025wp}. Consider a strong solution $(X,L)$ of \eqref{eq:asym_eq}.  Define,
for $0 \leq u < v < \infty$ and $i \in \ZZ$, 
\[
\begin{split}
\Ksd(i,[u,v]) = \sup\{k \geq 0 & : \exists u \leq s_{i + k - 1} \leq \cdots \leq s_{i + 1} \leq s_i \leq v \\
& \text{ such that, for } i \leq j \leq i + k - 1, X_{j+1}(s_j) = X_j(s_j)\},
\end{split}
\]
\[
\begin{split}
\Ksu(i,[u,v]) = \sup\{k \geq 0 & : \exists u \leq s_{i - k + 1} \leq \cdots \leq s_{i - 1} \leq s_i \leq v \\
& \text{ such that, for } i-k+1 \leq j \leq i, X_{j-1}(s_j) = X_j(s_j)\},
\end{split}
\]
and for $i_1 \leq i_2$ in $\ZZ$, let
\[
\Ks([i_1,i_2],[u,v]) = \max\left\{\Ksd(i_2,[u,v]), \Ksu(i_1,[u,v])\right\}.
\]
The quantity $\Ksd(i,[u,v])$ (resp.\ $\Ksu(i,[u,v])$) represents the maximal number of collisions in the time interval $[u,v]$ in a down-right (resp.\ up-right) chain of collisions ending at the $i$-th particle.  {The significance of these quantities is seen from Lemma \ref{lem:finsys} which says that on the event $\{\Ks([i_1,i_2],[0,T]) \le M\}$, for any $j \in [i_1, i_2]$, 
$X_j$ agrees up to time $T$ with a particle in a finite system of singularly interacting Brownian particles.} 

For $i_1 \leq i_2$ in $\ZZ$, let $\{X_j^{[i_1,i_2]}(t), t \in \half\}_{i_1 \leq j \leq i_2}$ to be the finite system of  Brownian particles on the real line, given as the solution to the system of equations
\begin{equation} \label{eq:finfam}
X_j^{[i_1,i_2]}(t) = x_j + B_j(t) + p L_{(j-1,j)}^{[i_1,i_2]}(t) - q L_{(j,j+1)}^{[i_1,i_2]}(t), \quad t \in \half, \quad i_1 \leq j \leq i_2,
\end{equation}
where $L_{(i_1-1,i_1)}^{[i_1,i_2]} \equiv L_{(i_2,i_2+1)}^{[i_1,i_2]} \equiv 0$; and for $i_1 \leq j \leq i_2-1$, $L_{(j,j+1)}^{[i_1,i_2]}$ is continuous, non-decreasing, equal to zero at $t = 0$, {increasing only when $\{X_j^{[i_1,i_2]}(t) = X_{j+1}^{[i_1,i_2]}(t)\}$}, and such that $X_j^{[i_1,i_2]}(t) \le X_{j+1}^{[i_1,i_2]}(t)$ for all $t\ge 0$. Recall that solutions to finite systems of the type \eqref{eq:finfam} exist in the strong sense and are pathwise unique, so $\{X_j^{[i_1,i_2]}\}$ and $\{L_{(j,j+1)}^{[i_1,i_2]}\}$ are well-defined (for more details, see \cite{HR} and \cite[Section 2.1]{karatzas2016systems}). 

\begin{lemma}\label{lem:finsys}
Fix an initial distribution $\gamma \in \clp_0(\RR^{\ZZ})$, an initial condition $x = \{x_i\}_{i \in \ZZ} \sim \gamma$, and Brownian motions $\{B_i\}_{i \in \ZZ}$ as in Section \ref{ssec:biinfinite}. Let 
$X = \{X_i(t), t \in \half\}_{i \in \ZZ}$, $L = \{L_{(i,i+1)}(t), t \in \half\}_{i \in \ZZ}$ be as in Definition \ref{def:biinf}. 
\begin{enumerate}[label = (\roman*)]
\item For $i_1 \leq i_2$ in $\ZZ$ and $M \in \NN$, on the event that $\Ks([i_1,i_2],[0,T]) \leq M$, for all $i_1 \leq j \leq i_2$ and $t \in [0,T]$, 
$$
X_j(t) = X_j^{[i_1-M,i_2+M]}(t),
$$
where the system $\{X_j^{[i_1-M,i_2+M]}\}$ is defined as in \eqref{eq:finfam}. 

\item Suppose that for any solution to \eqref{eq:asym_eq} with some initial distribution $\gamma$ and 
satisfying some property $(\clp)$, for all $i \in \ZZ$ and $T \in (0,\infty)$, $\Ksu(i,[0,T])$ and $\Ksd(i,[0,T])$ are finite almost surely. Then solutions to \eqref{eq:asym_eq} with initial distribution $\gamma$ and satisfying property $(\clp)$ are pathwise unique.
\end{enumerate}
\end{lemma}
\begin{proof} 
The proof follows along the lines of \cite[Lemma 3.2]{banbudrud2025wp}. 
We just sketch the parts that are similar.

{The basic idea in the proof of Lemma \ref{lem:finsys}(i) is as follows. Particles outside the cluster $\{X_j, i_1 \leq j \leq i_2\}$ can only affect the cluster through an up-right chain of collisions reaching the cluster from below or a down-right chain of collisions reaching the cluster from above. Hence, on the event $\Ks([i_1,i_2],[0,T]) \leq M$, particles outside of the larger cluster $\{X_j, i_1 - M \leq j \leq i_2 + M\}$ are causally isolated from the smaller cluster on the time interval $[0,T]$. Therefore, up to time $T$, the evolution of the smaller cluster is indistinguishable from the evolution we would observe if we were to remove all particles not in the larger cluster. We now give details.} 

Fix $i_1, i_2, M$ as in the statement of part (i) of the lemma. Define sequences of stopping times 
\[
\tau_{i_1-M-1}^- \leq \tau_{i_1-M}^- \leq \cdots \leq \tau_{i_1}^-, \quad\quad \tau_{i_2+M+1}^+ \leq \tau_{i_2+M}^+ \leq \cdots \leq \tau_{i_2}^+
\]
inductively as follows: $\tau_{i_1 - M - 1}^- = \tau_{i_2 + M + 1}^+ = 0$, and for each $i_1 - M \leq j \leq 1$ and $1 \leq k \leq i_2 + M$, 
\[
\tau_j^- = \inf\{t \geq \tau_{j-1} : X_{j-1}(t) = X_j(t)\}, \quad\quad \tau_k^+ = \inf\{t \geq \tau_{k+1} : X_{k+1}(t) = X_k(t)\}.
\]
Define the time-dependent integer interval $\mathcal{J}(t)$ as follows:
\[
\mathcal{J}(t) = [\min\{j : \tau_j^- > t\}, \max\{k : \tau_k^+ > t\}], \quad 0 \leq t < \min\{\tau_{i_1}^-, \tau_{i_2}^+\}.
\]
(Hence, $\mathcal{J}(0) = [i_1 - M, i_2 + M]$, and we remove an element $j$ from the boundary of $\mathcal{J}(t)$ each time an external particle collides with the particle $X_j(t)$ in the cluster of particles $\{X_i(t), i \in \mathcal{J}(t)\}$.) 

Observe that if $K([i_1,i_2],[0,T]) \leq M$, then $\min\{\tau_{i_1}^-, \tau_{i_2}^+\} > T$; otherwise, either the sequence $\{s_j = \tau_j^-\}_{i_1 - M \leq k \leq i_1}$ would violate the definition of $\Ksu(i_1,[0,T])$ or the sequence $\{s_k = \tau_k^+\}_{i_2 \leq k \leq i_2 + M}$ would violate the definition of $\Ksd(i_2,[0,T])$. Also note that $[i_1,i_2] \subset \mathcal{J}(t)$ for all $0 \leq t < \min\{\tau_{i_1}^-, \tau_{i_2}^+\}$. Thus, to prove part (i) of the lemma, it suffices to prove the following result.

\vspace{0.1in}

\begin{claim}{1}
For all $0 \leq t < \min\{\tau_{i_1}^-, \tau_{i_2}^+\}$, on the event that $\mathcal{J}(t) = [j_1^*,j_2^*]$, 
\[
X_i(t) = X_i^{[i_1 - M,i_2 + M]}(t) \text{ for } j_1^* \leq i \leq j_2^*.
\]
\end{claim}

Let 
$$
L = |\{j : 0 < \tau_j^- \leq \min\{\tau_{i_1}^-, \tau_{i_2}^+\}\}| + |\{k : 0 < \tau_k^+ \leq \min\{\tau_{i_1}^-, \tau_{i_2}^+\}\}|.
$$
Define stopping times $\tau_0 \leq \tau_1 \leq \tau_2 \leq \cdots \leq \tau_{L}$ as follows: $\tau_0 = 0$ and for $1 \leq j \leq 2M+1$, $\tau_j$ is the first time after $\tau_{j-1}$ that an element is removed from the interval $\mathcal{J}(t)$; equivalently, 
\[
\tau_i = \inf\{t > 0 : |\{j : 0 < \tau_j^- < t\}| + |\{k : 0 < \tau_k^+ < t\}| \geq i\}, \quad 1 \leq i \leq 2M+1.
\]
The claim will be proved by fixing $\omega \in \Omega$ and establishing (1) for all $0 \leq t < \min\{\tau_{i_1}^-, \tau_{i_2}^+\}$ and all $i \in \mathcal{J}(t)$, $X_i(t) = X_i^{\mathcal{J}(t)}(t)$, and (2) for all $0 \leq t < \min\{\tau_{i_1}^-, \tau_{i_2}^+\}$ and all $i \in \mathcal{J}(t)$, $X_i^{[i_1 - M,i_2 + M]}(t) = X_i^{\mathcal{J}(t)}(t)$. The sketch we provide below assumes that $\tau_k < \tau_{k+1}$ for each $0 \leq k \leq L-1$, but the proof can be easily modified in the case that the inequality is not strict.

We sketch the proof of the first statement (1). We show by induction that for $0 \leq k \leq 2M$, for all $t \in [\tau_k,\tau_{k+1})$, for all $i \in \mathcal{J}(t)$, $X_i(t) = X_i^{\mathcal{J}(t)}$. The base case is obtained as follows: On the time interval $[0,\tau_1)$, $\mathcal{J}(t) = [i_1 - M, i_2 + M]$, and $\{X_i\}_{i \in [i_1 - M, i_2 + M]}$, with associated local times $\{L_{(i,i+1)}\}_{i \in [i_1 - M-1, i_2 + M]}$, can be shown to solve the finite system of equations \eqref{eq:finfam} with the interval $[i_1 - M,i_2 + M]$ replacing $[i_1,i_2]$ (this can be seen from the fact that the boundary local times $L_{(i,i+1)}$ are constant on this time interval). Since both systems have the same initial data, they agree on the interval $[0,\tau_1)$ by pathwise uniqueness for finite systems of singularly interacting Brownian particles. For the induction step, suppose that the equality has been established for all $t < \tau_k$. Note that the interval $\mathcal{J}(t)$ is constant for $t \in [\tau_k,\tau_{k+1})$, say $\mathcal{J}(t) = [i_1^*,i_2^*]$. By the induction assumption and continuity, $X_i(\tau_k) = X_i^{[i_1^*,i_2^*]}(\tau_k)$ for all $i \in [i_1^*,i_2^*]$. Hence, the same argument as in the base case, again using pathwise uniqueness of solutions to the finite system, shows that for all $t \in [\tau_k,\tau_{k+1})$ and $i \in [i_1^*,i_2^*]$, $X_i(t) = X_i^{[i_1^*,i_2^*]}(t)$. This completes the induction step.

The proof of the second statement (2) is the same as the first, except that we replace the process $X$ by $X^{[i_1 - M, i_2 + M]}$. This completes the proof of the claim and the proof of part (i) of the lemma.

The proof of part (ii) of the lemma is a direct consequence of the first part and is completed exactly as in \cite[Lemma 3.3]{banbudrud2025wp}.
\end{proof}

By the lemma, to prove Theorem \ref{thm:unique}, it is enough to show that 
for any solution with initial condition satisfying \eqref{eq:ainit} and the associated local times satisfying \eqref{eq:aloc},
 $\Ksd(i,[0,T])$ and $\Ksu(i,[0,T])$ are both finite almost surely. 

Finiteness of these quantities depends crucially on the possibility of controlling the following quantities from Brownian last-passage percolation. For $i \in \ZZ, M \in \NN, 0 \le s \le t <\infty$, let
\begin{equation} \label{eq:mLPP}
\clB_M^-(i,[s,t]) = \inf_{s \leq s_{i + M - 2} \leq \cdots \leq s_i \leq t} \sum_{j = i}^{i+M-1} (B_j(s_{j-1}) - B_j(s_j)),
\end{equation}
and
\begin{equation} \label{eq:pLPP}
\clB_M^+(i,[s,t]) = \sup_{s \leq t_{i - M + 2} \leq \cdots \leq t_i \leq t} \sum_{j = i - M + 1}^i (B_j(t_{j+1}) - B_j(t_{j})),
\end{equation}
where by convention $s_{i + M - 1} = t_{i-M+1} = s$ and $s_{i-1} = t_{i + 1} = t$ in the sums appearing above. See \cite[Theorem 1]{smalldev2010} and \cite[Lemma 5.1]{banbudrud2025wp}. When $s=0$, we will write the above quantities as simply $\clB_M^{\pm}(i,t)$.

\subsection{Pathwise Uniqueness}
\label{sec:pathuniq}

In this section we provide the proof of Theorem \ref{thm:unique}.
For the proof we will consider certain auxiliary half-infinite systems of singularly interacting particles.

Let $(X,L)$ be a solution to \eqref{eq:asym_eq}. Associated with such a solution are processes $\hX = \{\hX_j\}_{j \in \NN}$, $\hL = \{\hL_{(j-1,j)}\}_{j \in \NN}$, $\cX = \{\cX_j\}_{j \in \NN}$, and $\cL = \{\cL_{(j-1,j)}\}_{j \in \NN}$, defined as follows: For $i \in \NN$ and $t \in \half$, let 
\[
\begin{split}
& \hX_j(t) = X_j(t), \quad \hL_{(j-1,j)}(t) = \begin{cases} 0 & \text{ if } j = 1, \\
L_{(j-1,j)}(t) & \text{ if } j \geq 2,
\end{cases} \\
& \cX_j(t) = -X_{-j}(t), \quad \cL_{(j-1,j)} = \begin{cases} 0 & \text{ if } j = 1, \\
L_{(-j,-j+1)}(t) & \text{ if } j \geq 2.
\end{cases}
\end{split}
\]
Also, for $j \in \NN$, let 
\[
\hx_j = x_j, \quad \cx_j = -x_{-j}, \quad \hB_j = B_j, \quad \cB_j = -B_{-j}.
\]
For $j \in \NN$ and $t \in \half$, let 
\begin{equation} \label{eq:Vhatch-def}
\hV_j(t) = \begin{cases}
\hB_1(t) + pL_{(0,1)}(t) & \text{ if } j = 1, \\
\hB_j(t) & \text{ if } j \geq 2,
\end{cases}
\quad\quad 
\cV_j(t) = \begin{cases}
\cB_1(t) + qL_{(-1,0)}(t) & \text{ if } j = 1, \\
\cB_j(t) & \text{ if } j \geq 2.
\end{cases}
\end{equation}
From \eqref{eq:asym_eq}, it is easy to check that the processes $(\hX,\hL)$ and $(\cX,\cL)$ satisfy the systems of equations
\begin{equation} \label{eq:upsystem}
\hX_j(t) = \hx_j + \hV_j(t) + p\hL_{(j-1,j)}(t) - q\hL_{(j,j+1)}(t), \quad t \in \half, \quad j \in \NN,
\end{equation}
and 
\begin{equation} \label{eq:downsystem}
\cX_j(t) = \cx_j + \cV_j(t) + q\cL_{(j-1,j)}(t) - p\cL_{(j,j+1)}(t), \quad t \in \half, \quad j \in \NN,
\end{equation}
respectively. Note that the roles of $p$ and $q$ have been reversed in the second system. In the terminology of \cite{AS2}, \eqref{eq:upsystem} and \eqref{eq:downsystem} are infinite systems of singularly interacting particles with driving functions $\{\hx_j+\hV_j(t)\}_{j \in \NN}$ and $\{\cx_j + \cV_j(t)\}_{j \in \NN}$, respectively.

Note that in Theorem \ref{thm:unique}, the assumption \eqref{eq:ainit} on the initial data is equivalent to 
\begin{equation} \label{eq:acinit}
\limsup_{M \to \infty} \frac{\cx_M}{\sqrt{M}} = \infty \text{ a.s.},
\end{equation}
and the assumption on the local times is equivalent to 
\begin{equation} \label{eq:acloc}
\forall T \in \half, \quad \limsup_{M \to \infty} \left(\frac{p}{q}\right)^M \cL_{(M,M+1)}(T) = 0 \text{ a.s.}
\end{equation}

For $0 \leq u < v < \infty$ and $i \in \NN$, let
\begin{equation} \label{eq:hchain}
\begin{split}
\hKs(i,[u,v]) = \sup\{k \geq 0 & : \exists u \leq s_{i + k - 1} \leq \cdots \leq s_{i + 1} \leq s_i \leq v \\
& \text{ such that, for } i \leq j \leq i + k - 1, \hX_{j+1}(s_j) = \hX_j(s_j)\},
\end{split}
\end{equation}
and let 
\begin{equation} \label{eq:cchain}
\begin{split}
\cKs(i,[u,v]) = \sup\{k \geq 0 & : \exists u \leq s_{i + k - 1} \leq \cdots \leq s_{i + 1} \leq s_i \leq v \\
& \text{ such that, for } i \leq j \leq i + k - 1, \cX_{j+1}(s_j) = \cX_j(s_j)\}.
\end{split}
\end{equation}
Clearly, for $i \in \ZZ$, if $i \geq 1$, then $\hKs(i,[u,v]) = \Ksd(i,[u,v])$, and if $i \leq -1$, then $\cKs(-i,[u,v]) = \Ksu(i,[u,v])$. Fix $T \in \ohalf$. To show that for all $i \in \ZZ$, $\Ksd(i,[0,T])$ is finite almost surely, it suffices to show that for all $i \in \NN$, $\hKs(i,[0,T])$ is finite almost surely. This follows from the first equality and noting that if $i \leq 0$, then $\Ksd(i,[0,T]) \leq -i + 1 + \Ksd(1,[0,T]) = -i + 1 + \hKs(1,[0,T])$. Similarly, to show that for all $i \in \ZZ$, $\Ksu(i,[0,T])$ is finite almost surely, it suffices to show that for all $i \in \NN$, $\cKs(i,[0,T])$ is finite almost surely.

Theorem \ref{thm:unique} is now an immediate consequence of the following proposition. Since the proof is an adaptation of \cite[Theorem 2.3(i)]{banbudrud2025wp}, we sketch the argument in the appendix.

\begin{proposition}
\label{prop:fin1210}
Let $X$ be any solution of \eqref{eq:asym_eq} with initial condition satisfying \eqref{eq:ainit} and the associated local times satisfying \eqref{eq:aloc}.
Then, for each $i \in \NN$ and $T \ge 0$,  $\hKs(i,[0,T]) < \infty$
and $\cKs(i,[0,T])<\infty$ a.s.
\end{proposition}

\begin{proof}[Proof of Theorem \ref{thm:unique}] 
From Lemma \ref{lem:finsys} (ii), it suffices to show that for any solution of \eqref{eq:asym_eq} with initial condition satisfying \eqref{eq:ainit} and the associated local times satisfying \eqref{eq:aloc},
for all $i \in \ZZ$ and $T \in (0,\infty)$, $\Ksu(i,[0,T])$ and $\Ksd(i,[0,T])$ are finite almost surely. For this, in view of the discussion above 
Proposition \ref{prop:fin1210}, it suffices to show that for each $i \in \NN$ and $T \ge 0$,  $\hKs(i,[0,T]) < \infty$
and $\cKs(i,[0,T])<\infty$ a.s.
The last statement is a consequence of Proposition \ref{prop:fin1210}, completing the proof of the theorem.
\end{proof}

\subsection{Strong existence}
\label{sec:strexi}
In this section we prove Theorem \ref{thm:exist}.
Choose $\gamma \in \clp_0(\RR^{\ZZ})$ as in the statement of the theorem, and consider the initial data $x = \{x_j\}_{j \in \ZZ} \sim \gamma$ given on some filtered probability space $(\Omega, \clf, \{\clf_t\}_{t \in \half}, \PP)$ on which is given a collection of i.i.d. standard $\{\clf_t\}$-Brownian motions $\{B_i\}_{i \in \ZZ}$. We assume that $x$ is $\clf_0$-measurable and is therefore independent of 
$\{B_i\}_{i \in \ZZ}$. These choices will be fixed throughout the section.

Let $\{(M_n, N_n)\}_{n \in \NN}$ be a sequence of integer pairs such that $M_n \leq N_n$, the sequences $\{-M_n\}$ and $\{N_n\}$ are both non-decreasing, and $\lim_{n \to \infty} M_n = -\infty$ and $\lim_{n \to \infty} N_n = \infty$.
Although the system that we are interested in (i.e. the one described by equation \eqref{eq:asym_eq0}) does not have a drift, for studying stationary distributions in Section \ref{sec:stationary}, it will be useful to consider certain related models with non zero drift. Specifically, we consider the following.
Let $g_-, g_+ \in \mathbb{R}$. For $m,n \in \NN$, consider the finite system of equations
\begin{equation} \label{eq:m_to_n_system}
X_j^{[m,n]}(t) = x_j + B_j(t) + g_j^{[m,n]} t + pL_{(j-1,j)}^{[m,n]}(t) - q L_{(j,j+1)}^{[m,n]}(t), \quad t \in \half, \quad M_m \leq j \leq N_n.
\end{equation}
Here $L_{(M_m-1,M_m)}^{[m,n]} = L_{(N_n,N_n+1)}^{[m,n]} \equiv 0$, and for $M_m \leq j \leq N_n - 1$, 
$X_j^{[m,n]}(t) \le X_{j+1}^{[m,n]}(t)$ for all $j$ and $t$,
$L_{(j,j+1)}^{[m,n]}$ is continuous, nondecreasing, equal to zero at $t = 0$, and can only increase at times $t$ when $X_j^{[m,n]}(t) = X_{j+1}^{[m,n]}(t)$. The drifts $g_j^{[m,n]}$ are given by 
$$
g_{M_m}^{[m,n]} = g_-, \quad g_{N_n}^{[m,n]} = g_+, \quad g_j^{[m,n]} = 0 \text{ for } M_m+1 \leq j \leq N_n - 1.
$$
Thus the lowest and the highest particles are given drifts $g_{-}$ and $g_{+}$, respectively, while the remaining receive a drift of $0$. We caution the reader that in Section \ref{sec:pathuniq}, the notation
$X^{[m,n]}$ and $L^{[m,n]}$ was used for a different collection of processes.

The solution $(X^{[m,n]},L^{[m,n]}) = \{(X_j^{[m,n]},L_{(j,j+1)}^{[m,n]})\}_{M_m \leq j \leq N_n}$ to \eqref{eq:m_to_n_system} exists in the strong sense and is pathwise unique. (See \cite[Section 2.1]{karatzas2016systems}). The next result implies Theorem \ref{thm:exist}.

\begin{theorem} \label{thm:exist+}
 For each $n_0 \in \NN$, and $M_{n_0} \leq j \leq N_{n_0}$, the following limits exist almost surely, uniformly on compact subsets of $\half$:
\begin{equation} \label{eq:balanced_lim}
X_j = \lim_{\substack{n \to \infty, \\ n \geq n_0}}  X_j^{[n,n]}, \quad\quad L_{(j,j+1)} = \lim_{\substack{n \to \infty, \\ n \geq n_0}} L_{(j,j+1)}^{[n,n]}, 
\end{equation}
\begin{equation} \label{eq:m_lim}
X_j^- = \lim_{\substack{n \to \infty, \\ n \geq n_0}} \lim_{\substack{m \to \infty, \\ n,m\geq n_0}} X_j^{[m,n]}, \quad\quad L_{(j,j+1)}^- = \lim_{\substack{n \to \infty, \\ n \geq n_0}} \lim_{\substack{m \to \infty, \\ n,m\geq n_0}} L_{(j,j+1)}^{[m,n]},
\end{equation}
\begin{equation} \label{eq:p_lim}
X_j^+ = \lim_{\substack{m \to \infty, \\ m \geq n_0}} \lim_{\substack{n \to \infty, \\ n,m\geq n_0}} X_j^{[m,n]}, \quad\quad L_{(j,j+1)}^+ = \lim_{\substack{m \to \infty, \\ m \geq n_0}} \lim_{\substack{n \to \infty, \\ n,m\geq n_0}} L_{(j,j+1)}^{[m,n]}.
\end{equation}
Moreover, the processes $\{(X_j,L_{(j,j+1)})\}_{j \in \ZZ}$, $\{(X^{\pm}_j,L_{(j,j+1)}^\pm)\}_{j \in \ZZ}$ solve \eqref{eq:asym_eq} (with initial data $\{x_j\}$ and Brownian motions $\{B_j\}$). The collections of local times $\{L_{(j,j+1)}\}$, $\{L_{(j,j+1)}^\pm\}$  satisfy the assumption \eqref{eq:aloc}. Furthermore, the three solutions coincide: $X_j = X_j^- = X_j^+$, and $L_{(j,j+1)} = L_{(j,j+1)}^- = L_{(j,j+1)}^+$ for all $j \in \ZZ$.
\end{theorem}

Let $(X^0,L^0) = \{(X_j^0, L_{(j,j+1)}^0)\}_{j \in \ZZ}$ be the unique solution to \eqref{eq:asym_eq} with $p = 0, q = 1$ and with initial data $x = \{x_j\}$ and Brownian motions $\{B_j\}$, as above. The existence of such a solution was originally proved in \cite[Section 2.3.1]{WFSbook}. The uniqueness of this solution, under the assumption \eqref{eq:ainit} on the initial data (which is implied by \eqref{eq:ainit2}), follows from \cite[Theorem 2.3(i)]{banbudrud2025wp}. Note that although \cite{banbudrud2025wp} only considers a one-sided infinite system of singularly interacting particles, we still get pathwise uniqueness for the two-sided system in this case, owing to the totally asymmetric nature of the interactions when $p = 0$. Indeed, for any $K \in \ZZ$, the subsystem $\{(X_j^0,L^0_{(j,j+1)})\}_{j = K}^\infty$ is a one-sided system of singularly interacting Brownian particles of the sort studied in \cite{banbudrud2025wp}, with initial data satisfying \eqref{eq:ainit}. Hence, these trajectories are uniquely determined for any $K$.

For each $j \in \ZZ$, $X_j^0$ is the Skorokhod reflection from below of the trajectory $x_j + B_j(\cdot)$ from $X_{j+1}(\cdot)$. Hence, the following recursive relationship holds:
\begin{equation} \label{eq:skor1}
X_j^0(t) = x_j + B_j(t) + \inf_{0 \leq s \leq t} \left( X_{j+1}^0(s) - x_j - B_j(s) \right) \wedge 0 = x_j + B_j(t) -L_{(j,j+1)}^0(t).
\end{equation}

To prove Theorem \ref{thm:exist+}, we will need the following lemma. Recall the definition of $\clB_M^-(i,[0,t])= \clB_M^-(i,t)$ from 
\eqref{eq:mLPP}.

\begin{lemma} \label{lem:fincomp}
Let $g_* := \min\{g_-,g_+,0\}$. For any $m,n,j,\ell \in \NN$ with $M_m \leq j \leq \ell \leq N_n$,
\begin{equation} \label{eq:finbds}
X_j^{[m,n]}(t) \geq X_j^0(t) + g_* t \geq x_j + \clB^-_{\ell - j + 1}(j,t) - L_{(\ell,\ell+1)}^0(t) + g_*t.
\end{equation}
\end{lemma}

\begin{proof}
For $M_m \leq j \leq N_n$, let 
$$
U_j^*(t) = \begin{cases}
B_{M_m}(t) + g_- t & \text{ if } j = M_m, \\
B_j(t) & \text{ if } M_m+1 \leq j \leq N_n - 1, \\
B_{N_n}(t) + g_+t - L_{(N_n,N_n+1)}^0(t) & \text{ if } j = N_n.
\end{cases}
$$
Let $X^* = \{X_j^*\}_{M_m \leq j \leq N_n}$ be the finite system of singularly interacting particles satisfying the system of equations 
$$
X_j^*(t) = x_j + U_j^*(t) - L_{(j,j+1)}^*(t), \quad t \in \half, \quad M_m \leq j \leq N_n,
$$
where $L^*_{(N_n,N_n+1)} \equiv 0$, for $M_m \leq j \leq N_n - 1$, $X_j^*(t) \le X_{j+1}^*(t)$, $L^*_{(j,j+1)}$ is continuous, nondecreasing, equal to zero at $t = 0$, and can only increase at times $t$ when $X_j^*(t) = X_{j+1}^*(t)$. 

We will prove that, for $M_m \leq j \leq N_n$ and $t \in \half$, 
\begin{equation} \label{eq:intermediate}
X_j^{[m,n]}(t) \geq X_j^*(t) \geq X_j^0(t) + g_*t.
\end{equation}
The first inequality in \eqref{eq:finbds} follows from \eqref{eq:intermediate}. To prove the first inequality in \eqref{eq:intermediate}, we proceed by downward induction on $j$, starting with $j=N_n$. The base case $j = N_n$ holds since
\[
\begin{split}
X_{N_n}^{[m,n]}(t) & = x_{N_n} + B_{N_n}(t) + g_+t + pL_{({N_n}-1,{N_n})}^{[m,n]}(t) \\
& \geq x_{N_n} + B_{N_n}(t) + g_+t - L^0_{({N_n},{N_n}+1)}(t) = X_{N_n}^*(t).
\end{split}
\]
Assume the inequality has been proved for all $j$ such that $i+1 \leq j \leq N_n$, where $M_m \leq i \leq N_n - 1$. Notice that for all $t \in \half$, we can write 
\begin{equation} \label{eq:skorform}
X_i^{[m,n]}(t) = x_i + R_i(t) - L'_{(i,i+1)}(t),
\end{equation}
where $R_i(t) = B_i(t) + g_i^{[m,n]}t + pL_{(i-1,i)}^{[m,n]}(t)$ and $L'_{(i,i+1)}(t) = qL_{(i,i+1)}^{[m,n]}(t)$. By the properties of $L_{(i,i+1)}^{[m,n]}$, we see that $X_i^{[m,n]}$ is the Skorokhod reflection from below of the trajectory $x_i + R_i(\cdot)$ from $X_{i+1}^{[m,n]}(\cdot)$. Thus
\[
\begin{split}
L'_{(i,i+1)}(t) & = -\inf_{0 \leq s \leq t} \left( X_{i+1}^{[m,n]}(s) - x_i - R_i(s) \right) \wedge 0 \\
& \leq -\inf_{0 \leq s \leq t} \left( X_{i+1}^*(s) - x_i - R_i(s) \right) \wedge 0.
\end{split}
\]
Here the second line follows by the induction assumption. Applying this bound to \eqref{eq:skorform}, we obtain 
\[
\begin{split}
X_i^{[m,n]}(t) & \geq x_i + R_i(t) + \inf_{0 \leq s \leq t} \left( X_{i+1}^*(s) - x_i - R_i(s) \right) \wedge 0 \\
& = \inf_{0 \leq s \leq t} \left( X_{i+1}^*(s) + R_i(t) - R_i(s) \right) \wedge \left(x_i + R_i(t) \right). \\
& \geq \inf_{0 \leq s \leq t} \left( X_{i+1}^*(s) + U^*_i(t) - U^*_i(s) \right) \wedge \left(x_i + U^*_i(t) \right) = X_i^*(t).
\end{split}
\]
In the last line, we use the fact that, for any $0 \leq s \leq t < \infty$, $R_i(t) \geq U^*_i(t)$ and $R_i(t) - R_i(s) \geq U^*_i(t) - U^*_i(s)$ (using properties of the local time $L_{(i-1,i)}^{[m,n]}$) and  that $X_i^*$ is the Skorohod reflection from below of the trajectory $x_i +U_i^*$ from $X_{i+1}^*$.  This completes the induction step, and the first inequality in \eqref{eq:intermediate} is proved.

To prove the second inequality in \eqref{eq:intermediate}, we also proceed by downward induction on $j$. The base case holds, since by definition 
$$
X_{N_n}^*(t) = x_{N_n} + B_{N_n}(t) + g_+t - L_{(N_n,N_n+1)}^0(t) = X_{N_n}^0(t) + g_+t \geq X_{N_n}^0(t) + g_*t.
$$
Suppose the second inequality in \eqref{eq:intermediate} has been established for all $j$ such that $i+1 \leq j \leq N_n$, where $M_m+1 \leq i \leq N_n-1$. By \eqref{eq:skor1} and the induction assumption, 
\[
\begin{split}
X_i^*(t) & = x_i + B_i(t) + \inf_{0 \leq s \leq t} \left( X_{i+1}^*(s) - x_i - B_i(s) \right) \wedge 0 \\
& \geq  x_i + B_i(t) + \inf_{0 \leq s \leq t} \left( X_{i+1}^0(s) + g_*s - x_i - B_i(s) \right) \wedge 0 \\
& \geq x_i + B_i(t) + \inf_{0 \leq s \leq t} \left( X_{i+1}^0(s) - x_i - B_i(s) \right) \wedge 0 + g_* t \\
& = X_i^0(t) + g_*t.
\end{split}
\]
Here the third line uses the fact that $g_* \leq 0$. This completes the induction step for $i \geq M_m+1$. If $i = M_m$, we similarly have 
\[
\begin{split}
X_{M_m}^*(t) & = x_{M_m} + B_{M_m}(t) + g_-t + \inf_{0 \leq s \leq t} \left( X_{{M_m}+1}^*(s) - x_i - B_{M_m}(s) - g_- s \right) \wedge 0 \\
& \geq x_{M_m} + B_{M_m}(t) + g_-t + \inf_{0 \leq s \leq t} \left( X_{{M_m}+1}^0(s) + g_*s - x_i - B_{M_m}(s) - g_- s \right) \wedge 0.
\end{split}
\]
Noting that $g_* - g_- \leq 0$, the quantity above is bounded below by 
\[
\begin{split}
& x_{M_m} + B_{M_m}(t) + \inf_{0 \leq s \leq t} \left\{ g_*s + g_-(t - s) + \left( X_{{M_m}+1}^0(s) - x_{M_m} - B_{M_m}(s) \right) \wedge 0 \right\} \\
& \geq x_{M_m} + B_{M_m}(t) + g_*t + \inf_{0 \leq s \leq t} \left( X_{{M_m}+1}^0(s) - x_{M_m} - B_{M_m}(s) \right) \wedge 0 \\
& = X_{M_m}^0(t) + g_*t.
\end{split}
\]
This completes  the proof of \eqref{eq:intermediate} and thus proves the first inequality in \eqref{eq:finbds}.

To prove the second equality in \eqref{eq:finbds}, fix $\ell$ such that $M_m \leq \ell \leq N_n$. It suffices to show that for $M_m \leq j \leq \ell$,
\begin{equation}
X_j^0(t) \geq x_j + \clB_{\ell - j + 1}^-(j,t) - L_{(\ell,\ell+1)}^0(t).
\end{equation}
We proceed by downward induction on $j$. The base case $j = \ell$ is clear since, by definition, $\clB_1^-(\ell,t) = B_\ell(t)$, and hence
$$
X_\ell^0(t) = x_\ell + B_\ell(t) - L_{(\ell,\ell+1)}^0(t) = x_\ell + \clB_1^-(\ell,t) - L_{(\ell,\ell+1)}^0(t).
$$
Suppose the result has been proved for $i+1 \leq j \leq \ell$ where $M_m \leq i \leq \ell - 1$. Then 
\begin{equation} \label{eq:x0lb}
\begin{split}
X_i^0(t) & = x_i + B_i(t) - L^0_{(i,i+1)}(t) \\
& = x_i + B_i(t) + \inf_{0 \leq s \leq t} \left(X_{i+1}^0(s) - x_i - B_i(s) \right) \wedge 0 \\
& = \inf_{0 \leq s \leq t} \left( X_{i+1}^0(s) + B_i(t) - B_i(s) \right) \wedge (x_i + B_i(t)) \\
& \geq \inf_{0 \leq s \leq t} \left( x_{i} + \clB^{-}_{\ell - i}(i+1,s) - L^0_{(\ell,\ell+1)}(s) + B_i(t) - B_i(s) \right) \wedge (x_i + B_i(t)).
\end{split}
\end{equation}
Here the last line follows from the induction assumption and the fact that $x_i \leq x_{i+1}$. Observe that since $\clB_{\ell - i}(i+1,s)$ and $- L^0_{(\ell,\ell+1)}(s)$ are non-positive, and $B_i(0) = 0$, 
$$
\inf_{0 \leq s \leq t} \left( x_{i} + \clB^{-}_{\ell - i}(i+1,s) - L^0_{(\ell,\ell+1)}(s) + B_i(t) - B_i(s) \right) \leq (x_i + B_i(t)).
$$
Thus, the last line of \eqref{eq:x0lb} is equal to 
\[
\begin{split}
 & \inf_{0 \leq s \leq t} \left( x_{i} + \clB^{-}_{\ell - i}(i+1,s)  - L^0_{(\ell,\ell+1)}(s) + B_i(t) - B_i(s) \right) \\
 & \geq x_i + \inf_{0 \leq s \leq t} \left(\clB_{\ell - i}^-(i+1,s) + B_i(t) - B_i(s) \right)  - L^0_{(\ell,\ell+1)}(t) \\
 & = x_i + \clB_{\ell - i + 1}^-(i,t) - L^0_{(\ell,\ell+1)}(t),
\end{split}
\]
where the second line follows from the fact that $L^0   _{(\ell,\ell+1)}$ is nondecreasing, and the last line follows from the definition \eqref{eq:mLPP} of $\clB_{\ell - i + 1}^-(i,[0,t])$. This completes the induction step and the proof of the second inequality in \eqref{eq:finbds}.
\end{proof}

Next we will prove the following fluctuation estimate. Let $r:= p/q$ and note that $0< r<1$.

\begin{lemma}

Let $g^* = \max\{g_-,g_+,0\}$. Fix $m,n,j,\ell\in \NN$ with $M_m \leq j \leq \ell \leq N_n-1$. There exists $m_0 = m_0(\omega) \in \NN$ such that if $m \geq m_0$ and $0 \leq s \leq t < \infty$,
\begin{equation} \label{eq:loc_contbd}
\begin{split}
L_{(j,j+1)}^{[m,n]}(t) - L_{(j,j+1)}^{[m,n]}(s) & \leq \frac{q^{-1}(g^* - g_*)(t - s)}{1-r} + \frac{q^{-1}(L_{(\ell,\ell+1)}^0(t) - L_{(\ell,\ell+1)}^0(s))}{1 - r} \\
& \quad + q^{-1}\sum_{k = -\infty}^j r^{j - k}(B_k(t) - B_k(s) - \clB^-_{\ell - k + 1}(k,[s,t])).
\end{split}
\end{equation}
\end{lemma}

\begin{proof} 
We will only consider the case when $s=0$. The general case follows on considering the evolution of $X(s+\cdot)-X(s)$(see \cite[Lemma 3.1]{banbudrud2025wp}).
Using \eqref{eq:m_to_n_system}, we compute
\[
\begin{split}
\sum_{k = M_m}^j r^{j - k} X_k^{[m,n]}(t) & = \sum_{k = M_m}^j r^{j - k}(x_k + g_k^{[n,m]}t + B_k(t)) \\
& \hspace{0.5in} + \sum_{k = M_m}^j r^{j - k}\left(pL_{(k-1,k)}^{[m,n]}(t) - qL_{(k,k+1)}^{[m,n]}(t)\right) \\
& = \sum_{k = M_m}^j r^{j - k}(x_k + g_k^{[n,m]}t + B_k(t)) \\
& \hspace{0.5in} + \sum_{k = M_m}^j q\left(r^{j - k + 1}L_{(k-1,k)}^{[m,n]}(t) - r^{j - k}L_{(k,k+1)}^{[m,n]}(t)\right) \\
& = \sum_{k = M_m}^j r^{j - k}(x_k + g_k^{[n,m]}t + B_k(t)) - qL_{(j,j+1)}^{[m,n]}(t).
\end{split}
\]
Hence,  
\begin{equation}
L_{(j,j+1)}^{[m,n]}(t) = q^{-1} \sum_{k = M_m}^j r^{j - k}\left(g_k^{[n,m]}t + B_k(t) + x_k - X_k^{[m,n]}(t)\right).
\end{equation}
By Lemma \ref{lem:fincomp}, the right-hand side is bounded above by 
\begin{equation} \label{eq:locub1}
\begin{split}
& q^{-1} \sum_{k = M_m}^j r^{j - k}\left((g_k^{[n,m]} - g_*)t + B_k(t) - \clB^-_{\ell-k+1}(k,t) + L_{(\ell,\ell+1)}^0(t)\right) \\
& \leq  q^{-1} \sum_{k = M_m}^j r^{j - k}\left((g^*-g_*)t + B_k(t) - \clB^-_{\ell-k+1}(k,t) + L_{(\ell,\ell+1)}^0(t)\right),
\end{split}
\end{equation}
noting that $g_k^{[n,m]}\le g^*$. 
Note that the summand on the right side above is  nonnegative with probability 1.
Hence,  the quantity \eqref{eq:locub1} is bounded above by 
$$
\frac{q^{-1}(g^* - g_*)t}{1-r} + q^{-1} \sum_{k = -\infty}^j r^{j - k}\left(B_k(t) - \clB^-_{\ell-k+1}(k,t)\right) + \frac{q^{-1}L_{(\ell,\ell+1)}^0(t)}{1 - r},
$$
which is what we needed to show.
\end{proof}

\begin{proof}[Proof of Theorem \ref{thm:exist+}]
We note that the final statement is immediate from the previous statements in the theorem and Theorem \ref{thm:unique} on noting that assumption \eqref{eq:ainit2} on the initial condition implies that \eqref{eq:ainit} is satisfied.
We will only prove  the statements on the limits  in \eqref{eq:balanced_lim}. Corresponding statements for the quantities 
in \eqref{eq:m_lim} and \eqref{eq:p_lim} are proved similarly.

We begin by proving that the limits \eqref{eq:balanced_lim} exist and satisfy \eqref{eq:asym_eq}. Fix $T<\infty$. For $d \in \NN$, $f \in C([0,T],\RR^{d})$, $\delta > 0$, let 
\[
\osc_\delta(f) := \sup\{ \|f(t) - f(s)\|_1 : 0 \leq s \leq t \leq T, |t - s| \leq \delta\}. 
\]
Here, for $x = \{x_k\} \in \RR^d$, $\|x\|_1 = \sum_{k = 1}^d |x_k|$. From \eqref{eq:loc_contbd} we obtain, for $n_0, n, j \in \NN$, $n_0 \le n$, with $M_n \leq j \leq N_{n_0}-1 \leq N_n-1$, and $0 \leq s \leq t \leq T$, 
\begin{equation} \label{eq:loc_contbd2}
\begin{split}
& L_{(j,j+1)}^{[n,n]}(t) - L_{(j,j+1)}^{[n,n]}(s) \\
& \quad \leq \frac{q^{-1}(g^* - g_*)(t - s)}{1-r} + \frac{q^{-1}(L_{(N_{n_0}-1,N_{n_0})}^0(t) - L_{(N_{n_0}-1,N_{n_0})}^0(s))}{1 - r} \\
& \quad\quad + q^{-1}\sum_{k = -\infty}^j r^{j - k}(B_k(t) - B_k(s) - \clB^-_{N_{n_0} - k}(k,[s,t])).
\end{split}
\end{equation}
We claim that, almost surely, the right-hand side of \eqref{eq:loc_contbd2} is continuous as a function of $(s,t)$ in the compact region of the plane $0 \leq s \leq t \leq T$ (where we define $\clB^-_{N_{n_0} - k + 1}(k,[s,t])) = 0$ when $s = t$, by continuity). To prove this claim, observe that, by \eqref{eq:mLPP}, for all $0 \leq s \leq t \leq T$,
\[
\begin{split}
& B_k(t) - B_k(s) - \clB^-_{N_{n_0} - k + 1}(k,[s,t])) \\ 
& \leq \sup_{0 \leq s \leq t \leq T} (B_k(t) - B_k(s)) - \sum_{i = k}^{N_{n_0}} \inf_{0 \leq s \leq t \leq T} (B_i(t) - B_i(s)).
\end{split}
\]
An argument using standard Gaussian tail bounds to control the right-hand side above shows that, almost surely, the infinite sum in \eqref{eq:loc_contbd2} converges uniformly on $0 \leq s \leq t \leq T$, and the claim follows. Furthermore, noting that the right-hand side of \eqref{eq:loc_contbd2} is equal to zero when $s = t$, we conclude that, with
$[L^{[n,n]}]_{n_0} = (L^{[n,n]}_{M_{n_0}, M_{n_0}+1}, \dots ,
L^{[n,n]}_{N_{n_0}-1, N_{n_0}})$,
\begin{equation} \label{eq:Lequicontinuity}
\lim_{\delta \to 0} \sup_{n \geq n_0} \osc_{\delta}([L^{[n,n]}]_{n_0}) \to 0 \text{ a.s.}
\end{equation}
Also, the bound \eqref{eq:loc_contbd2}, with $s = 0$, also implies that, for all $t \in [0,T]$,
\begin{equation} \label{eq:Lptbd}
\sup_{n \geq n_0} \| [L^{[n,n]}(t)]_{n_0} \|_1 < \infty \text{ a.s.}
\end{equation}
In view of \eqref{eq:m_to_n_system}, it follows from \eqref{eq:Lequicontinuity} and \eqref{eq:Lptbd} that, with $[X^{[n,n]}]_{n_0}= (X^{[n,n]}_{M_{n_0}}, \cdots, X^{[n,n]}_{N_{n_0}})$, 
\begin{equation} \label{eq:Xequicontinuity}
\lim_{\delta \to 0} \sup_{n \geq n_0} \osc_{\delta}([X^{[n,n]}]_{n_0}) \to 0 \text{ a.s.},
\end{equation}
and, for all $t \in [0,T]$,
\begin{equation} \label{eq:Xptbd}
\sup_{n \geq n_0} \| [X^{[n,n]}(t)]_{n_0} \|_1 < \infty \text{ a.s.}
\end{equation}
In view of \eqref{eq:Lequicontinuity}, \eqref{eq:Lptbd}, \eqref{eq:Xequicontinuity}, and \eqref{eq:Xptbd}, by the Arzel\`{a}-Ascoli Theorem, for a.e. $\omega$, there is a subsequence along which the following limits hold,  in $C([0,T], \RR^{2n_0+1})$, for any $T \in \ohalf$:
\begin{equation} \label{eq:subseq_lims}
\begin{split}
&\{L^{[n,n]}_{(j,j+1)}\}_{j = M_{n_0}}^{N_{n_0}} \to \{L_{(j,j+1)}\}_{j = M_{n_0}}^{N_{n_0}}, \\
&\{X^{[n,n]}_j\}_{j = M_{n_0}}^{N_{n_0}} \to \{X_j\}_{j = M_{n_0}}^{N_{n_0}} \quad \text{ as } \quad n \to \infty.
\end{split}
\end{equation}
By a diagonal argument, we may suppose that a.s., there is a subsequence along which convergence holds for all $n_0$ and $T$, and consequently the limits $X = \{X_j\}_{j \in \NN}$ and $L = \{L_{(j,j+1)}\}_{j \in \NN}$ do not depend on $n_0$.

By uniqueness of solutions to \eqref{eq:asym_eq} (see Theorem \ref{thm:unique}), to show that \eqref{eq:balanced_lim} holds, it suffices to show that the subsequential limit $(X,L)$ solves \eqref{eq:asym_eq}, and the local times $L$ satisfy \eqref{eq:aloc}. (Note that the initial data $x = \{x_j\}$ satisfy \eqref{eq:ainit} by assumption.) The fact that \eqref{eq:asym_eq} is satisfied is immediate from \eqref{eq:m_to_n_system} and \eqref{eq:subseq_lims}. To see that \eqref{eq:aloc} holds, note that by \eqref{eq:loc_contbd2}, with $s = 0$ and $t = T$, and \eqref{eq:subseq_lims}, for $n_0, j \in \NN$ with $j \leq N_{n_0} - 1$,
\begin{equation} \label{eq:limlocbd}
\begin{split}
r^{-j}L_{(j,j+1)}(T) & \leq \frac{r^{-j}q^{-1}(g^* - g_*)T}{1 - r} + \frac{r^{-j}q^{-1} L_{(N_{n_0}-1,N_{n_0})}^0(T)}{1 - r} \\ 
& \quad\quad + q^{-1}\sum_{k = -\infty}^j r^{-k}(B_k(T) - \clB^-_{N_{n_0} - k}(k,[0,T])).
\end{split}
\end{equation}
Since the sum on the right-hand side of \eqref{eq:limlocbd} converges, the right-hand side goes to zero as $j \to -\infty$, and this proves \eqref{eq:aloc}.
\end{proof}

\section{Stationary measures for the two-sided gap process} \label{sec:stationary}

In this section we prove Theorem \ref{thm:stationary}.

To prove the theorem, we first recall some known results on the stationary distributions of finite systems of singularly interacting Brownian particles. Fix $N \in \NN$ and a (possibly random) nondecreasing sequence $x^N=\{x_j^N\}_{-N \leq j \leq N} \in \RR^{2N+1}$. Let $\{g_j\}_{-N \leq j \leq N} \in \RR^{2N+1}$, and consider a finite system of singularly interacting Brownian particles $\{X_j^N\}_{-N \leq j \leq N}$, solving the system of equations
\begin{equation} \label{eq:approx_system}
X_j^N(t) = x_j^N + g_j t + B_j(t) + pL_{(j-1,j)}^N(t) - qL_{(j,j+1)}^N(t), \quad t \in \half, \quad -N \leq j \leq N,
\end{equation}
where, as before, $\{B_j\}$ are mutually independent Brownian motions, independent of  $x^N$, $X_j^N(t)\le X_{j+1}^N(t)$,
$L_{(-N-1,-N)}^N = L_{(N,N+1)}^N \equiv 0$, and the collection $\{L^N_{(j,j+1)}\}_{-N+1 \leq j \leq N-1}$ satisfies the usual local time conditions, namely $L^N_{(j,j+1)}$ is continuous, nondecreasing, equal to zero at $t = 0$, and can only increase at times $t$ when $X_j^N(t) = X_{j+1}^{N}(t)$.   For $-N \leq j \leq N-1$, let $\mu_j = g_{j+1} - g_j$. Define the gap process $\{Z_j^N\}_{-N \leq j \leq N-1}$ by 
$$
Z_j^N(t) = X_{j+1}^N(t) - X_j^N(t), \quad -N \leq j \leq N-1, \quad t \in \half.
$$
Suppose that $\bar{\lambda} := (\lambda_1,\dots,\lambda_{N-1}) \in \ohalf^{N-1}$ solves the system of linear equations:
\begin{equation} \label{eq:linsystem}
\begin{split}
& \lambda_{-N} - q \lambda_{-N+1} = -\mu_{-N}, \\
& -p \lambda_{k-2} + \lambda_{k-1} - q \lambda_k = -\mu_{k-1}, \text{ for } -N+2 \leq k \leq N-1, \\
& -p \lambda_{N-2} + \lambda_{N - 1} = -\mu_{N-1}.
\end{split}
\end{equation}
From classical results of \cite{harrison1987multidimensional} (cf. \cite[Proposition 2.1]{sarantsev2017infinite}),
$$
\pi := \bigotimes_{k = -N}^{N-1} \text{Exp}(\lambda_k)
$$
is the unique stationary distribution for the gap process $\{Z_j^N\}_{1 \leq j \leq N-1}$.

\begin{proof}[Proof of Theorem \ref{thm:stationary}]

Let $x= \{x_j\}_{j \in \ZZ}\sim \gamma_{\lambda}$. For $N \in \NN$, let $X^N = \{X_j^N\}_{-N \leq j \leq N}$ be the process defined by \eqref{eq:approx_system}, where we take the same driving Brownian motions as used to define $X^{\lambda}$,  $x^N = \{x_j\}_{-N \leq j \leq N}$ and 
\[
g_j = \begin{cases} 
p\lambda & \text{ if } j = -N, \\
0 & \text{ if } -N+1 \leq j \leq N-1, \\
-q\lambda & \text{ if } j = N.
\end{cases}
\]
Then $\lambda_j = \lambda$ for $-N \leq j \leq N$ uniquely solves the system of equations \eqref{eq:linsystem}. Hence, using the results on stationary distributions for finite systems from \cite{harrison1987multidimensional, sarantsev2017infinite} noted above, if $Z^N = \{Z_j^N = X_{j+1}^N - X_j^N\}_{-N \leq j \leq N-1}$ is the corresponding gap process, then $Z^N(t)$ is an i.i.d. sequence of Exp($\lambda$) random variables for all $t \in \half$. Moreover, by Theorem \ref{thm:exist+}, for any $j,N_0 \in \NN$ with $-N_0 \leq j \leq N_0$, and $t \in \half$, the following limit holds almost surely, uniformly on compact time intervals:
\[
\lim_{N \to \infty, N \geq N_0} X_j^N(t) = X^{(\lambda)}_j(t).
\]
Therefore, $\{Z_j^{(\lambda)}(t)\}_{-N_0 \leq j \leq N_0-1}$ is an i.i.d. sequence of Exp($\lambda$) random variables, and since $N_0$ is arbitrary, $Z^{(\lambda)}(t) \sim \pi_\lambda$, for every $t\ge 0$.
\end{proof}

\section{Preliminary estimates on particle trajectories and local times}

In this section, we first prove some general bounds on the particle trajectories and local times, for the system described in Definition \ref{def:biinf}, that hold for an arbitrary $p \in [0,1/2)$. 

We once again assume throughout that  $\gamma \in \clp_0(\RR^{\ZZ})$ satisfies \eqref{eq:ainit2} with some $\chi>1/2$ (and therefore also satisfies \eqref{eq:ainit}). Let $\{(X_j, L_{(j,j+1)})\}_{j \in \ZZ}$ be the unique strong solution to \eqref{eq:asym_eq}, with initial data $\{x_j\} \sim \gamma$ and Brownian motions $\{B_j\}$, such that the local times satisfy \eqref{eq:aloc}. By Lemma \ref{lem:fincomp} and \eqref{eq:balanced_lim}, with $g_+ = g_- = 0$, for all $j \in \ZZ$, $t \in \half$, 
\begin{equation} \label{eq:plb-pointwise}
X_j(t) \geq X_j^0(t),
\end{equation}
where $\{X_j^0\}_{j \in \ZZ}$ is the solution to \eqref{eq:asym_eq} with $p = 1 - q = 0$, defined as in the discussion following Theorem \ref{thm:exist+}. By uniqueness of solutions to \eqref{eq:asym_eq}, $\{X_j^0\}_{j \in \ZZ}$ is given by the explicit formula \cite[(2.3.2)]{WFSbook} (appropriately modified to account for the fact that in \cite{WFSbook}, the roles of $p$ and $q$ have been reversed), namely, for all $j \in \ZZ$ and $t \in \half$,
\begin{equation} \label{eq:WFSsolution}
X_j^0(t) = \min_{k \geq j} \left\{ x_k + \clB_{k - j + 1}^-(j,t) \right\},
\end{equation}
where $\clB_{k - j + 1}^-(j,t) = \clB_{k - j + 1}^-(j,[0,t])$ is defined as in \eqref{eq:mLPP}.

Assumption \eqref{eq:ainit2} (which implies Assumption \eqref{eq:ainit}) will be made throughout and therefore will not be explicitly noted in the statement of results.
To obtain good control on the trajectories and local times, the following additional assumption on the initial distribution $\gamma$ will be needed. 

\begin{assumption} \label{assump:initdiff}
For all $c,\xi \in \ohalf$, 
\begin{equation} \label{eq:sqrtcomp}
\sup_{j \in \ZZ} \sum_{k \geq j} e^{c\sqrt{k - j}} \EE_\gamma[e^{-\xi(x_k - x_j)}] < \infty.
\end{equation}
Moreover, there exists $c_0 = c_0(\xi) \in (0,\infty)$ such that $c_0(\xi) \to \infty$ as $\xi \to \infty$, and the following bound holds for any $\xi \in (0,\infty)$:
\begin{equation} \label{eq:posbd}
\sup_{j \in \NN} e^{c_0(\xi) j}\EE_{\gamma}[e^{-\xi x_j}] < \infty. 
\end{equation}
\end{assumption}

We note that the above assumption is satisfied if  $\gamma = \gamma_\lambda$ for some $\lambda > 0$ (see \eqref{eq:hominit}). Indeed, in that case,  for all $j,k \in \ZZ$ with $j \leq k$,
\[
\EE_\gamma[e^{-\xi x_j}] = \left( \frac{\lambda}{\lambda + \xi} \right)^j, \quad\quad \EE_\gamma[e^{-\xi (x_k - x_j)}] = \left( \frac{\lambda}{\lambda + \xi} \right)^{k - j},
\]
{and it is easy to check that the assumption is satisfied with $c_0(\xi) = \log(1 + \xi/\lambda)$.}

\begin{lemma} \label{lem:plbXj}
Under Assumption \ref{assump:initdiff}, for all $T \in \half$ and $\xi \in \ohalf$,
\begin{equation} \label{eq:changeprobbd}
\sup_{j \in \ZZ, u \in \RR} e^{u\xi}\PP( \inf_{0 \leq s \leq T} X_j(s) - x_j \leq -u ) < \infty,
\end{equation}
and 
\begin{equation} \label{eq:posprobbd}
\sup_{j \in \ZZ, u \in \RR} e^{u\xi + c_0(\xi) j} \PP( \inf_{0 \leq s \leq T} X_j(s) \leq -u ) < \infty,
\end{equation}
where $c_0(\xi)$ is the quantity appearing in \eqref{eq:posbd}.
\end{lemma}

\begin{proof}
By \eqref{eq:plb-pointwise} and \eqref{eq:WFSsolution}, 
\begin{equation} \label{eq:Xjlb1}
\begin{split}
\inf_{0 \leq s \leq T} X_j(s) & \geq \inf_{0 \leq s \leq T} X_j^0(s) \\
& = \min_{k \geq j} \left\{ x_k + \inf_{0 \leq s \leq T}\clB_{k - j + 1}^-(j,s) \right\} \\
& \geq \min_{k \geq j} \left\{ x_k + \clB_{k - j + 2}^-(j-1,T) + B_{j-1}^*(T) \right\},
\end{split}
\end{equation}
where 
\[
B_{j-1}^*(T) = -\sup_{0 \leq s \leq T} (B_{j-1}(T) - B_{j-1}(s)).
\]
Here the last inequality is obtained by observing that 
$$\clB_{k - j + 2}^-(j-1,T) = \inf_{s \in [0,T]} \left(\clB_{k - j + 1}^-(j,s) +  B_{j-1}(T) - B_{j-1}(s)\right).$$
By \eqref{eq:Xjlb1}, for all $u \in \RR$ and $\xi \in \half$,
\begin{align} \label{eq:Xjproblb1}
& \PP(\inf_{0 \leq s \leq T} X_j(s) - x_j \leq -u) \leq \PP\left( \min_{k \geq j} \left\{ x_k - x_j + \clB_{k - j + 2}^-(j-1,T) + B_{j-1}^*(T) \right\} \leq -u \right) \nonumber\\
& \leq \PP\left( \max_{k \geq j} \exp({-\xi\left\{ x_k - x_j + \clB_{k - j + 2}^-(j-1,T) + B_{j-1}^*(T) \right\}}) \geq e^{\xi u} \right) \nonumber\\
& \leq e^{-\xi u}\EE\left[ \max_{k \geq j} \exp({-\xi\left\{ x_k - x_j + \clB_{k - j + 2}^-(j-1,T) + B_{j-1}^*(T) \right\}}) \right] \nonumber\\
& \leq e^{-\xi u} \sum_{k \geq j} \EE\exp({-\xi\left\{ x_k - x_j + \clB_{k - j + 2}^-(j-1,T) + B_{j-1}^*(T) \right\}}) \nonumber\\
& = e^{-\xi u} \sum_{k \geq j} \EE[e^{-\xi(x_k - x_j)}] \EE[e^{-2\xi \clB_{k - j + 2}^-(j-1,T)}]^{1/2} \EE[ e^{-2\xi B_{j-1}^*(T)} ]^{1/2}.
\end{align}
Here the last line follows by independence of the initial data $\{x_j\}$ and the Brownian motions $\{B_j\}$, and the Cauchy-Schwarz inequality. By standard Gaussian estimates, 
\begin{equation} \label{eq:Bstarest}
\EE[e^{-2\xi B_{j-1}^*(T)}] \leq 2e^{2\xi^2T}.
\end{equation}
Furthermore, recall that $-\clB_{k - j + 2}^-(j-1,T)$ is distributed as the maximum eigenvalue of a $(k - j + 2) \times (k- j + 2)$ complex Hermitian matrix with $N(0,T)$ entries (see \cite{GTWLimitThms2001} and \cite{bary2001}). Using the Lipschitz property of the maximum eigenvalue map with the Frobenius norm on the space of matrices, together with standard Gaussian concentration bounds and the known mean value of the above eigenvalue, it follows that (cf. \cite[Lemmas 2.3.2 and 2.3.3]{AGZintroRM})
\begin{equation} \label{eq:topeigconc}
\EE[e^{-2\xi \clB_{k - j + 2}^-(j-1,T)}] \leq e^{4\xi\sqrt{(k - j + 2)T} + 2\xi^2 T}.
\end{equation}
Applying the bounds \eqref{eq:Bstarest} and \eqref{eq:topeigconc} to \eqref{eq:Xjproblb1}, we obtain, for some $c_1<\infty$,
\begin{equation} \label{eq:probjustabout}
\PP(\inf_{0 \leq s \leq T} X_j(s) - x_j \leq -u) \leq c_1 e^{-\xi u + 2\xi^2 T} \sum_{k \geq j} e^{2\xi\sqrt{(k - j + 2)T}} \EE[e^{-\xi(x_k - x_j)}].
\end{equation}
The bound \eqref{eq:changeprobbd} follows from this and \eqref{eq:sqrtcomp}.

In the calculation \eqref{eq:Xjproblb1}, if we replace $X_j(s) - x_j$ with simply $X_j(s)$, then the everything carries through as before except that instead of \eqref{eq:probjustabout}, we obtain 
\[
\begin{split}
\PP(\inf_{0 \leq s \leq T} X_j(s) \leq -u) & \leq c_1 e^{-\xi u + 2\xi^2 T} \sum_{k \geq j} e^{2\xi\sqrt{(k - j + 2)T}} \EE[e^{-\xi x_k}] \\
& = c_1 e^{-\xi u - c_0(\xi)j + 2\xi^2 T} \sum_{k \geq j} e^{2\xi\sqrt{(k - j + 2)T} - c_0(\xi)(k - j)} e^{c_0(\xi)k}\EE[e^{-\xi x_k}],
\end{split}
\]
and the bound \eqref{eq:posprobbd} follows from this and \eqref{eq:posbd}. 
\end{proof}

\begin{lemma} \label{lem:loctimemoments}
Under Assumption \ref{assump:initdiff}, for all $T \in \half$, $\xi \in \ohalf$, and $r \in [1,\infty)$, the following bounds hold: 
\begin{equation} \label{eq:laplocbd}
\sup_{j \in \ZZ} \EE e^{\xi L_{(j,j+1)}(T)} < \infty;
\end{equation}
and consequently,
\begin{equation} \label{eq:momlocbd}
\sup_{j \in \ZZ} \EE|L_{(j,j+1)}(T)|^r < \infty.
\end{equation}
\end{lemma}

\begin{proof}
For all $M \in \ZZ$ and $j \geq -M$, we have 
\[
\begin{split}
\sum_{i = -M}^j (\frac{q}{p})^i X_i(T) & = \sum_{i = -M}^j (\frac{q}{p})^i \left( x_i + B_i(T) + p L_{(i-1,i)}(T) - q L_{(i,i+1)}(T) \right) \\
& = \sum_{i = -M}^j (\frac{q}{p})^i \left( x_i + B_i(T) \right) + \sum_{i = -M}^j (\frac{q}{p})^{i-1}qL_{(i-1,i)}(T) - (\frac{q}{p})^i qL_{(i,i+1)}(T) \\
& = \sum_{i = -M}^j (\frac{q}{p})^i \left( x_i + B_i(T) \right) + (\frac{q}{p})^{-M-1}q L_{(-M-1,-M)}(T) - q( \frac{q}{p})^j L_{(j,j+1)}(T).
\end{split}
\]
In view of the assumption \eqref{eq:aloc}, taking the limit as $M \to \infty$ and rearranging terms, we obtain
\[
L_{(j,j+1)}(T) = q^{-1}\sum_{i = -\infty}^j (\frac{q}{p})^{i - j} \left( x_i - X_i(T) + B_i(T)\right),
\]
where the infinite sum on the right-hand side converges a.s. By Jensen's inequality, equation \eqref{eq:WFSsolution}, and the bound \eqref{eq:plb-pointwise},
\[
\begin{split}
e^{\xi L_{(j,j+1)}(T)} & \leq \frac{q - p}{q} \sum_{i = -\infty}^j (\frac{q}{p})^{i - j} \exp( \frac{\xi}{q - p} (x_i - X_i(T) + B_i(T)) ) \\
& \leq \frac{q - p}{q} \sum_{i = -\infty}^j (\frac{q}{p})^{i - j} \max_{k \geq i}\exp( \frac{\xi}{q - p} (x_i - x_k - \clB_{k-i+1}^-(i,T) + B_i(T)) ).
\end{split}
\]
Replacing the maximum with a sum and taking the expectation of both sides, we obtain 
\begin{equation} \label{eq:laplacelocbd1}
\begin{split}
\EE e^{\xi L_{(j,j+1)}(T)} & \leq \frac{q - p}{q} \sum_{i = -\infty}^j (\frac{q}{p})^{i - j} \sum_{k \geq i} \EE[ e^{-\frac{\xi(x_k - x_i)}{q - p}} ] \EE[ e^{-\frac{2\xi\clB_{k - i + 1}^-(i,T)}{q - p}} ]^{1/2} \EE[ e^{\frac{2\xi B_i(T)}{q - p}} ]^{1/2},
\end{split}
\end{equation}
here using independence of the initial data and the Brownian motions and the Cauchy-Schwarz inequality. Note that
\begin{equation} \label{eq:lapbmbd}
\EE[ e^{\frac{2\xi B_i(T)}{q - p}} ] = e^{\frac{2\xi^2 T}{(q - p)^2}},
\end{equation}
and  as in \eqref{eq:topeigconc},
\begin{equation} \label{eq:lapeigbd}
\EE[ e^{-\frac{2\xi\clB_{k - i + 1}^-(T)}{q - p}} ] \leq e^{\frac{4\xi\sqrt{(k - i + 1)T}}{(q - p)} + \frac{2\xi^2 T}{(q - p)^2}}.
\end{equation}
Applying \eqref{eq:lapbmbd} and \eqref{eq:lapeigbd} to \eqref{eq:laplacelocbd1}, we conclude that 
\[
\EE e^{\xi L_{(j,j+1)}(T)} \leq c_1(\frac{q - p}{q}) e^{\frac{2\xi^2 T}{(q - p)^2}} \sum_{i = -\infty}^j (\frac{q}{p})^{i - j} \sum_{k \geq i} e^{\frac{4\xi\sqrt{(k - i + 1)T}}{(q - p)}} \EE[ e^{-\frac{\xi(x_k - x_i)}{q - p}} ].
\]
The bound \eqref{eq:laplocbd} follows since by Assumption \ref{assump:initdiff}, the inner sum is bounded by some finite constant $C(p,\xi,T) \in \ohalf$ not depending on $i$. The bound \eqref{eq:momlocbd} is immediate from \eqref{eq:laplocbd}.
\end{proof}

\section{Fluctuation estimates for the counting function} \label{ssec:countfluc}

We next establish some key fluctuation estimates for the counting function \eqref{eq:rescaled_count} in the weakly asymmetric case. Recall the choice $p=p_{\eps}$ introduced below Theorem \ref{thm:stationary}. Recall the scaled processes \eqref{eq:rescaled_sys} and the scaled counting function \eqref{eq:rescaled_count}. The lemma below furnishes local fluctuation estimates for the counting function when the initial data $\{x_i\}$ (before rescaling) is drawn from the distribution $\gamma_1$, defined by \eqref{eq:hominit} (i.e. $\init$ introduced below \eqref{eq:weak-asym-intro} equals $\gamma_1$). These estimates are key to the uniform approximation of $\clg^{\veps}(t,x)$ by $\tilde \cle^{\eps}_{\theta}(t,x)$ on compact subsets, claimed in \eqref{eq:uniapprox}.

{The proof proceeds by reducing the local oscillation of the counting function to estimates for a finite collection of crossing events. First, we control the displacement of the reference particle \(X_0^\varepsilon\) over \([0,T]\) in \eqref{eq:415a}-\eqref{eq:415b} by comparison with the totally asymmetric systems \(p=0\) and \(p=1\), obtaining Gaussian tail bounds that enables us to restrict our fluctuation analysis to a deterministic compact space-time set, depending on $\eps$. We then partition this set into rectangles of suitably chosen dimensions, and use stationarity of the gap process viewed from the tagged particle to further reduce the analysis to each such rectangle via a union bound. The fluctuation estimate for the counting function within each rectangle is controlled by three elementary quantities introduced below \eqref{eq:751n}: the number \(U\) of particles initially inside the interval, and the numbers \(V\) and \(W\) of particles entering the rectangle from below and above during the short time window of length \(\varepsilon\). The variable \(U\) has Poisson tails under the exponential-gap initialization, while \(V\) and \(W\) are controlled by comparison with the one-sided totally asymmetric systems and Brownian last-passage estimates. Combining these estimates gives the tail bound in the first part of the lemma; the moment estimate in the second part follows by integrating this tail bound, with separate treatments of the regimes \(\theta>1/2\) and \(\theta\le 1/2\).}


\begin{lemma} \label{lem:countfluc}
Let $\{x_i\}_{i \in \ZZ} \sim \gamma_1$. Fix $T \in \half$ and $\theta \in (0,1)$. For $(w,t_0) \in \RR \times [0,T]$, let 
\begin{equation}\label{eq:thetadefn}
\Theta(w,t_0, \theta) \equiv \Theta(w,t_0) = \left[ w - \eps^{\frac{1}{4\theta}}, w + \eps^{\frac{1}{4\theta}} \right] \times \left[ t_0, t_0 + \eps \right].
\end{equation}
There exists $\eps_0 \in (0,1)$ and $C_1, C_2, C_3 \in \half$ such that for all $\eps \in (0,\eps_0]$, $w \in \RR$, $t_0 \in [0,T]$,  
\[
\begin{split}
& \PP\left( \sup_{(x,t) \in \Theta(w,t_0)} \left| N_\eps(t,x-\sigma\eps^{-1/4}t) - N_\eps(t_0,w - \sigma\eps^{-1/4}t_0) \right| \geq \ell \right) \\
& \leq C_1 \left( \eps^{-\frac{1}{2} - (\frac{1}{4\theta} \wedge \frac{3}{4})}  + \eps^{-(\frac{1}{4\theta} \wedge \frac{3}{4})} (|w| + 1)\right) \left[ \exp( -C_2 \ell \log( \eps^{\frac{1}{2} - (\frac{1}{4\theta} \wedge \frac{3}{4})} \ell) ) + \exp( - C_2 \ell \log \ell )\right],
\end{split}
\]
for all $\ell \geq \ell_0(\eps) := (C_3\eps^{-\frac{1}{2} + (\frac{1}{4\theta} \wedge \frac{3}{4}})\vee 4$.

Furthermore, for all $r \geq 1$, there exists a $d=d(r)>0$ and $\eps_0 = \eps_0(r)  \in (0,1)$, such that
\begin{equation}\label{eq:907nn}
\begin{split}
\sup_{w \in \RR} \sup_{\eps \in (0,\eps_0]} \frac{\eps^{-\frac{1}{4}[(\frac{1}{\theta} - 1) \wedge \frac{1}{2}]}}{(|w| + 1)^{d}} \left\| e^{\sigma \eps^{1/4}  \sup_{(x,t) \in \Theta(w,t_0)} \left| N_\eps(t,x-\sigma\eps^{-1/4}t) - N_\eps(t_0,w - \sigma\eps^{-1/4}t_0) \right|} -1 \right\|_r
\end{split}
\end{equation}
 is finite for $t_0 \in [0,T]$.
\end{lemma}
\begin{proof}
Fix $T, \theta$ as in the statement of the lemma.
Recall the processes $(X^0, L^0)$ with $p=0$ and $q=1$, introduced below Theorem \ref{thm:exist+}. 
{We also introduce the analogous process $(X^1,L^1)$ with $p = 1$ and $q = 0$ and starting from the same initial data as $X$ and $X^0$. By \eqref{eq:plb-pointwise} and a symmetric argument applied to the process $X^1$, we have} 
\begin{equation} \label{eq:dombyta}
{X_j^0(t) \leq X_j(t) \leq X_j^1(t) \text{ for all } t \geq 0, j \in \ZZ.}
\end{equation}
Since the initial distribution for $X$ and $X^0$ is $\gamma_1$, it follows from \cite[Proposition 2.7]{WFSbook} that $X^0_0(t) +t$ is a standard Brownian motion.
Thus, for any $b \ge 1$, with $B$  a standard Brownian motion,
\begin{multline}
\PP(\inf_{0\le t \le T} X_0^{\eps}(t) \le - \eps^{-1/2}T-b)
\le \PP(\inf_{0\le t \le T} (B(t)- \eps^{-1/2}t) \le - \eps^{-1/2}T-b)\\
\le  \PP(\inf_{0\le t \le T} B(t) \le -b) \le 2 \exp\{-b^2/2T\}.\label{eq:415a}
\end{multline}
{A similar calculation applied to the process $X^1$ shows that}
\begin{equation}\label{eq:415b}
\PP(\sup_{0\le t \le T} X_0^{\eps}(t) \ge  \eps^{-1/2}T+b) \le 2 \exp\{-b^2/2T\}.
\end{equation}

For $w \in \RR$, let $A=A(w) := |w|+1$.
For a given $b\ge 1$, consider the interval
$$\cli(b):= [-A - \eps^{-1/2}(1+\sigma)(1+T)-b, A + \eps^{-1/2}(1+\sigma)(1+T)+b],$$
and consider a partition of this interval as 
$$-A - \eps^{-1/2}(1+\sigma)(1+T)-b=y_0 < y_1 <\cdots < y_N = A + \eps^{-1/2}(1+\sigma)(1+T)+b, 
$$
such that, with $a(\theta):= \frac{1}{4\theta} \wedge \frac{3}{4}$, 
$$
y_i-y_{i-1} = (1+\sigma) \eps^{a(\theta)} \text{ for $1 \leq i \leq N-1$}, \quad y_N - y_{N-1} \leq (1+\sigma) \eps^{a(\theta)},
$$
and
\begin{equation}
N \le 3(1+\sigma)^{-1}\eps^{-a(\theta)}(A+b) + 3\eps^{-1/2-a(\theta)}(1+T). \label{eq:758mm}
\end{equation}
For $t, t_0 \in [0,T]$, consider the  point process on $\RR$ with  points
$$\{X^{\eps, t_0}_i(t) := X_i^{\eps}(t)- X_0^{\eps}(t_0), \; i \in \ZZ\},
$$
and let for $x \in \RR$
$N_{\eps}^{t_0}(t, x) := \max\{i \in \ZZ: X^{\eps, t_0}_i(t) \le x\}$.
Then, for $\ell >0$,
\begin{equation} \label{eq:diffneps-probbd}
\begin{split} 
&\PP\left( \sup_{(x,t) \in \Theta(w,t_0)} \left| N_\eps(t,x-\sigma\eps^{-1/4}t) - N_\eps(t_0,w - \sigma\eps^{-1/4}t_0) \right| \geq \ell \right)\\
&\quad= \PP\left( \sup_{(x,t) \in \Theta(w,t_0)} \left| N_\eps^{t_0}(t,x-X_0^{\eps}(t_0)-\sigma\eps^{-1/4}t) - N_\eps^{t_0}(t_0,w - X_0^{\eps}(t_0) -\sigma\eps^{-1/4}t_0) \right| \geq \ell \right)\\
&\quad\le \PP\left( \sup_{(y,z,t) \in S_{\eps}(A,b)} \left| N_\eps^{t_0}(t,y) - N_\eps^{t_0}(t_0,z) \right| \geq \ell \right)\\
&\qquad+ \PP (\inf_{0\le t \le T} X_0^{\eps}(t) \le - \eps^{-1/2}T-b)
+ \PP (\sup_{0\le t \le T} X_0^{\eps}(t) \ge  \eps^{-1/2}T+b)\\
&\quad\le \PP\left( \sup_{(y,z,t) \in S_{\eps}(A,b)} \left| N_\eps^{t_0}(t,y) - N_\eps^{t_0}(t_0,z) \right| \geq \ell \right) + 4 \exp\{-b^2/2T\},
\end{split}
\end{equation}
where 
\begin{equation*}
S_{\eps}(A,b):= \{(y,z,t): t \in [t_0, t_0+\eps],
y,z \in \cli(b),\;
|y-z| \le (1+\sigma) \eps^{a(\theta)}\},
\end{equation*}
and the second inequality in the above display is a consequence of \eqref{eq:415a}-\eqref{eq:415b}, whereas the first inequality comes from the observation that, on the set
where $X_0^{\eps}(t) \in [- \eps^{-1/2}T-b,  \eps^{-1/2}T+b]$,
$$|(x-X_0^{\eps}(t_0)-\sigma\eps^{-1/4}t) - (w - X_0^{\eps}(t_0) -\sigma\eps^{-1/4}t_0|
\le |x-w| + \sigma\eps^{-1/4} |t-t_0| \le \eps^{1/4\theta} + \sigma\eps^{3/4} \le
(1+\sigma) \eps^{a(\theta)},
$$
and
$$
x-X_0^{\eps}(t_0)-\sigma\eps^{-1/4}t \mbox{ and } w - X_0^{\eps}(t_0) -\sigma\eps^{-1/4}t_0
$$
are both in $\cli(b)$.

By the definition of the partition $\{y_j\}$, for any $(y,z,t) \in S_{\eps}(A,b)$, there is at most one $y_{j+1}$ between $y$ and $z$, {for some $0 \leq j \leq N-2$. If no such $y_{j+1}$ lies between $y$ and $z$ and $[y_j,y_{j+1}]$ is the interval containing both, then we write
\begin{equation} \label{eq:nosplitbound}
|N_\eps^{t_0}(t,y) - N_\eps^{t_0}(t_0,z)| \leq |N_\eps^{t_0}(t,y) - N_\eps^{t_0}(t_0,y_j)| + |N_\eps^{t_0}(t_0,z) - N_\eps^{t_0}(t_0,y_j)|.
\end{equation}
On the other hand, if $y \leq y_{j+1} \leq z$, then by definition of the partition and $S_\eps(b,A)$, $y \in [y_j,y_{j+1}]$ and $z \in [y_{j+1},y_{j+2}]$, and we write
\begin{equation} \label{eq:splitbound}
\begin{split}
& \left| N_\eps^{t_0}(t,y) - N_\eps^{t_0}(t_0,z) \right| \\
& \leq \left| N_\eps^{t_0}(t,y) - N_\eps^{t_0}(t_0,y_{j}) \right| + \left| N_\eps^{t_0}(t_0,y_{j+1}) - N_\eps^{t_0}(t_0,y_j) \right| + \left| N_\eps^{t_0}(t_0,z) - N_\eps^{t_0}(t_0,y_{j+1}) \right|.
\end{split}
\end{equation}
A symmetric bound holds when $z \leq y_{j+1} \leq y$. In both the cases of \eqref{eq:nosplitbound} and \eqref{eq:splitbound}, if the left-hand side is bounded below by $\ell$, then one of the terms on the right-hand side is bounded below by $\ell/3$. Hence, a union bound argument applied to \eqref{eq:diffneps-probbd} gives  
}
\begin{multline*}
\PP\left( \sup_{(x,t) \in \Theta(w,t_0)} \left| N_\eps(t,x-\sigma\eps^{-1/4}t) - N_\eps(t_0,w - \sigma\eps^{-1/4}t_0) \right| \geq \ell \right)\\
\le 3\PP\left( \max_{j\le N}\sup_{t \in [t_0, t_0+\eps]}\sup_{y \in [y_j, y_{j+1}]} \left| N_\eps^{t_0}(t,y) - N_\eps^{t_0}(t_0,y_j) \right| \geq \ell/3 \right) + 4 \exp\{-b^2/2T\},
\end{multline*}
where $N$ is as introduced below \eqref{eq:415b}.
Using the observation that from the stationary property of $\gamma_1$,
$\{N_{\eps}^{t_0}(t+t_0, \cdot), t \ge 0\}$ has the same distribution as
$\{N_{\eps}(t, \cdot), t \ge 0\}$, 
the right side above can be bounded above by
\begin{equation}\label{eq:751n}
N \times 3\sup_{a \in \cli(b)} \PP\left(\sup_{s \in [0,\eps]}
\sup_{y \in [a, a+(1+\sigma)\eps^{a(\theta)}]} |N_{\eps}(s,y) - N_{\eps}(0,a)| \ge \ell/3\right) + 4 \exp\{-b^2/2T\}.
\end{equation}
We now estimate the above probability.
Define
$$
U:= \#\{i \in \ZZ: X_i^{\eps}(0) \in [a, a+ (1+\sigma)\eps^{a(\theta)}], \;\;
V:= \#\{i \in \ZZ: X_i^{\eps}(0) < a, \; \sup_{s \in [0, \eps]} X_i^{\eps}(s) >a\}
$$
and
$$
W:= \#\{i \in \ZZ: X_i^{\eps}(0) > a + (1+\sigma)\eps^{a(\theta)}, \;
\inf_{s\in [0,\eps]}X_i^{\eps}(s) < a + (1+\sigma)\eps^{a(\theta)}\}.
$$
Then, using the inequality
$$
|N_{\eps}(s,y) - N_{\eps}(0,a)| \le |N_{\eps}(s,y) - N_{\eps}(0,y)|
+ |N_{\eps}(0,y) - N_{\eps}(0,a)|,
$$
we see that
\begin{equation}\label{eq:433n}
\sup_{s \in [0,\eps]}
\sup_{y \in [a, a+(1+\sigma)\eps^{a(\theta)}]} |N_{\eps}(s,y) - N_{\eps}(0,a)| \le 2U+V+W.
\end{equation}
Recalling the definition of $\gamma_1$ we see that 
\begin{equation}\label{eq:514n}
U \sle \mbox{Poi}((1+\sigma)\eps^{a(\theta)-1/2}) + 1,
\end{equation}
where $\mbox{Poi}(\lambda)$ denotes a Poisson random variable with mean $\lambda$.

{Next, 
observe that by the second inequality in \eqref{eq:dombyta},}
\begin{equation}\label{eq:413n}
V \sle \tilde V_a^{\eps},
\end{equation}
where
\begin{align*}
\tilde V_a^{\eps} &:= \#\{i: X_i^{\eps}(0) <a, \; \sup_{s \in [0,\eps]} \eps^{1/2} X_i^1(\eps^{-1}s) > a\}\\
&= \#\{i: X_i(0) <a\eps^{-1/2}, \; \sup_{s \in [0,1]}  X_i^1(s) > a\eps^{-1/2}\}.
\end{align*}
For $y \ge 1$, write $\tilde V_a^{\eps} = \tilde V_{a,y}^{1,\eps} + \tilde V_{a,y}^{2,\eps}$, where
\begin{align*}
\tilde V_{a,y}^{1,\eps} &:= \#\{i: X_i(0) <a\eps^{-1/2}-y, \; \sup_{s \in [0,1]}  X_i^1(s) > a\eps^{-1/2}\},\\
\tilde V_{a,y}^{2,\eps} &:= \#\{i: a\eps^{-1/2}-y\le X_i(0) <a\eps^{-1/2}, \; \sup_{s \in [0,1]}  X_i^1(s) > a\eps^{-1/2}\}.
\end{align*}
Note that
\begin{equation}\label{eq:410nn}
\tilde V_{a,y}^{2,\eps} \sle \mbox{Poi}(y) +1.
\end{equation}
We will now estimate $\PP(\tilde V_{a,y}^{1,\eps} \ge 1)$ for large $y$.
For this, we will construct an auxiliary particle system that suitably dominates the original particle system, as follows.
Let $I_k:= [\eps^{-1/2}a - (k+1)^2y, \eps^{-1/2}a - k^2y]$.
Consider the particle system as in Definition \ref{def:biinf} with $p=1$, denoted as $\tilde X_i^1$, driven using the same Brownian motions as $X$ and $X^1$,  but with
initial values $\tilde X_i^1(0)$ defined as 
$$
\tilde X_i^1(0) := 
\begin{cases}
\eps^{-1/2}a - k^2y & \mbox{ if } X_i(0) \in I_k, \; k \ge 1\\
X_i(0) & \mbox{ if } X_i(0) > \eps^{-1/2}a - y.
\end{cases}
$$
{Note that the initial data of the system $\tilde X^1$ dominate the initial data of $X^1$. Hence, it follows from the representation analogous to \eqref{eq:WFSsolution} with $p = 1, q = 0$ that 
\begin{equation} \label{eq:dombytilde}
X_i^1(t) \le \tilde X_i^1(t)
\end{equation} 
for all $t\ge 0$ and $i \in \ZZ$}.

We now introduce another singularly interacting Brownian particle system with $p=1$, in which the particles  are divided into clusters and particles in different clusters do not interact with each other. The state process for this system will be denoted as $\hat X^1 = \{\hat X^1_i, i \in \ZZ\}$. 
For $k\ge 1$, the $k$-th cluster in this particle system will be
$\clc_k:= \{i \in \ZZ: \tilde X_i^1(0) = \eps^{-1/2}a - k^2y\}$
and we set $\clc_0:= \{i \in \ZZ:  \tilde X_i^1(0) = X_i^1(0) > \eps^{-1/2}a - y\}$.
The motion of particles in the $k$-th cluster is defined using the same initial condition and Brownian motions as the corresponding particles for $\{\tilde X_i^1\}$, but the evolution equation \eqref{eq:asym_eq} is modified so that the particles in different clusters do not interact.  This is done by setting $L_{(i^*_{k}, i^*_{k}+1)} = L_{(j^*_{k}-1, j^*_{k})}=0$ where $i^*_k$ (resp. $j^*_k$) is the index of the highest (resp. lowest) particle in the $k$-th cluster. For $k=0$, $i^*_k=\infty$ and so only the lowest particle has a modified evolution. Condition 2 in Definition \ref{def:biinf} is also modified so that only particles in each cluster satisfy the monotonicity property. 

For $k \ge 1$, let $E_k$ be the event that the particles in clusters $\clc_k$ and $\clc_{k+1}$ do not meet up to time $t = 1$, i.e.
$$E_k := \{\sup_{i\in \clc_{k+1}}\sup_{s\in [0,1]} \hat X_i^1(s) < \inf_{i\in \clc_{k}}\inf_{s\in [0,1]} \hat X_i^1(s)\}.$$
For $M>1$ and mutually independent Brownian motions $\{B_j, j \ge 1\}$ define
$$
\clv^+_M = \sup_{0 \le t_1\le \cdots \le t_{M-1}\le 1} \sum_{j=1}^{M-1}(B_j(t_j) - B_j(t_{j-1})).$$
Then, from \cite[Lemma 4.2]{banbudrud2025wp} and the observation that $\clv^+_M \leq \clv_{M'}^+$ whenever $M \leq M'$, it follows that, for any $c_0>0$, 
$$
\PP(\sup_{i\in \clc_{k+1}}\sup_{s\in [0,1]} \hat X_i^1(s) > \eps^{-1/2}a - (k+1)^2y +ky)
\le \PP(|\clc_{k+1}| > c_0ky^2) + \PP(\clv^+_{\lceil c_0ky^2\rceil} \ge ky).$$

Choose $y_0>1$ and $c_0$ sufficiently small so that, for all $y\ge y_0$,
$$
\PP(\clv^+_{\lceil c_0ky^2\rceil} \ge ky) = 
\PP\left(\frac{\clv^+_{\lceil c_0ky^2\rceil}}{\sqrt{\lceil c_0ky^2\rceil}} -2 \ge \frac{ky}{\sqrt{\lceil c_0ky^2\rceil}}-2\right) \le \PP\left(\frac{\clv^+_{\lceil c_0ky^2\rceil}}{\sqrt{\lceil c_0ky^2\rceil}} -2 \ge \sqrt{k}\right).
$$
Then from \cite[Lemma 5.1(i)]{banbudrud2025wp} we have that for some $C_1, C_2 <\infty$ (depending on $c_0$),
$$\PP(\clv^+_{\lceil c_0ky^2\rceil} \ge ky) \le C_1e^{-C_2k^{7/4}y^2} \mbox{ for all } y \ge y_0,\; k \ge 1.
$$
Next, recalling the initial distribution $\gamma_1$, there is a $y_1>y_0$ and $C_1', C_2'<\infty$, such that, for all $y \ge y_1$ and $k\ge 1$,
$$
\PP(|\clc_{k+1}| > c_0ky^2) \le \PP(\mbox{Poi}((2k+3)y) + 1 > c_0ky^2) \le C_1'e^{-C_2'ky^2}.
$$
Thus, for all $y\ge y_1$ and $k\ge 1$,
\begin{equation}\label{eq:343nn}
\PP(\sup_{i\in \clc_{k+1}}\sup_{s\in [0,1]} \hat X_i^1(s) > \eps^{-1/2}a - (k+1)^2y +ky)
\le C_1'e^{-C_2'ky^2} + C_1e^{-C_2k^{7/4}y^2}.
\end{equation}
Also, 
$$
\PP(\inf_{i\in \clc_{k}}\inf_{s\in [0,1]} \hat X_i^1(s)
< \eps^{-1/2}a - k^2y -ky)
=\PP(\inf_{0\le s \le 1} B(s) \le -ky) \le 2 e^{-k^2y^2/2},
$$
here noting that $X_{j_k^*} \overset{d}{=} B$, where $B$ is a standard Brownian motion.

Combining the last two displays, there exist $C_1^*, C_2^*<\infty$ such that, for all $k\ge 1$ and $y\ge y_1$,
$$\PP(E_k^c) \le C_1^* e^{-C_2^*ky^2}.$$
Consequently, choosing $y_2\ge y_1$ such that $1- e^{-C_2^*y_2^2} \ge 1/2$, with $\tilde C_1^* = 2C_1^*$,
$$\PP(\bigcap_{k=1}^{\infty} E_k) \ge 1 - \tilde C_1^* e^{-C_2^*y^2} \mbox{ for all } y \ge y_2.
$$
Note that, {in view of \eqref{eq:dombytilde},} on the event $\bigcap_{k=1}^{\infty} E_k$, 
\begin{equation} \label{eq:yjump-ub}
\sup_{i :  X_i^1(0) < \eps^{-1/2}a-y}\sup_{s \in [0,1]}
X_i^1(s)
\le \sup_{i \in \clc_1}\sup_{s \in [0,1]} \tilde X_i^1(s) = \sup_{i \in \clc_1}\sup_{s \in [0,1]} \hat X_i^1(s).
\end{equation}
{Here the equality holds since, on the event in question, particles starting in $\clc_1$ do not interact with the lower clusters in the interval $[0,1]$, and they are not affected by the higher up particles due to the totally asymmetric nature of the particle interactions.}
Calculations similar to those leading to \eqref{eq:343nn} show that, there exists $\bar C_1, \bar C_2<\infty$ and $y_3 \ge y_2$, so that, for all $y\ge y_3$,
$$
\PP(\sup_{i\in \clc_1}\sup_{s \in [0,1]} \hat X_i^1(s) \ge a\eps^{-1/2})
\le \bar C_1 e^{-\bar C_2 y^2}.
$$
Consequently, in view of \eqref{eq:yjump-ub}, for all $y\ge y_3$,
\begin{equation}\label{eq:810nn}
\PP(\tilde V_{a,y}^{1,\eps}\ge 1) \le \PP((\bigcap_{k=1}^{\infty} E_k)^c)
+ \PP(\sup_{i\in \clc_1}\sup_{s \in [0,1]} \hat X_i^1(s) \ge a\eps^{-1/2}) 
\le \tilde C_1^* e^{-C_2^*y^2}+ \bar C_1 e^{-\bar C_2 y^2}.\end{equation}
Next, using standard tail probability estimates for a Poisson random variable,
there is a $\ell_0>0$ and $c_1>0$ such that, for all $\ell \ge \ell_0$,
\begin{equation}\label{poicon}
 \PP(\mbox{Poi}(\sqrt{\ell \log \ell})+1 \ge \ell -1) \le e^{-c_1 \ell \log \ell}.
\end{equation}
Now for $\ell \ge \max\{4, \ell_0)$ such that $\sqrt{\ell \log \ell} \ge y_3$, taking $y = y(\ell):= \sqrt{\ell \log \ell}$, we have, with $D_1 = \max\{C_1^*, \bar C_1, 1\}$ and
$D_2 = \min\{C_2^*, \bar C_2, c_1\}$,
\begin{align*}
\PP( V \ge \ell) \le \PP(\tilde V^{\eps}_a \ge \ell)
& \le \PP(\tilde V_{a,y(\ell)}^{1,\eps} \ge 1)
+ \PP(\tilde V_{a,y(\ell)}^{2,\eps} \ge \ell -1)\\
&\quad \le 
2D_1 e^{-D_2\ell \log \ell} + \PP(\mbox{Poi}(\sqrt{\ell \log \ell})+1 \ge \ell -1)
\le 3D_1 e^{-D_2\ell \log \ell},
\end{align*}
where the first inequality is from \eqref{eq:413n} and the third inequality follows from \eqref{eq:810nn}, \eqref{eq:410nn} and \eqref{poicon}.

An analogous calculation for $W$ gives a similar upper bound for $\PP(W \ge \ell)$, and we conclude that, for some $\tilde D_1, \tilde D_2 <\infty$,
\begin{equation}\label{eq:516nn}
\PP(V \ge \ell) + \PP(W \ge \ell) \le \tilde D_1 e^{-\tilde D_2\ell \log \ell}, \; \mbox{ for all } \ell \ge 1 \mbox{ and } a \in \RR.
\end{equation}
From \eqref{eq:433n} we now have that
\begin{align*}
\sup_{a\in \RR}\PP(\sup_{s \in [0,\eps]}
\sup_{y \in [a, a+(1+\sigma)\eps^{a(\theta)}]} |N_{\eps}(s,y) - N_{\eps}(0,a)| > \ell/3)
\le \sup_{a\in \RR} \PP(2U+V+W > \ell/3)\\
\le \PP(\mbox{Poi}((1+\sigma)\eps^{a(\theta)-1/2})+1 \ge \ell/18) + 
\sup_{a\in \RR}(\PP(V \ge \ell/9) + \PP(W \ge \ell/9)),
\end{align*}
where the second inequality follows from \eqref{eq:514n}.
Using \eqref{eq:516nn} along with the inequality that for some $\kappa_1, \kappa_1', \kappa_2 < \infty$,
\begin{equation}\label{eq:839n}
\PP(1+\mbox{Poi}((1+\sigma)\lambda) \ge \ell/18) \le \kappa_1'\exp\{-\kappa_1 \ell \log \frac{\ell}{\lambda}\} \mbox{ for } \ell \ge \max\{4, \kappa_2\lambda\}, \; \lambda>0,
\end{equation}
we now have that, 
for some
$D_1^*, D_2^*, <\infty$, whenever $\ell \ge \ell(\eps) = \max\{4, \kappa_2 \eps^{a(\theta)-1/2}\}$,
\begin{align*}
&\sup_{a\in \RR}\PP(\sup_{s \in [0,\eps]}
\sup_{y \in [a, a+(1+\sigma)\eps^{a(\theta)}]} |N_{\eps}(s,y) - N_{\eps}(0,a)| > \ell/3)\\
&\quad \le D_1^*\left(\exp\left\{-D_2^* \ell \log \frac{\ell}{\eps^{a(\theta)-1/2}}\right\} +
\exp\{-D_2^* \ell \log \ell\}\right).
\end{align*}
Thus taking $b= \ell$, for $\ell \ge \ell(\eps)$, in \eqref{eq:751n}, we have, using \eqref{eq:758mm} and the above display
\begin{multline}
\PP\left( \sup_{(x,t) \in \Theta(w,t_0)} \left| N_\eps(t,x-\sigma\eps^{-1/4}t) - N_\eps(t_0,w - \sigma\eps^{-1/4}t_0) \right| \geq \ell \right)\\
\le 9D_1^*\left((1+\sigma)^{-1}\eps^{-a(\theta)}(A+\ell) +  \eps^{-1/2-a(\theta)}(1+T)\right)
\left(\exp\left\{-D_2^* \ell \log \frac{\ell}{\eps^{a(\theta)-1/2}}\right\} +
\exp\{-D_2^* \ell \log \ell\}\right)\\
+ 4 \exp\{-\ell^2/2T\}. \label{eq:843nn}
\end{multline}
This completes the proof of the first statement in the lemma.

Consider now the second statement.  Write
$$
\Psi_{\eps}= \Psi_{\eps}(w, t_0) := \sup_{(x,t) \in \Theta(w,t_0)} \left| N_\eps(t,x-\sigma\eps^{-1/4}t) - N_\eps(t_0,w - \sigma\eps^{-1/4}t_0) \right|.
$$
We consider two cases.

{\em Case 1:} $\theta >1/2$. Fix $r\ge 1$. Recall $\kappa_2$ from \eqref{eq:839n}.  We can find $\eps_0 \in (0,1)$ so that $\kappa_2\sigma \eps_0^{\frac{1}{4\theta}-\frac{1}{4}} \le 1$. Then for all $\eps \le \eps_0$, 
\begin{align}
&\EE\left| \exp\{\sigma \eps^{1/4} \Psi_{\eps}\} -1\right|^r\nonumber\\
&\le \EE\left[\left| \exp\{\sigma \eps^{1/4} \Psi_{\eps}\} -1\right|^r
\mathbf{1}\{\Psi_{\eps} \ge \kappa_2 \eps^{\frac{1}{4\theta}-\frac{1}{2}}\}\right]
+ (e\kappa_2\sigma)^r\eps^{r(\frac{1}{4\theta}-\frac{1}{4})}\nonumber\\
&\le 2^r\left(\EE \exp\{2r\sigma\eps^{1/4}\Psi_{\eps}\} +1\right)^{1/2}
\left(\PP(\Psi_{\eps} \ge \kappa_2 \eps^{\frac{1}{4\theta}-\frac{1}{2}})\right)^{1/2}
+ (e\kappa_2\sigma)^r\eps^{r(\frac{1}{4\theta}-\frac{1}{4})}.\label{eq:854nn}
\end{align}
Since $\frac{1}{2} - \frac{1}{4\theta} >0$, we can choose $\eps_1 \le \eps_0$ such that $\kappa_2\eps_1^{\frac{1}{4\theta}-\frac{1}{2}} > 4$.
Then, from \eqref{eq:843nn}, we can find $c_1, c_2<\infty$ such that, for all $z\ge \kappa_2 \vee e$  and $\eps \le \eps_1$,
$$
\PP(\Psi_{\eps} \ge z \eps^{\frac{1}{4\theta}-\frac{1}{2}}) \le c_1 A \exp\{-c_2 z \eps^{\frac{1}{4\theta}-\frac{1}{2}}\}.
$$
Thus, in particular, choosing $\eps_2 = \eps_2(r) \le \eps_1$ such that $c_2 \eps_2^{\frac{1}{4\theta}-\frac{1}{2}} > 2r\sigma$, we obtain for some $C_r <\infty$, for all $\eps \le \eps_2$, $w\in \RR$ and $t_0 \in [0,T]$,
$$\EE[\exp\{2r\sigma\eps^{1/4}\Psi_{\eps}\}] \le 
\EE[\exp\{2r\sigma\eps^{-\frac{1}{4\theta}+\frac{1}{2}}\Psi_{\eps}\}] \le C_rA.
$$
Using the last two estimates in \eqref{eq:854nn}, we have, for some $C'_r < \infty$ and all $\eps \le \eps_2$, 
$w\in \RR$ and $t_0 \in [0,T]$,
$$
\EE\left| \exp\{\sigma \eps^{1/4} \Psi_{\eps}\} -1\right|^r\le
 C'_r\left(\eps^{r(\frac{1}{4\theta}-\frac{1}{4})} + A \exp\{-\frac{c_2\kappa_2}{2}\eps^{\frac{1}{4\theta}-\frac{1}{2}}\}\right).
$$
This proves \eqref{eq:907nn} with $d= 1/r$ for the case $\theta >1/2$. 

{\em Case 2:} $\theta \le 1/2$. Fix $r\ge 1$. Then, choosing $\eps_0 \in (0,1)$ so that $\sigma \eps_0^{1/8}\le 1$, for all
$\eps \le \eps_0$,
\begin{align}
\EE\left| \exp\{\sigma \eps^{1/4} \Psi_{\eps}\} -1\right|^r
&\le \EE\left[\left| \exp\{\sigma \eps^{1/4} \Psi_{\eps}\} -1\right|^r\mathbf{1}\{\Psi_{\eps} \ge \eps^{-1/8}\}\right] + (e\sigma)^r \eps^{r/8}\nonumber\\
&\le 2^r\left(\EE e^{2r\sigma \eps^{1/4}\Psi_{\eps}} +1\right)^{1/2}
\left(\PP(\Psi_{\eps} \ge \eps^{-1/8})\right)^{1/2} + (e\sigma)^r \eps^{r/8}.\label{eq:1019nn}
\end{align}
Choosing $\eps_1\le \eps_0$ so that $\eps^{-1/8} \ge \max\{\kappa_2, 4\}$, we have, applying  \eqref{eq:843nn} with $\ell = \eps^{-1/8}$,  for some $c_1, c_2<\infty$, and all $\eps \le \eps_1$, and all
$w\in \RR$ and $t_0 \in [0,T]$, 
\begin{equation}\label{eq:1018n}
\PP(\Psi_{\eps} \ge \eps^{-1/8}) \le c_1 A e^{-c_2 \eps^{-1/8}\log \eps^{-1}}.
\end{equation}
Furthermore, using \eqref{eq:843nn} once more, we can find a $\beta>0$ such that, for some $c_1', c_2'<\infty$, and all $\eps \le \eps_1$, $m >1$ and all
$w\in \RR$ and $t_0 \in [0,T]$,
$$
\PP(\Psi_{\eps} \ge \beta (\log \eps^{-1} + \log A) +m) \le c_1' e^{-c_2' m \log m}.
$$
Consequently, we can find $C_r<\infty$ such that for all $\eps \le \eps_1$,  and all
$w\in \RR$ and $t_0 \in [0,T]$,
$$
\EE e^{2r\sigma\eps^{1/4}\Psi_{\eps}} \le C_r e^{2r\sigma \eps^{1/4} (\beta(\log \eps^{-1} +\log A) +1)}, 
$$
and thus for some $C'_r<\infty$, for all $\eps \le \eps_1$,  and all
$w\in \RR$ and $t_0 \in [0,T]$,
$$\EE e^{2r\sigma\eps^{1/4}\Psi_{\eps}} \le C'_r A^{2r\sigma \beta}.$$
Using the above estimate and \eqref{eq:1018n} in \eqref{eq:1019nn}, we have, for some $C''_r<\infty$
$$
\EE\left| \exp\{\sigma \eps^{1/4} \Psi_{\eps}\} -1\right|^r \le 
C''_r A^{(2r\sigma \beta+1)/2} e^{-\frac{c_2}{2}\eps^{-1/8}\log \eps^{-1}} + (e\sigma)^r \eps^{r/8}.
$$
This proves \eqref{eq:907nn} with $d= (1+2r\sigma \beta)/2r$, for the case $\theta \le 1/2$.

The result follows. 
\end{proof}

\section{Tightness and Convergence} \label{sec:tight-conv}
{We now turn to the proof of the convergence of the process $\clG^\eps$ and its mollification $\tlE_\theta^\eps$, defined by \eqref{eq:Gdef0} and \eqref{eq:tEdef0}, respectively,} to the stochastic heat equation. As before, throughout this section $\{(X_j,L_{(j,j+1)})\}_{j \in \ZZ}$ denotes the unique strong solution to \eqref{eq:asym_eq} with initial data $\{x_j\} \sim \gamma$ and Brownian motions $\{B_j\}$, where $\gamma$ satisfies \eqref{eq:ainit2} and the local times satisfy \eqref{eq:aloc}. 
Recall we  consider the weakly asymmetric case, i.e. $p= p_{\eps} = (1+ e^{\sigma\eps^{1/4}})^{-1}$.
The  rescaled processes $\{X_j^\eps, L_{(j,j+1)}^\eps,B_j^\eps\}$ are defined as in \eqref{eq:rescaled_sys}. 

{While the convergence of $\clG^\eps$ to SHE will be obtained for the stationary initial distribution $\gamma = \gamma_1$, in view of Remark \ref{rem:gentightness} and Remark \ref{rem:gencont} it is worthwhile to consider more general initial distributions as well. We will specify the assumptions on the distribution $\gamma$ at the beginning of each subsection.}


As noted in the introduction, a major difficulty in our singular interaction model is that the microscopic Hopf--Cole field \(\clG^\eps\) does not, by itself, admit a representation in terms of a microscopic SHE with martingale noise, as in the case of \cite{BertiniGiacomin}. Instead, we work with the mollified field \(\tlE_\theta^\eps\), whose smoothing scale is chosen so as to retain the fluctuation-scale roughness of \(\clG^\eps\), while still allowing a useful semimartingale representation.

The proof has four main steps. First, in equations \eqref{eq:Edef} and \eqref{eq:te-def}, we represent the mollified field as a weighted infinite sum of functions evaluated at particle locations. This representation is designed so that, after applying It\^o's formula, the singular local time terms arising from collisions cancel. The result is a martingale representation for \(\tlE_\theta^\eps\), with explicit drift and bracket terms. This cancellation is the substitute, in the present local time model, for the almost closed microscopic stochastic heat equation available in lattice models such as WASEP \cite{BertiniGiacomin}.

Second, we obtain uniform moment and continuity estimates for \(\tlE_\theta^\eps\). These estimates imply tightness of the sequence $\{\tlE_\theta^\eps\}$ as $\eps \to 0$ in the local-uniform topology and place the prelimit fields in the same  parabolic H\"older regularity class as the limiting SHE. Third, we show that the mollified field $\{\tlE_\theta^\eps\}$ and the original microscopic Hopf--Cole field \(\clG^\eps\) are uniformly close on compact sets with high probability. This last step uses the fluctuation estimates for the counting function from Section~\ref{ssec:countfluc}, together with the continuity bounds proved below, to show that the mollification error vanishes locally uniformly in probability.

Finally, using the martingale representation, we show that any limit point of $\tlE_\theta^\eps$ solves a certain martingale problem that characterizes SHE solutions, completing the proof of Theorem~\ref{thm:maincgce}.

\subsection{{Weighted infinite sums, local time cancellation and martingale representation of the mollified field}} \label{ssec:weighted}

In this subsection, we will assume throughout that $\gamma$ satisfies Assumption \ref{assump:initdiff}. Fix $(t,x) \in \half \times \RR$, $\delta > 0$, and for $(s,y) \in [0,t+\delta) \times \RR$, let 
\[
\psi^{t,x,\delta}(s,y) = \left(\mathbf{1}_{[y,\infty)}\ast p_{t-s+\delta}\right)(x) = \int_y^\infty p_{t - s + \delta}(u - x)du,
\]
and 
\begin{equation} \label{eq:Edef}
\clE^{t,x,\delta, \eps}(s) = \sum_{i \in \ZZ} e^{\sigma\eps^{1/4}i}\psi^{t,x,\delta}(s,X_i^\eps(s)).
\end{equation}
The sum is well-defined (possibly infinite) since all of the terms are non-negative. We also define 
\begin{equation} \label{eq:te-def}
\tlE^{t,x,\delta, \eps}(s) = (1 - e^{-\sigma\eps^{1/4}}) e^{-\sigma \eps^{-1/4}x + \frac{1}{2}\sigma^2\eps^{-1/2}(t+\delta)}\clE^{t,x-\sigma\eps^{-1/4}(t+\delta),\delta,\eps}(s), \text{ for } s \in [0,t+\delta).
\end{equation}
As shown in Lemma \ref{lem:Erep}, this process is connected to the process defined by \eqref{eq:tEdef0} through the equality $\tlE_\theta^\eps(t,x) = \tlE^{t,x,\eps^{1/2\theta},\eps}(t)$.

\subsubsection{{Convergence of infinite sums and a sum-integral identity}}

The following result on the convergence of infinite random sums will be used several times.


\begin{lemma} \label{lem:convergence}
Fix $\eps>0$. Let $T > 0$, and suppose $f \in C([0,T) \times \RR)$ is such that for all $c \in (0,\infty)$ and $T' \in (0,T)$, 
\begin{equation} \label{eq:convcond}
\sup_{(s,y) \in [0,T'] \times \RR} e^{cy^+} |f(s,y)| < \infty.
\end{equation}
Then for all $m \in \ohalf$, 
$r \in [1,\infty)$ and $T' \in [0,T)$,
\begin{equation} \label{eq:Lrtriangle}
\sum_{i \in \ZZ} e^{m i}\EE[\sup_{0 \leq s \leq T'}|f(s,X_i^\eps(s))|^r]^{1/r} <\infty.
\end{equation}
Consequently, for all $m \in \ohalf$,
the sum 
\[
\sum_{i \in \ZZ} e^{m i} f(s,X_i^\eps(s)), \quad s \in [0,T),
\]
converges uniformly on compact subsets of $[0,T)$, in $L^r(\PP)$, for all $r \in [1,\infty)$.
\end{lemma}

{The next identity connects weighted infinite sums of a specific form to integrals involving the Hopf--Cole field. This {\em summation by parts} formula will be repeatedly used throughout the paper.}

\begin{lemma} \label{lem:SBPformula}
Fix $m >0, \eps>0$. For some $T > 0$, suppose $F$ is an element of $C^{0,1}([0,T] \times \RR)$ such that for all $c \in \ohalf$ and $T' \in (0,T)$, 
$$
\sup_{(s,y) \in [0,T'] \times \RR} e^{cy^+} |e^{-m\sigma\eps^{-1/4}y}F(s,y)| < \infty.
$$ 
For $i \in \ZZ$ and $s \in \half$, let $Y_i^\eps(s) = X_i^\eps(s) + \sigma\eps^{-1/4}s$. The following equality holds, for all $s \in [0,T)$,
\begin{equation} \label{eq:sumtoint}
\begin{split}
& (1 - e^{-m\sigma\eps^{1/4}}) \sum_{i \in \ZZ} e^{m\sigma\eps^{1/4}i - m\sigma\eps^{-1/4}Y_i^\eps(s) + \frac{1}{2}m\sigma^2\eps^{-1/2}s} F(s,Y_i^\eps(s)) \\
& \hspace{1.5in} = \int_{\RR} \clG^\eps(s,z)^m \left[ m\sigma\eps^{-1/4} F(s,z) - \partial_z F(s,z) \right]dz,
\end{split}
\end{equation}
where the sums and integrals appearing above converge uniformly on compact subsets of $[0,T)$ in $L^r(\PP)$ for all $r \in [1,\infty)$. 
\end{lemma}

\begin{proof}[Proof of Lemma \ref{lem:convergence}]
Fix $T$ and $f$ as in the lemma. It suffices to show that for all $m \in (0,\infty)$, $r \in [1,\infty)$ and $T' \in [0,T)$, \eqref{eq:Lrtriangle} holds.
 Fix $i \in \NN$. In view of \eqref{eq:convcond}, for any $c \in \ohalf$, there exists a constant $K_1 = K_1(c,T') \in \ohalf$ such that for all $s \in [0,T']$
\begin{equation}
|f(s,X_i^\eps(s))| \leq K_1(c,T') e^{-c(X_i^\eps(s))^+}.
\end{equation}
Furthermore, by the second statement in Lemma \ref{lem:plbXj}, for any $\xi \in \ohalf$, there exists $K_2 = K_2(\xi,T',\eps) \in \ohalf$ such that, for all $u \in \RR$,
\[
\PP(\inf_{0 \leq s \leq T'} X_i^\eps(s) \leq u) \leq K_2(\xi,T',\eps) e^{\xi \eps^{-1/2} u - c_0(\xi)i}.
\]
Hence, for all $c,\xi \in \ohalf$ and $u \in \RR$,
\begin{equation} \label{eq:fXibd1}
\begin{split}
\EE[\sup_{0 \leq s \leq T'}|f(s,X_i^\eps(s))|^r] & = \EE[\sup_{0 \leq s \leq T'}|f(s,X_i^\eps(s))|^r \mathbf{1}\{\inf_{0 \leq s \leq T'} X_i^\eps(s) \leq u\}] \\
& \quad\quad + \EE[\sup_{0 \leq s \leq T'}|f(s,X_i^\eps(s))|^r \mathbf{1}\{\inf_{0 \leq s \leq T'} X_i^\eps(s) > u\}] \\
& \leq K_1(c,T')^r \PP(\inf_{0 \leq s \leq T'} X_i^\eps(s) \leq u) + K_1(c,T')^r e^{-cru} \\
& \leq K_1(c,T')^r K_2(\xi,T',\eps) e^{\xi \eps^{-1/2} u-c_0(\xi)i} + K_1(c,T')^r e^{-cru}.
\end{split}
\end{equation}
Recalling the function $c_0(\cdot)$ from Assumption \ref{assump:initdiff}, we now choose $\xi$ sufficiently large that $c_0(\xi) > 3mr$. We then 
let
\[
u = \xi^{-1}\eps^{1/2}mri, \quad\quad c = 2\xi \eps^{-1/2}r^{-1}.
\]
Then the estimate \eqref{eq:fXibd1} becomes 
\begin{equation} \label{eq:termbd}
\EE[\sup_{0 \leq s \leq T'}|f(s,X_i^\eps(s))|^r] \leq c_1(m,r,T',\eps) e^{ - 2mr i},
\end{equation}
where $c_1(m,r,T',\eps,\eta) \in \ohalf$ is a constant determined by the choices of parameters above. Hence, for $i \geq 1$, we can bound each term of the sum \eqref{eq:Lrtriangle} by $c_2(m,r,T',\eps,\eta)e^{-m i}$ and for $i \leq 0$, we can use \eqref{eq:convcond} to bound each term by $K_1(c,T')e^{mi}$. We conclude that the sum converges.
\end{proof}

\begin{proof}[Proof of Lemma \ref{lem:SBPformula}]
We may write 
\[
\begin{split}
(1 - e^{-m\sigma\eps^{1/4}})e^{m\sigma\eps^{1/4}i} & = ( e^{m\sigma\eps^{1/4}i} - e^{m\sigma\eps^{1/4}(i-1)}) \\ 
& = \left( e^{m\sigma\eps^{1/4}N_\eps(s, Y_i^\eps(s) - \sigma\eps^{-1/4}s)} - e^{m\sigma\eps^{1/4}N_\eps(s, Y_{i-1}^\eps(s) - \sigma\eps^{-1/4}s)}\right).
\end{split}
\]
Hence, substituting the above into the left-hand side of \eqref{eq:sumtoint}, denoting
$a_i:=m\sigma\eps^{1/4}N_\eps(s, Y_i^\eps(s) - \sigma\eps^{-1/4}s) + \frac{1}{2}m\sigma^2\eps^{-1/2}s$,
and performing summation by parts, we obtain
 \[
\begin{split}
& \sum_{i \in \ZZ} \left( e^{a_i} - e^{a_{i-1}} \right)e^{- m\sigma\eps^{-1/4}Y_i^\eps(s)} F(s,Y_i^\eps(s)) \\
& = -\sum_{i \in \ZZ} e^{a_i} \Big( e^{-m\sigma\eps^{-1/4}Y_{i+1}^\eps(s)} F(s,Y_{i+1}^\eps(s)) - e^{- m\sigma\eps^{-1/4}Y_i^\eps(s)} F(s,Y_i^\eps(s)) \Big) \\
& = -\sum_{i \in \ZZ} e^{a_i} \int_{[Y_i^\eps(s), Y_{i+1}^\eps(s))} \frac{\partial }{\partial z} \Big[ e^{-m\sigma\eps^{-1/4}z}F(s,z) \Big] dz \\
& =  -\sum_{i \in \ZZ} \int_{[Y_i^\eps(s), Y_{i+1}^\eps(s))} e^{m\sigma\eps^{1/4}N_\eps(s, z - \sigma\eps^{-1/4}s)+ \frac{1}{2}m\sigma^2\eps^{-1/2}s} \frac{\partial }{\partial z} \Big[ e^{-m\sigma\eps^{-1/4}z}F(s,z) \Big] dz \\
& = - \int_{\RR} e^{m\sigma\eps^{1/4}N_\eps(s, z - \sigma\eps^{-1/4}s) + \frac{1}{2}m\sigma^2\eps^{-1/2}s} \frac{\partial }{\partial z} \Big[ e^{-m\sigma\eps^{-1/4}z}F(s,z) \Big] dz \\
& = \int_{\RR} e^{m\sigma\eps^{1/4}N_\eps(s, z - \sigma\eps^{-1/4}s)+ \frac{1}{2}m\sigma^2\eps^{-1/2}s} \Big[ m\sigma\eps^{-1/4} e^{-m\sigma\eps^{-1/4}z}F(s,z) - e^{-m\sigma\eps^{-1/4}z}\partial_z F(s,z) \Big] dz. 
\end{split}
\]
Note that the sums and integrals appearing above converge uniformly on compact subsets of $[0,T)$, in $L^r(\PP)$ for any $r \in [1,\infty)$, by Lemma \ref{lem:convergence}, with $f(s,y) = e^{-m\sigma\eps^{-1/4}y + \frac{1}{2}m\sigma^2\eps^{-1/2}s}F(s,y - \sigma\eps^{-1/2}s)$. The fourth equality uses the fact that
$Y^{\eps}_i(s) \to \infty$ (resp. $\to -\infty$) as $i\to \infty$ (resp. $i\to -\infty$).
Combining the last two displays we have
\begin{align*}
&(1 - e^{-m\sigma\eps^{1/4}}) \sum_{i \in \ZZ} e^{m\sigma\eps^{1/4}i - m\sigma\eps^{-1/4}Y_i^\eps(s) + \frac{1}{2}m\sigma^2\eps^{-1/2}s} F(s,Y_i^\eps(s))\\
&=  \int_{\RR} \clG^\eps(s,z)^m \Big[ m\sigma\eps^{-1/4} F(s,z) - \partial_z F(s,z) \Big] dz.
\end{align*}
The result follows.
\end{proof}

\subsubsection{{Martingale representation of the mollified field}}

In Lemma~\ref{lem:Eprocess}, we show that the sum in \eqref{eq:Edef} is well-defined in a cetain uniform sense and identify a key local time cancellation property to obtain a martingale representation for the process defined in \eqref{eq:Edef}. In Lemma~\ref{lem:Erep}, the process \eqref{eq:te-def} is shown to arise as a mollification of the original Hopf--Cole field $\clG^\eps$.


\begin{lemma} \label{lem:Eprocess} 
Fix $(t,x) \in \half\times \RR$ and $\delta, \eps>0$.
\begin{enumerate}[label = (\roman*)]

\item The sum \eqref{eq:Edef} defining $\{\clE^{t,x,\delta, \eps}(s), s \in [0,t+\delta)\}$ converges uniformly on compacts subsets of $[0,t+\delta)$, in $L^r(\PP)$, for all $r \in [1,\infty)$. In particular, $\sup_{0 \leq s \leq T'}  \clE^{t,x,\delta,\eps}(s) \in L^r(\PP)$
for all $r \in [1,\infty)$ and $T' \in [0,t+\delta)$. 

\item For all $s \in [0,t+\delta)$, the following equality holds:
\begin{equation} \label{eq:mart_rep}
\clE^{t,x,\delta,\eps}(s) - \clE^{t,x,\delta,\eps}(0) = -\sum_{i \in \ZZ} e^{\sigma\eps^{1/4}i} \int_0^s p_{t-u+\delta}(X_i^\eps(u) - x) dB_i^\eps(u),
\end{equation}
where the sum on the right-hand side converges uniformly on compact subsets of $[0,t + \delta)$, in $L^r(\PP)$, for all $r \in [1,\infty)$. In particular, $\{\clE^{t,x,\delta,\eps}(s), s \in [0, t + \delta)\}$ is a continuous square-integrable martingale.
\end{enumerate}
\end{lemma}


\begin{lemma} \label{lem:Erep}
Fix $(t,x) \in \half\times \RR$ and $\delta \in (0,1)$.
For all $s \in [0,t + \delta)$ and $\eps>0$,
\[
\tlE^{t,x,\delta,\eps}(s) = \int_{\RR} \clG^\eps(s,u) p_{t - s + \delta}(u - x)du.
\]
\end{lemma}

\begin{proof}[Proof of Lemma \ref{lem:Eprocess}]
Consider first part (i). We apply Lemma \ref{lem:convergence} with $m=\sigma \eps^{1/4}$, $f(s,y) = \psi^{t,x,\delta}(s,y)$ and $T = t + \delta$. To see that the hypothesis is satisfied, note that for each $(s,y) \in [0,t+\delta) \times \RR$, a standard Gaussian tail bound gives
\begin{equation} \label{eq:psitailbd}
0 \leq \psi^{t,x,\delta}(s,y) \leq \exp(-\frac{((y - x)^+)^2}{2(t - s + \delta)}),
\end{equation}
from which \eqref{eq:convcond} follows, for all $T' \in [0,t+\delta)$. 

We now prove part (ii). Note that $\psi^{t,x,\delta}$ is in $C^{1,2}([0,t+\delta)\times \RR)$ and it solves the time-reversed heat equation on $[0, t+\delta)$, namely
$$\frac{\partial}{\partial s} \psi^{t,x,\delta}(s,y) + \frac{1}{2}
\frac{\partial}{\partial y^2}\psi^{t,x,\delta}(s,y)=0, \mbox{ for all } s \in [0,t + \delta).$$
Thus, It\^{o}'s formula yields, for $0\le s < t+\delta$,
\[
\begin{split}
\psi^{t,x,\delta}(s,X_i^\eps(s)) & = \psi^{t,x,\delta}(0,X_i^\eps(0)) - \int_0^s p_{t-u+\delta}(X_i^\eps(u) - x) dX_i^\eps(u) \\
& = \psi^{t,x,\delta}(0,X_i^\eps(0)) - \int_0^s p_{t-u+\delta}(X_i^\eps(u) - x) dB_i^\eps(u) \\
& \hspace{0.5in} + \frac{1}{1+e^{-\sigma\eps^{1/4}}}\int_0^s p_{t-u+\delta}(X_i^\eps(u) - x) dL_{(i,i+1)}^\eps(u) \\
& \hspace{1in} - \frac{e^{-\sigma\eps^{1/4}}}{1+e^{-\sigma\eps^{1/4}}}\int_0^s p_{t-u+\delta}(X_i^\eps(u) - x) dL_{(i-1,i)}^\eps(u).
\end{split}
\]
Consequently, for all $M_-,M_+ \in \NN$, $M_{-} < M_+$
\begin{align*} 
\clE^{t,x,\delta,\eps}_{M_-,M_+}(s) & := \sum_{i = -M_-}^{M_+} e^{\sigma\eps^{1/4}i}\psi^{t,x,\delta}(s,X_i^\eps(s)) \\
& = \sum_{i = -M_-}^{M_+} e^{\sigma\eps^{1/4}i} \psi^{t,x,\delta}(0,X_i^\eps(0)) - \sum_{i = -M_-}^{M_+} e^{\sigma\eps^{1/4}i} \int_0^s p_{t-u+\delta}(X_i^\eps(u) - x) dB_i^\eps(u) \\
& \hspace{0.5in} + \sum_{i = -M_-}^{M_+} \bigg\{ \frac{e^{\sigma\eps^{1/4}i}}{1+e^{-\sigma\eps^{1/4}}}\int_0^s p_{t-u+\delta}(X_i^\eps(u) - x) dL_{(i,i+1)}^\eps(u) \\
& \hspace{1.2in} - \frac{e^{\sigma\eps^{1/4}(i-1)}}{1+e^{-\sigma\eps^{1/4}}}\int_0^s p_{t-u+\delta}(X_i^\eps(u) - x) dL_{(i-1,i)}^\eps(u) \bigg\} .
\end{align*}
Rearranging the terms, the right-hand side is equal to
\begin{equation} \label{eq:Ito-M}
\begin{split}
& \sum_{i = -M_-}^{M_+}e^{\sigma\eps^{1/4}i}\psi^{t,x,\delta}(0,X_i^\eps(0)) - \sum_{i = -M_-}^{M_+} e^{\sigma\eps^{1/4}i}\int_0^s p_{t-u+\delta}(X_i^\eps(u) - x) dB_i^\eps(u) \\
& \hspace{0.5in} + \sum_{i = -M_-}^{M_+ - 1} \frac{e^{\sigma\eps^{1/4}i}}{1 + e^{-\sigma\eps^{1/4}}} \int_0^s \{ p_{t - u + \delta}(X_i^\eps(u) - x) - p_{t - u + \delta}(X_{i+1}^\eps(u) - x) \}dL_{(i,i+1)}^\eps(u) \\
& \hspace{0.5in} +\tau^{t,x,\delta,\eps,M_+}_+(s) -\tau^{t,x,\delta,\eps,M_-}_-(s) \\
& = \sum_{i = -M_-}^{M_+}e^{\sigma\eps^{1/4}i}\psi^{t,x,\delta}(0,X_i^\eps(0)) - \sum_{i = -M_-}^{M_+} e^{\sigma\eps^{1/4}i}\int_0^s p_{t-u+\delta}(X_i^\eps(u) - x) dB_i^\eps(u) \\
& \hspace{0.5in} +\tau^{t,x,\delta,\eps,M_+}_+(s) - \tau^{t,x,\delta,\eps,M_-}_-(s),
\end{split}
\end{equation}
where
\[
\begin{split}
& \tau^{t,x,\delta,\eps,M_+}_+(s) := \frac{e^{\sigma\eps^{1/4}M_+}}{1+e^{-\sigma\eps^{1/4}}}\int_0^s p_{t-u+\delta}(X_{M_+}^\eps(u) - x) dL_{(M_+,M_++1)}^\eps(u), \\
& \tau^{t,x,\delta,\eps,M_-}_-(s) := \frac{e^{\sigma\eps^{1/4}(-M_--1)}}{1+e^{-\sigma\eps^{1/4}}}\int_0^s p_{t-u+\delta}(X_{-M_-}^\eps(u) - x) dL_{(-M_--1,-M_-)}(u).
\end{split}
\]
Note that to obtain the last equality in \eqref{eq:Ito-M}, we have used the fact that the terms of the third sum in the previous line vanish, since $dL_{(i,i+1)}^\eps(u)$ is supported on $\{u : X_{i+1}^\eps(u) = X_i^
\eps(u)\}$. 

We will now show that, as $M \to \infty$, $\tau^{t,x,\delta,\eps,M}_\pm \to 0$  uniformly on compact subsets of  $[0,t+\delta)$, in $L^r(\PP)$, for all $r \in [1,\infty)$. Fix $T \in [0,t+\delta)$, and fix $\alpha \in \ohalf$, to be chosen later. We have
\begin{align} \label{eq:Mterm-moment}
& \EE\left[ \left| \sup_{0 \leq s\leq T}\int_0^s p_{t-u+\delta}(X_{M}^\eps(u) - x) dL_{(M,M+1)}^\eps(u) \right|^r \right]\nonumber \\
& = \EE\left[ \left| \sup_{0 \leq s\leq T}\int_0^s p_{t-u+\delta}(X_{M}^\eps(u) - x) dL_{(M,M+1)}^\eps(u) \right|^r \mathbf{1}\left\{ \inf_{0 \leq v \leq T} X_M^\eps(v) \leq x + \alpha M \right\} \right]\nonumber \\
& \quad\quad + \EE\left[ \left| \sup_{0 \leq s\leq T}\int_0^s p_{t-u+\delta}(X_{M}^\eps(u) - x) dL_{(M,M+1)}^\eps(u) \right|^r \mathbf{1}\left\{ \inf_{0 \leq v \leq T} X_M^\eps(v) > x + \alpha M \right\} \right]\nonumber \\
& \leq [2\pi(t + \delta - T)]^{-r/2}\EE\left[ |L_{(M,M+1)}^\eps(T)|^r \mathbf{1}\left\{ \inf_{0 \leq v \leq T} X_M^\eps(v) \leq x + \alpha M \right\} \right]\nonumber \\ 
& \quad\quad + [2\pi(t + \delta - T)]^{-r/2} e^{-\frac{\alpha^2r M^2}{2(t+\delta)}} \EE[|L_{(M,M+1)}^\eps(T)|^r]\nonumber \\
& \leq [2\pi(t + \delta - T)]^{-r/2}(\EE |L_{(M,M+1)}^\eps(T)|^{2r})^{1/2} \PP\left( \inf_{0 \leq v \leq T} X_M^\eps(v) \leq x + \alpha M \right)^{1/2}\nonumber  \\ 
& \quad\quad + [2\pi(t + \delta - T)]^{-r/2} e^{-\frac{\alpha^2r M^2}{2(t+\delta)}} \EE[|L_{(M,M+1)}^\eps(T)|^r].
\end{align}
Here the first inequality in \eqref{eq:Mterm-moment} follows by noting that, for all $u \in [0,T]$ 
\begin{equation}\label{eq:pt758}
\begin{split}
& p_{t - u + \delta}(y - x) \leq (2\pi(t + \delta - T))^{-1/2}, \text{ and } \\
& p_{t - u + \delta}(y - x) \leq (2\pi(t + \delta - T))^{-1/2}e^{-\frac{\alpha^2 M^2}{2(t+\delta)}} \text{ if $y \geq x + \alpha M$}.
\end{split}
\end{equation}
By the second statement in Lemma \ref{lem:plbXj} and Lemma \ref{lem:loctimemoments}, the last last line of \eqref{eq:Mterm-moment} is bounded above by
\[
c_3(r,T,\xi,\eps)[2\pi(t + \delta - T)]^{-r/2}\left\{ e^{\frac{1}{2}\xi \eps^{-1/2}(x + M\alpha) - \frac{1}{2}c_0(\xi)M} + e^{-\frac{\alpha^2r M^2}{2(t+\delta)}} \right\},
\]
where $\xi \in \ohalf$ is arbitrary. Therefore, 
\[
\begin{split}
& \EE[\sup_{0 \leq s \leq T}| \tau^{t,x,\delta,\eps,M}_+(s)|^r] \\ & \leq c_4(r,T,\xi, \eps) (t + \delta - T)^{-r/2} e^{\sigma\eps^{1/4}rM}\left\{ e^{\frac{1}{2}\xi \eps^{-1/2}(x + M\alpha) - \frac{1}{2}c_0(\xi)M} + e^{-\frac{\alpha^2r M^2}{2(t+\delta)}} \right\}.
\end{split}
\]
Using the fact that $c_0(\xi) \to \infty$ as $\xi \to \infty$, we choose $\xi$ sufficiently large so that 
\[
\eta := \frac{1}{2}c_0(\xi) - \frac{1}{2} - \sigma\eps^{1/4}r > 0,
\]
and let $\alpha = \eps^{1/2}/\xi$. Then
\[
\EE[\sup_{0 \leq s \leq T}| \tau^{t,x,\delta,\eps,M}_+(s)|^r] \leq c_4(r,T,\xi, \eps)(t + \delta - T)^{-r/2} \left\{ e^{\frac{1}{2}\xi\eps^{-1/2}x - \eta M} + e^{-\frac{\eps r M^2}{2\xi^2(t+\delta)} + \sigma\eps^{1/4}rM} \right\},
\]
and the right-hand side converges to zero, as $M \to \infty$. Furthermore, using \eqref{eq:pt758} we have 
\[
\begin{split}
 \EE[\sup_{0 \leq s \leq T}| \tau^{t,x,\delta,\eps,M}_-(s)|^r] 
& \leq e^{r\sigma\eps^{1/4}(-M-1)} \cdot \frac{1}{(2\pi(t + \delta - T))^{r/2}} \cdot \EE[|L_{(-M-1,-M)}^{\eps}(T)|^r] \\
& \leq e^{r\sigma\eps^{1/4}(-M-1)} \cdot \frac{1}{(2\pi(t + \delta - T))^{r/2}} c_5(r,T,\eps),
\end{split}
\]
using Lemma \ref{lem:loctimemoments} to obtain the last inequality. The right-hand side converges to zero as $M \to \infty$. This establishes the desired convergence of the terms $\tau^{t,x,\delta,\eps,M}_\pm$ to zero.

To conclude the proof of (ii), note that by (i), as $M_- \wedge M_+ \to \infty$, $\clE_{M_-,M_+}^{t,x,\delta,\eps} \to \clE^{t,x,\delta,\eps}$ on $[0,t+\delta)$, uniformly on compacts, in $L^r(\PP)$, for all $r \in [1,\infty)$. 
Hence, we also have
\[
\sum_{i = -M_-}^{M_+} e^{\sigma\eps^{1/4}i}\psi^{t,x,\delta}(0,X_i^\eps(0)) = \clE_{M_-,M_+}^{t,x,\delta,\eps}(0) \to \clE^{t,x,\delta,\eps}(0)
\]
in $L^r(\PP)$, as $M_- \wedge M_+ \to \infty$. Combining this with the convergence of the terms $\tau^{t,x,\delta,\eps,M_{\pm}}_\pm$ to zero, the martingale term in the last expression in \eqref{eq:Ito-M} must also converges  uniformly on compact subsets of $[0, t+\delta)$, in $L^r(\PP)$, for all $r \in [1,\infty)$, as $M_- \wedge M_+ \to \infty$. The result follows. 
\end{proof}

\begin{proof}[Proof of Lemma \ref{lem:Erep}]
Let
\[
A=\sigma\eps^{-1/4}, \quad a=\sigma\eps^{1/4}, \quad
\tau=t-s+\delta .
\]
Take any $\delta' \in (0,\delta)$. We apply Lemma \ref{lem:SBPformula} with \(m=1\), \(T=t+\delta'\), and
\[
F(s,y)
=
\exp\left\{A(y-x)+\frac{1}{2}A^2(t-s+\delta)\right\}
\int_y^\infty p_{t-s+\delta}\left(z-x+A(t-s+\delta)\right)dz .
\]
The function \(F\) belongs to the class of test functions allowed in Lemma
\ref{lem:SBPformula}. Indeed, \(F\in C^{0,1}([0,t+\delta']\times \RR)\), and the
Gaussian tail gives the required decay in \(y^+\), uniformly for \(s\) in compact
subintervals of \([0,t+\delta')\).

Take any $s \in [0, t+ \delta')$. Recall that $Y_i^\eps(s)=X_i^\eps(s)+As$.
For this choice of \(F\), the left side of \eqref{eq:sumtoint} is
\[
\begin{split}
& (1-e^{-a})
\sum_{i\in\ZZ}
e^{ai-AY_i^\eps(s)+\frac{1}{2}A^2s}
F(s,Y_i^\eps(s)) \\
& =
(1-e^{-a})e^{-Ax+\frac{1}{2}A^2(t+\delta)}
\sum_{i\in\ZZ}e^{ai}
\int_{Y_i^\eps(s)}^\infty
p_{t-s+\delta}\left(z-x+A(t-s+\delta)\right)dz .
\end{split}
\]
Changing variables \(z=u+As\), and using \(Y_i^\eps(s)=X_i^\eps(s)+As\), this becomes
\[
\begin{split}
& (1-e^{-a})e^{-Ax+\frac{1}{2}A^2(t+\delta)}
\sum_{i\in\ZZ}e^{ai}
\int_{X_i^\eps(s)}^\infty
p_{t-s+\delta}\left(u-x+A(t+\delta)\right)du  \\
& =
(1-e^{-a})e^{-Ax+\frac{1}{2}A^2(t+\delta)}
\clE^{t,x-A(t+\delta),\delta,\eps}(s)
= \tlE^{t,x,\delta,\eps}(s),
\end{split}
\]
where the last equality follows from the definition \eqref{eq:te-def}.

It remains to identify the right side of \eqref{eq:sumtoint}. Differentiating \(F\) in
the space variable gives
\[
\partial_y F(s,y)
=
AF(s,y)
-
\exp\left\{A(y-x)+\frac{1}{2}A^2(t-s+\delta)\right\}
p_{t-s+\delta}\left(y-x+A(t-s+\delta)\right).
\]
Therefore
\[
AF(s,y)-\partial_yF(s,y)
=
\exp\left\{A(y-x)+\frac{1}{2}A^2(t-s+\delta)\right\}
p_{t-s+\delta}\left(y-x+A(t-s+\delta)\right).
\]
Noting the identity,
\[
\exp\left\{A(y-x)+\frac{1}{2}A^2(t-s+\delta)\right\}
p_{t-s+\delta}\left(y-x+A(t-s+\delta)\right)
=
p_{t-s+\delta}(y-x),
\]
 Lemma \ref{lem:SBPformula} gives
\[
\tlE^{t,x,\delta,\eps}(s)
=
\int_{\RR}\clG^\eps(s,y)p_{t-s+\delta}(y-x)dy,
\]
for any $\eps>0$ and $s \in [0, t+ \delta')$. As $\delta' \in (0,\delta)$ is arbitrary, this proves the lemma.
\end{proof}

\subsection{{Moment bounds for $\clG^\eps$}} \label{ssec:Gmoments}

We will assume in this subsection that $\gamma = \gamma_1$, defined by \eqref{eq:hominit}. {In particular, Assumption \ref{assump:initdiff} holds and the results from the previous subsection apply.} Under $\gamma_1$, the coordinate sequence $\{x_i\}_{i\in \ZZ}$ on $\RR_+^{\ZZ}$ satisfies $\{x_{i+1}-x_{i}\}_{i\in \ZZ}\sim \pi_1$, the law of the bi-infinite sequence of i.i.d. Exp($1$) random variables and $x_0 = 0$.
Observe that the choice of initial data implies that 
\begin{equation} \label{eq:initPMs}
N_\eps(0,x) = \begin{cases}
\Npe[0,x] & \text{ if } x \geq 0, \\
-\Nme[0,-x) & \text{ if } x < 0,
\end{cases}
\end{equation}
where $\Npe(dx)$ and $\Nme(dx)$ are two independent  Poisson random measures on $\half$, each with intensity $\eps^{-1/2} dx$. 

We begin with preliminary results on uniform moment control at time $0$ for the process in \eqref{eq:te-def}, and boundedness of certain normalized  moments of $\clG^\eps$, that will be used in our moment estimate.


\begin{lemma}\label{lem:zerobd}
Fix $r \in [1, \infty)$. 
There exists $a \in \ohalf$ such that, for all $T<\infty$,
\begin{equation}\label{eq:zerobd}
\sup_{\delta \in (0,1)} \sup_{\eps \in (0,1]} \sup_{t \in [0,T]} \sup_{x \in \RR} e^{-a|x|}\left\|\tlE^{t,x,\delta,\eps}(0)\right\|_r < \infty.
\end{equation}
\end{lemma}
\begin{lemma}\label{lem:zerobdN}
Fix $r \in [1, \infty)$. 
There exists $a \in \ohalf$ such that, for all $t<\infty$,
and any $\eps \in (0,1]$, 
\begin{equation}\label{fin}
    \sup_{x \in \RR} e^{-a|x|}\|\clG^\eps(t,x)\|_r < \infty.
\end{equation}
\end{lemma}
\begin{remark} \label{rem:cond-momentbd} \normalfont
The assumption that $\gamma = \gamma_1$ is used to prove Lemmas \ref{lem:zerobd} and \ref{lem:zerobdN}. The remaining results of this subsection are consequences of these two lemmas and do not make a direct use of this assumption on the initial condition. Thus a step in extending the main result of this work to more general initial conditions will be to prove these lemmas for such initial conditions.
\end{remark}

We now present a key moment bound.

\begin{lemma} \label{lem:Gmoments}
Fix  $r \in [1, \infty)$. Then, there exist  $c_1 \in \ohalf$, and for each $T<\infty$,  $c_2=c_2(r,T) \in \ohalf$, such that for all $\eps \in (0,1]$, and $(t, x) \in [0,T] \times \RR$, 
\[
    \EE[\clG^\eps(t,x)^r] \leq c_2 e^{c_1|x|}.
\]
\end{lemma}

We will first prove Lemma \ref{lem:Gmoments} assuming Lemmas \ref{lem:zerobd} and \ref{lem:zerobdN}. {Before giving the proof, we briefly describe the main idea. The argument exploits the martingale representation \eqref{eq:mart_rep}, the Burkholder--Davis--Gundy inequality, and the summation by parts identity in Lemma \ref{lem:SBPformula} to bound the $r$-th moment of the mollified field $\tlE^{t,x,\delta,\eps}$ in terms of its initial time moment and an integral involving the $r$-th moment of the Hopf--Cole field $\clG^\eps$, the Gaussian heat kernel and its derivative. This, in turn, gives an integral inequality for the squared $r$-th moment of $\clG^\eps$ (see \eqref{mom4}), from which the result follows by an application of the singular generalized Gr\"onwall Lemma. We highlight here that to obtain \eqref{mom4} we have to tune $\delta = \eps^{1/2}$. 
The additional parameter $\delta$ in \eqref{eq:Edef} and \eqref{eq:te-def} is needed not just for well-definedness of associated sums and integrals,
but also to carefully control the mollification as required by our results.
}

\begin{proof}[Proof of Lemma \ref{lem:Gmoments}]
We will prove the bound for $r \ge 2$. The general bound will then follow from Jensen's inequality. In the proof, $C,C'$ will denote generic positive constants depending only on $r$.
Fix $(t,x) \in \half \times \RR$ and $\delta>0$.
From \eqref{eq:te-def} and  \eqref{eq:mart_rep},
\begin{align*}
&\tlE^{t,x,\delta,\eps}(t)
=
\tlE^{t,x,\delta,\eps}(0)\\
&\quad -
(1 - e^{-\sigma\eps^{1/4}}) \sum_{i\in\ZZ}
e^{\sigma\eps^{1/4} i -\sigma \eps^{-1/4}x + \frac{1}{2}\sigma^2\eps^{-1/2}(t+\delta)}
\int_0^t p_{t-s+\delta}(X_i^\eps(s)-x+\sigma\eps^{-1/4}(t+\delta))
\, dB_i^\eps(s).
\end{align*}
Using the Burkholder-Davis-Gundy inequality and the martingale property in Lemma \ref{lem:Eprocess}(ii), we have
\begin{align}
&\left\|\tlE^{t,x,\delta,\eps}(t)\right\|_r
\le
\left\|\tlE^{t,x,\delta,\eps}(0)\right\|_r
\notag\\
&\quad
+
C\,\left\|\int_0^t
(1 - e^{-\sigma\eps^{1/4}})^2\sum_{i\in\ZZ}
e^{2\sigma\eps^{1/4} i -2\sigma \eps^{-1/4}x + \sigma^2\eps^{-1/2}(t+\delta)}
p_{t-s+\delta}^2
(X_i^\eps(s)-x+\sigma\eps^{-1/4}(t+\delta))
\, ds\right\|_{r/2}^{1/2}.
\label{mom1}
\end{align}
Let
$
Y_i^\eps(s)
=
X_i^\eps(s) + \sigma \eps^{-1/4}s.
$
Then, using the identity
\begin{multline}\label{eq:magic}
e^{-\sigma \eps^{-1/4}\left(y+\sigma\eps^{-1/4}s\right)} p_{t-s+\delta}(y-x+\sigma\eps^{-1/4}s)\\
 = e^{ -\sigma \eps^{-1/4}x + \frac{1}{2}\sigma^2\eps^{-1/2}(t-s+\delta)} p_{t-s+\delta}(y-x+\sigma\eps^{-1/4}(t+\delta))
\end{multline}
in \eqref{mom1}, we obtain
\begin{align}
&\left\|\tlE^{t,x,\delta,\eps}(t)\right\|_r
\le
\left\|\tlE^{t,x,\delta,\eps}(0)\right\|_r
\notag\\
&\quad
+
C\,\left\|\int_0^t(1 - e^{-\sigma\eps^{1/4}})^2
\sum_{i\in\ZZ}
e^{2\sigma\eps^{1/4} i -2\sigma \eps^{-1/4}Y_i^\eps(s)  + \sigma^2\eps^{-1/2}s}
p_{t-s+\delta}^2
(Y_i^\eps(s)-x)
\, ds\right\|_{r/2}^{1/2}.
\label{mom2}
\end{align}
Appealing to the summation by parts identity in Lemma \ref{lem:SBPformula} with $m=2$ and $F(s,z) = p_{t-s+\delta}^2(z-x)$, we get
\begin{align*}
   &(1 - e^{-2\sigma\eps^{1/4}})
\sum_{i\in\ZZ}
e^{2\sigma\eps^{1/4} i -2\sigma \eps^{-1/4}Y_i^\eps(s)  + \sigma^2\eps^{-1/2}s}
p_{t-s+\delta}^2
(Y_i^\eps(s)-x)
\, ds \\
&= 2\sigma \eps^{-1/4} \int_{\RR}\left(\clG^\eps(s,z)\right)^2\left(1 + \frac{\eps^{1/4}}{\sigma}\frac{z-x}{t-s+\delta}\right)p_{t-s+\delta}^2(x-z)\,dz.
\end{align*}
Using this in \eqref{mom2}, and applying Minkowski's inequality and the bound
\begin{equation}\label{eq:159mm}\sup_{\eps \in (0,1)} \Lambda(\eps) <\infty, \quad\quad \text{ where } \Lambda(\eps) := \frac{2\sigma(1- e^{-\sigma\eps^{1/4}})^2}{\eps^{1/4}(1-e^{-2\sigma \eps^{1/4}})},
\end{equation}
we have
\begin{align}
\left\|\tlE^{t,x,\delta,\eps}(t)\right\|_r
&\le
\left\|\tlE^{t,x,\delta,\eps}(0)\right\|_r
+
C'\,\left\|\int_0^t\int_{\RR}\left(\clG^\eps(s,z)\right)^2\left(1 + \frac{\eps^{1/4}}{\sigma}\frac{z-x}{t-s+\delta}\right)p_{t-s+\delta}^2(x-z)\,dz
\, ds\right\|_{r/2}^{1/2}
\notag\\
&\le
\left\|\tlE^{t,x,\delta,\eps}(0)\right\|_r
+
C'\,\left(\int_0^t\int_{\RR}\left\|\clG^\eps(s,z)\right\|_{r}^2\left(1 + \frac{\eps^{1/4}}{\sigma}\frac{|z-x|}{t-s+\delta}\right)p_{t-s+\delta}^2(x-z)\,dz
\, ds\right)^{1/2},
\end{align}
and hence,
\begin{align}
    \left\|\tlE^{t,x,\delta,\eps}(t)\right\|_r^2 \le 2 \left\|\tlE^{t,x,\delta,\eps}(0)\right\|_r^2 + 2(C')^2\,\int_0^t\int_{\RR}\left\|\clG^\eps(s,z)\right\|_{r}^2\left(1 + \frac{\eps^{1/4}}{\sigma}\frac{|z-x|}{t-s+\delta}\right)p_{t-s+\delta}^2(x-z)\,dz
\, ds.
\label{mom3}
\end{align}
To convert the above bound into an integral inequality for $\left\|\clG^\eps(s,z)\right\|_{r}^2$, amenable to Gr\"onwall's Lemma, we observe that for any $t \in \half, x \in \RR,$ and $ \delta>0$,
\begin{align*}
\clE^{t,x,\delta, \eps}(t) = \sum_{i \in \ZZ} e^{\sigma\eps^{1/4}i}\int_{X_i^\eps(t)}^\infty p_{\delta}(u - x)du \ge \sum_{i : X_i^\eps(t) \le x} e^{\sigma\eps^{1/4}i} \frac{1}{2} = \frac{1}{2}(1 - e^{-\sigma\eps^{1/4}})^{-1}e^{\sigma\eps^{1/4}N_\eps(t,x)},
\end{align*}
which gives
\begin{align*}
    \tlE^{t,x,\delta,\eps}(t) \ge \frac{1}{2}e^{-\sigma \eps^{-1/4}x + \frac{1}{2}\sigma^2\eps^{-1/2}(t+\delta)+\sigma\eps^{1/4} N_\eps(t,x - \sigma\eps^{-1/4}(t+ \delta))} = \frac{1}{2}e^{-\frac{1}{2}\sigma^2\eps^{-1/2}\delta}\clG^\eps(s,x - \sigma \eps^{-1/4}\delta).
\end{align*}
We now let $\delta = \eps^{1/2}$. Then, using the above in \eqref{mom3}, we obtain
\begin{align*}
    \left\|\clG^\eps(s,x - \sigma \eps^{1/4})\right\|_r^2 &
    \le 4e^{\sigma^2} \left\|\tlE^{t,x,\eps^{1/2},\eps}(t)\right\|_r^2\\
    &\le 8e^{\sigma^2} \left\|\tlE^{t,x,\eps^{1/2},\eps}(0)\right\|_r^2\notag\\
    & \quad + 8(C')^2e^{\sigma^2}\,\int_0^t\int_{\RR}\left\|\clG^\eps(s,z)\right\|_{r}^2\left(1 + \frac{1}{\sigma}\frac{|z-x|}{\sqrt{t-s+\delta}}\right)p_{t-s+\delta}^2(x-z)\,dz
\, ds,
\end{align*}
which gives
\begin{align}
\left\|\clG^\eps(s,x)\right\|_r^2 &\le 8e^{\sigma^2} \left\|\tlE^{t,x+\sigma\eps^{1/4},\delta,\eps}(0)\right\|_r^2\notag\\
    & \quad + 8C'^2e^{\sigma^2}\,\int_0^t\int_{\RR}\left\|\clG^\eps(s,z)\right\|_{r}^2\left(1 + \frac{1}{\sigma}\frac{|z-x- \sigma \eps^{1/4}|}{\sqrt{t-s+\delta}}\right)p_{t-s+\delta}^2(x+ \sigma \eps^{1/4}-z)\,dz
\, ds.
 \label{mom4}   
\end{align}
Recall the constant $a$ from Lemmas \ref{lem:zerobd} and \ref{lem:zerobdN}. Consider the minimum of these two values of $a$ and abusing notation, denote it once more as $a$. Define
\begin{equation*}
  F^\eps(t,x) := e^{-2a|x|} \left\|\clG^\eps(t,x)\right\|_r^2, \qquad (t,x) \in [0,T] \times \RR.
\end{equation*}
Then from \eqref{eq:zerobd} and \eqref{mom4}, we obtain constants $C_1$ depending on $r,T$ and $C_2$ depending only on $r$ such that
\begin{align*}
    F^\eps(t,x) \le C_1 + C_2\,\int_0^t\int_{\RR}F^\eps(s,z)e^{-2a(|x| - |z|)}\left(1 + \frac{1}{\sigma}\frac{|z-x- \sigma \eps^{1/4}|}{\sqrt{t-s+\delta}}\right)p_{t-s+\delta}^2(x+ \sigma \eps^{1/4}-z)\,dz
\, ds.
\end{align*}
From standard Gaussian calculations,
$$
\sup_{z \in \RR} e^{-2a(|x| - |z|)}\left(1 + \frac{1}{\sigma}\frac{|z-x- \sigma \eps^{1/4}|}{\sqrt{t-s+\delta}}\right)p_{t-s+\delta}(x+ \sigma \eps^{1/4}-z) \le \frac{C_3}{\sqrt{t-s + \delta}},
$$
where $C_3$ is a positive constant depending only on $T$. Thus, writing $C_4:= C_2C_3$,
\begin{equation*}
    F^\eps(t,x) \le C_1 + C_4\,\int_0^t \int_{\RR} \frac{F^\eps(s,z)}{\sqrt{t-s}}p_{t-s+\delta}(x+ \sigma \eps^{1/4}-z)\,dz\,ds.
\end{equation*}
Define
$$
F^{\eps}(t) := \sup_{x \in \RR} F^\eps(t,x), \qquad t \in [0,T],
$$
which is finite by Lemma \ref{lem:zerobdN}. Then, we have from the above integral inequality,
\begin{equation*}
    F^{\eps}(t) \le C_1 + C_4\,\int_0^t \frac{F^\eps(s)}{\sqrt{t-s}}\,ds.
\end{equation*}
The result now follows from the singular generalized Gr\"onwall's lemma.
\end{proof}

\begin{proof}[Proof of Lemma \ref{lem:zerobd}]
In view of Lemma \ref{lem:Erep} and Minkowski's Inequality, 
\begin{equation} \label{eq:tE0-Lrbd1}
\| \tlE^{t,x,\delta,\eps}(0) \|_r \leq \int_{\RR} \|\clG^\eps(0,u)\|_r p_{t + \delta}(u - x)du.
\end{equation}
Using the moment generating function formula for  a Poisson distribution and \eqref{eq:initPMs}, we have for all $u \geq 0$,
\[
\EE[\clG^\eps(0,u)^r] = e^{-r\sigma\eps^{-1/4}u} \EE[e^{r\sigma\eps^{1/4}\clN_\eps^+[0,u]}] = \exp(-r\sigma\eps^{-1/4}u + \eps^{-1/2}u(e^{r\sigma\eps^{1/4}} - 1)).
\]
Applying the bound $e^z - 1 \leq z + \frac{1}{2}z^2 + \frac{1}{6}z^3 e^z$, for all $z \geq 0$, we obtain for some $c_1 = c_1(r)$, 
\[
\EE[\clG^\eps(0,u)^r]^{1/r} \leq e^{\frac{1}{2}r\sigma^2 u + c_1\eps^{1/4} u}.
\]
If on the other hand, $u < 0$, then  
\[
\EE[\clG^\eps(0,u)^r] = e^{-r\sigma\eps^{-1/4}u} \EE[e^{-r\sigma\eps^{1/4}\clN_\eps^-[0,-u)}] = \exp(-r\sigma\eps^{-1/4}u - \eps^{-1/2}u(e^{-r\sigma\eps^{1/4}} - 1)),
\]
and we similarly obtain $\EE[\clG^\eps(0,u)^r]^{1/r} \leq e^{-\frac{1}{2}r\sigma^2 u - c_1\eps^{1/4} u}$. We conclude that for all $u \in \RR$, 
\begin{equation} \label{eq:G0bd}
\EE[\clG^\eps(0,u)^r]^{1/r} \leq e^{(\frac{1}{2}r\sigma^2 + c_1\eps^{1/4})|u|} \leq  e^{(\frac{1}{2}r\sigma^2 + c_1\eps^{1/4})u} +  e^{-(\frac{1}{2}r\sigma^2 + c_1\eps^{1/4})u}.
\end{equation}
Using this bound in \eqref{eq:tE0-Lrbd1} and computing the resulting Gaussian integrals, we obtain 
\[
\| \tlE^{t,x,\delta,\eps}(0) \|_r \leq c_2 e^{(2^{-1}r\sigma^2 + c_1 \eps^{1/4})|x| + \frac{1}{2}(2^{-1}r\sigma^2 + c_1 \eps^{1/4})^2(t + \delta) }.
\]
This bound in turn implies \eqref{eq:zerobd}.
\end{proof}

\begin{proof}[Proof of Lemma \ref{lem:zerobdN}]


{ Fix $r \in [1,\infty)$, $\eps \in (0,1]$, and $t \in (0,T]$. Observe that for some $c_1 = c_1(\eps,T,r) \in \ohalf$, 
\begin{equation} \label{eq:Grbd}
\EE |\clG^\eps(t,x)|^r \leq c_1 \EE e^{r\sigma\eps^{1/4}N_\eps(t,x - \sigma \eps^{-1/4}t) -r\sigma\eps^{-1/4}x}.
\end{equation}
To bound the expectation on the right-hand side, fix $u \in \ohalf$, set 
$$
J_x = r\sigma\eps^{1/4}N_\eps(t,x - \sigma \eps^{-1/4}t) -r\sigma\eps^{-1/4}x, \quad\quad m = \left\lfloor \frac{u + r\sigma\eps^{-1/4}x}{r\sigma\eps^{1/4}} \right\rfloor,
$$
and observe that, by Lemma \ref{lem:plbXj} and comments below Assumption \ref{assump:initdiff}, with $c_0(\xi) = \log(1+\xi)$, for all $\xi \in \ohalf$, and for some $c_2 = c_2(\eps,T,\xi) \in \ohalf$,
\[
\begin{split}
\PP(J_x > u) & \leq \PP(X_m^\eps(t) \leq x - \sigma \eps^{-1/4}t) = \PP(X_m(\eps^{-1}t) \leq \eps^{-1/2}(x - \sigma \eps^{-1/4}t)) \\
& \leq c_2 e^{\eps^{-1/2}(x - \sigma\eps^{-1/4}t)\xi - c_0(\xi)m}.
\end{split}
\]
Hence, we obtain the bound, for all $\xi \in \ohalf$, for some $c_3 = c_3(\eps,T,\xi) \in \ohalf$,
\begin{equation} \label{eq:mainJxbd}
\begin{split}
\PP(J_x > u) & \leq c_2 \exp( \eps^{-1/2}(x - \sigma\eps^{-1/4}t)\xi - \log(1 + \xi)(\frac{u}{r\sigma\eps^{1/4}} + \eps^{-1/2}x - 1) ) \\
& \leq c_3 \exp( \eps^{-1/2}(\xi - \log(1 + \xi))x^+ - \frac{\log(1+\xi) u}{r\sigma\eps^{1/4}}),
\end{split}
\end{equation}
where the last inequality uses the fact that  $\xi \geq \log(1+\xi)$ for $\xi >0$.
Let $\beta > 1$. We choose $\xi = \beta r\sigma\eps^{1/4} e^{\beta r\sigma\eps^{1/4}}$. Then 
\[
\begin{split}
& \xi - \log(1 + \xi) \leq \frac{1}{2}\xi^2 \leq \frac{1}{2}\beta^2 r^2 \sigma^2 \eps^{1/2} e^{2\beta r \sigma}, \hspace{0.1in} \text{ and } \\
& \log(1 + \xi) = \log(1 + \beta r\sigma\eps^{1/4} e^{\beta r\sigma\eps^{1/4}}) \geq \beta r \sigma \eps^{1/4}.
\end{split}
\]
Substituting these two bounds into \eqref{eq:mainJxbd} gives us 
\begin{equation} \label{eq:posxcase}
\PP(J_x > u) \leq c_3 \exp( \frac{1}{2}\beta^2 r^2 \sigma^2 e^{2\beta r \sigma} x^+ - \beta u ).
\end{equation}
Therefore, for some $c_4 = c_4(\eps,T,r) \in \ohalf$, recalling that $\beta > 1$,
\[
\EE e^{r\sigma\eps^{1/4}N_\eps(t,x - \sigma \eps^{-1/4}t) - r\sigma\eps^{-1/4}x} 
= \int_{\RR} e^u \PP(J_x > u) du
\leq c_4 \exp(\frac{1}{2}\beta^2 r^2 \sigma^2 e^{2\beta r\sigma} x^+).
\]
This bound, combined with \eqref{eq:Grbd}, completes the proof of the lemma, with $a = 2^{-1}\beta^2 r \sigma^2 e^{2\beta r\sigma}$.
}
\end{proof}

\subsection{{Continuity estimates for the mollified field}{ and tightness of $\{\tilde\clE_\theta^\eps\}$}} \label{ssec:conttight}

{We next prove continuity estimates which will be used to show that the sequence $\{\tlE_\theta^\eps\}_{\eps > 0}$ is tight, also demonstrating its SHE-type regularity properties. We will again assume that the initial distribution is given by $\gamma = \gamma_1$, defined in \eqref{eq:hominit}. As in the proof of Lemma \ref{lem:Gmoments}, the starting point in obtaining these estimates is the martingale representation~\eqref{eq:mart_rep}. This representation, in conjunction with the summation by parts formula in Lemma \ref{lem:SBPformula}, the moment bounds in Lemma \ref{lem:Gmoments}, and the initial time continuity bounds for the process $\tlE^{t,x,\delta,\eps}(\cdot)$ obtained below in Lemma~\ref{lem:continuity0-bd}, allows us to transfer the continuity properties of the heat kernel to the mollified field.

The initial time continuity bounds below again reflect the heat kernel regularity via the integral representation in Lemma \ref{lem:Erep}.} 


\begin{lemma} \label{lem:continuity0-bd} 
Let $\theta \in (0,1)$, and for $\eps \in (0,1]$, choose $\delta = \eps^{1/2\theta}$, so that $\tlE_\theta^\eps(t,x) = \tlE^{t,x,\delta,\eps}(t)$. Fix $T > 0$ and $r \in [1,\infty)$. 
\begin{enumerate}[label = (\roman*)]
\item There exist constants $b_1 = b_1(r,T)$ and $b_2 = b_2(\theta,r,T)$ in $\ohalf$ such that, for all $\eps \in (0,1]$, $t \in [0,T]$ and $x,x' \in \RR$, 
\begin{equation} \label{eq:E0spacediffbd}
\| \tlE^{t,x,\delta,\eps}(0) - \tlE^{t,x',\delta,\eps}(0) \|_r \leq b_1 e^{b_2(|x| + |x'|)}|x' - x|^{\theta/2}.
\end{equation}
\item There exist constants $b_3 = b_3(r,T)$ and $b_4 = b_4(\theta, r,T)$ in $\ohalf$ such that, for all $\eps \in (0,1]$, $t,t' \in [0,T]$ and $x \in \RR$, 
\begin{equation} \label{eq:E0timediffbd}
\| \tlE^{t,x,\delta,\eps}(0) - \tlE^{t',x,\delta,\eps}(0) \|_r \leq b_3 e^{b_4|x|}|t' - t|^{\theta/4}.
\end{equation}
\end{enumerate}
\end{lemma}

{The following gives the main continuity estimate for $\tlE_\theta^\eps(t,x)$. 
}

\begin{lemma} \label{lem:continuity}
Fix $r \in [1,\infty)$, $\theta \in (0,1)$, and $T \in \ohalf$.
\begin{enumerate}[label = (\roman*)]
\item There exist constants $C_1 = C_1(\theta,r,T)$ and $C_2 = C_2(\theta,r,T)$ in $\ohalf$ such that for all $\eps \in (0,1]$, $t \in [0,T]$ and $x,x' \in \RR$, 
\begin{equation} \label{eq:spaceest}
\|\tlE_\theta^\eps(t,x') - \tlE_\theta^\eps(t,x)\|_r \leq C_1e^{C_2(|x|+|x'|)}|x' - x|^{\theta/2}.
\end{equation}
\item 
There exist constants $C_3 = C_3(\theta,r,T)$ and $C_4 = C_4(\theta,r,T)$ in $\ohalf$ such that for all $\eps \in (0,1]$, $x \in \RR$ and $t,t' \in [0,T]$,
\begin{equation} \label{eq:timeest}
\| \tlE_\theta^\eps(t',x) - \tlE_\theta^\eps(t,x) \|_r \leq C_3 e^{C_4|x|}|t' - t|^{\theta/4}.
\end{equation}
\end{enumerate}

\end{lemma}

\begin{remark}\label{rem:Ereg} \normalfont
The exponents in Lemma~\ref{lem:continuity} are the microscopic analogue of
the regularity of the limiting stochastic heat equation.  The SHE driven by
space-time white noise, and with our Brownian-type initial condition, has spatial H\"older regularity strictly below
\cite[Exercise 6.9]{Khoshnevisan2009}.  Since $\theta\in(0,1)$ is arbitrary, the estimates
\eqref{eq:spaceest} and \eqref{eq:timeest} place
$\tlE^\eps_\theta$ uniformly in this same H\"{o}lder regularity class.  Thus
Lemma~\ref{lem:continuity} is not merely a tightness estimate but shows that
the mollified microscopic Hopf--Cole fields have exactly the a priori
regularity expected of their SHE limit. 
\end{remark}

{The estimates obtained {above and in the previous section}
culminate in the following tightness result.} 
\begin{corollary} \label{cor:tightcont}
Fix $\theta \in (0,1)$. The sequence $\{\tlE_\theta^\eps(t,x) : (t,x) \in \half \times \RR \}_{\eps > 0}$ is tight in the space $C(\half \times \RR : \RR)$.
\end{corollary}

\begin{remark}\label{rem:gencont} \normalfont
{Tightness of the collection $\{\tlE_\theta^\eps\}$ will continue to hold if, instead of taking $\gamma = \gamma_1$, we assume that the initial distribution $\gamma$ is such that

\begin{enumerate}[label= D\arabic*]
\item The  bound \eqref{eq:ainit2} and Assumption \ref{assump:initdiff} hold; \label{it:weaker1}
\item  \label{it:weaker2}  Lemmas \ref{lem:zerobd} and \ref{lem:zerobdN} hold;
\item \label{it:weaker3} Lemma \ref{lem:continuity0-bd} holds.
\end{enumerate}

\noindent This observation follows from the proof of Lemma \ref{lem:continuity}, which does not use the initial distribution $\gamma_1$, except through the estimates from Lemmas \ref{lem:zerobd}, \ref{lem:zerobdN}, and \ref{lem:continuity0-bd}, and noting that the proof of Corollary \ref{cor:tightcont} is a direct consequence of Lemma \ref{lem:continuity} and the Kolmogorov-Chentsov criterion.}
\end{remark}

\subsubsection{{Proofs of continuity estimates}}

\begin{proof}[Proof of Lemma \ref{lem:continuity0-bd}] 
\textit{Proof of \eqref{eq:E0spacediffbd}.} Fix $t \in [0,T]$, and $x,x' \in \RR$. We may assume without loss of generality $x' > x$. Let $\delta = \eps^{1/2\theta}$. By Lemma \ref{lem:Erep}, 
\begin{equation} \label{eq:Erep-ctbd}
\tlE^{t,x',\delta,\eps}(0) - \tlE^{t,x,\delta,\eps}(0) = \int_{\RR} \clG^\eps(0,u)[p_{t + \delta}(u - x') - p_{t + \delta}(u - x)]du.
\end{equation}
We will obtain two different estimates for the above quantity in $L^r(\PP)$. By Lemma \ref{lem:Gmoments},
\[
\begin{split}
& \| \tlE^{t,x',\delta,\eps}(0) - \tlE^{t,x,\delta,\eps}(0) \|_r \leq \int_{\RR} \| \clG^\eps(0,u) \|_r |p_{t + \delta}(u - x') - p_{t + \delta}(u - x)|du \\
& \leq \int_{\RR} c_1 e^{c_2|u|} |p_{t + \delta}(u - x') - p_{t + \delta}(u - x)|du.
\end{split}
\]
Note that 
\[
|p_{t + \delta}(u - x') - p_{t + \delta}(u - x)|  = \left| \int_x^{x'} \partial_v p_{t+\delta}(u - v)dv \right| \leq \int_x^{x'} \frac{|u - v|}{t+\delta}p_{t+\delta}(u - v) dv.
\]
Hence,  
\begin{align*}
& \int_{\RR} c_1 e^{c_2|u|} |p_{t + \delta}(u - x') - p_{t + \delta}(u - x)|du \leq \int_x^{x'} \int_{\RR} c_1e^{c_2|u|} \frac{|u - v|}{t+\delta}p_{t+\delta}(u - v) du dv \\
& \leq \frac{c_1}{t + \delta}\int_x^{x'} e^{c_2|v|} \int_{\RR} (e^{c_2|w|}) |w|p_{t+\delta}(w) dw dv \\ 
& \leq \frac{c_1}{t + \delta}\int_x^{x'} e^{c_2|v|} \int_{\RR} (e^{c_2w} + e^{-c_2w}) |w|p_{t+\delta}(w) dw dv \\
& = \frac{c_1}{t + \delta}\int_x^{x'} e^{c_2|v|} \int_{\RR} e^{c_2^2(t+\delta)/2} |w| [ p_{t +\delta}(w + c_2(t + \delta)) + p_{t +\delta}(w - c_2(t + \delta))]dw dv \\
& \leq \frac{c_3}{t+\delta} \int_x^{x'} e^{c_2|v|} ((t+\delta)^{1/2} + (t+\delta))dv \\
& \leq c_4 e^{c_2(|x'| + |x|)} \delta^{-1/2} |x' - x|.
\end{align*}
We obtain the estimate
\begin{equation} \label{eq:diffest1}
\| \tlE^{t,x',\delta,\eps}(0) - \tlE^{t,x,\delta,\eps}(0) \|_r \leq c_4 e^{c_2(|x'| + |x|)} \eps^{-\frac{1}{4\theta}} |x' - x|.
\end{equation}

On the other hand, by changing variables in \eqref{eq:Erep-ctbd},  
\[
\tlE^{t,x',\delta,\eps}(0) - \tlE^{t,x,\delta,\eps}(0) = \int_{\RR} [\clG^\eps(0,u + x') - \clG^\eps(0,u+x)]p_{t + \delta}(u) du.
\]
Hence
\begin{equation} \label{eq:E0xdiff}
\| \tlE^{t,x',\delta,\eps}(0) - \tlE^{t,x,\delta,\eps}(0) \|_r \leq \int_{\RR} \| \clG^\eps(0,u + x') - \clG^\eps(0,u+x) \|_r p_{t + \delta}(u) du.
\end{equation}
Fix $u,v \in \RR$ with $v > u$, and observe that 
\begin{equation} \label{eq:therbpoiss}
\clG^\eps(0,v) - \clG^\eps(0,u) = \clG^\eps(0,u) \left\{ e^{ \sigma\eps^{1/4} (N_\eps(0,v) - N_\eps(0,u)) - \sigma\eps^{-1/4}(v - u)} - 1 \right\}.
\end{equation}

\begin{claim}{1} 
For every $q\ge 1$, there is a $C(q)<\infty$, not depending on $u$ and $v$, such that
\begin{equation}\label{eq:517nn+}
\|e^{\sigma \eps^{1/4} (N^\varepsilon(v)- N^\varepsilon(u)) - \sigma \eps^{-1/4}(v-u)}-1\|_q
\le C(q)e^{C(q)(|u|+|v|)}( \eps^{\frac{1}{4}} + \eps^{\frac{1}{4} - \frac{1}{4q}}|v - u|^{\frac{1}{2q}} + |v - u|^{\frac{1}{2}} ).
\end{equation}
\end{claim}

\textit{Proof of claim.} To see this, note that $N_\eps(0,v) - N_\eps(0,u) = N^{\eps}_{u,v}+\mathbf{1}[uv <0]$, where
$N^{\eps}_{u,v} \sim \mbox{Poi}(\eps^{-1/2}(v-u))$ and consequently, for some constant $C>0$,
\begin{equation}\label{eq:515n+}
|e^{\sigma \eps^{1/4} (N^\varepsilon(v)- N^\varepsilon(u)) - \sigma \eps^{-1/4}(v-u)}-1| \le
C\left(|e^{\sigma \eps^{1/4} N_{u,v}^\varepsilon - \sigma \eps^{-1/4}(v-u)}-1| + \eps^{1/4}\right).
\end{equation} 
Denote by $C_i(q), i \ge 1,$ positive constants depending only on $q$. If $R\sim \mbox{Poi}(\lambda)$ and $a>0$, then, since $|e^x-1| \le |x|(e^x+1)$ for $x\in \RR$,
\begin{align*}
\EE|e^{aR-a\lambda}-1|^q &\le a^q \EE[|R-\lambda|^q(e^{aR-a\lambda}+1)^q]\\
&\le C_1(q) a^q\left(\EE|R-\lambda|^{2q}\right)^{1/2} \left(\EE(e^{2qa(R-\lambda)}+1)\right)^{1/2}\\
&\le C_2(q) a^q\left(\lambda + \lambda^q\right)^{1/2}\left( e^{\lambda(e^{2qa}-1-2qa)} +1)\right)^{1/2}.
\end{align*}
Taking $a=\sigma \eps^{1/4}$ and $\lambda = \eps^{-1/2}(v-u)$ in the above,
$$
\EE\left|e^{\sigma \eps^{1/4} N_{u,v}^\varepsilon - \sigma \eps^{-1/4}(v-u)}-1\right|^q
\le C_3(q) e^{C_4(q)(|u|+|v|)}\left( \eps^{(q - 1)/4}|v - u|^{1/2} + |v - u|^{q/2} \right).
$$
Using this bound in \eqref{eq:515n+} proves the claim in \eqref{eq:517nn+}. 

In view of equation \eqref{eq:therbpoiss}, Lemma \ref{lem:Gmoments}, and \eqref{eq:517nn+} with $q = 2r$ we obtain 
\begin{equation} \label{eq:G0diffbd}
\begin{split}
\| \clG^\eps(0,v) - \clG^\eps(0,u) \|_r & \leq \|\clG^\eps(0,u)\|_{2r} \left\| e^{ \sigma\eps^{1/4} (N_\eps(0,v) - N_\eps(0,u)) - \sigma\eps^{-1/4}(v - u)} - 1 \right\|_{2r} \\
& \leq c_5 e^{c_6 (|u| + |v|)}\left( \eps^{\frac{1}{4}} + \eps^{\frac{1}{4} - \frac{1}{8r}}|v - u|^{\frac{1}{4r}} + |v - u|^{\frac{1}{2}} \right).
\end{split}
\end{equation}
Applying this bound to \eqref{eq:E0xdiff}, we obtain the estimate
\begin{equation} \label{eq:diffest2}
\begin{split}
& \| \tlE^{t,x',\delta,\eps}(0) - \tlE^{t,x,\delta,\eps}(0) \|_r \\
& \leq c_5  e^{c_6(|x| + |x'|)} \left( \int_{\RR} e^{c_6|u|} p_{t + \delta}(u)du \right) \left( \eps^{\frac{1}{4}} + \eps^{\frac{1}{4} - \frac{1}{8r}}|x' - x|^{\frac{1}{4r}} + |x' - x|^{\frac{1}{2}} \right) \\
& \leq c_7 e^{c_6(|x| + |x'|)} \left( \eps^{\frac{1}{4}} + \eps^{\frac{1}{4} - \frac{1}{8r}}|x' - x|^{\frac{1}{4r}} + |x' - x|^{\frac{1}{2}} \right).
\end{split}
\end{equation}
Combining the estimates \eqref{eq:diffest1} and \eqref{eq:diffest2} and setting $k = |x' - x|$, we arrive at the bound 
\begin{equation*}
\begin{split}
\| \tlE^{t,x',\delta,\eps}(0) - \tlE^{t,x,\delta,\eps}(0) \|_r \leq c_8 e^{c_9(|x| + |x'|)} \min\left\{ \eps^{-\frac{1}{4\theta}} k, \eps^{\frac{1}{4}} + \eps^{\frac{1}{4} - \frac{1}{8r}}k^{\frac{1}{4r}} + k^{\frac{1}{2}} \right\} .
\end{split}
\end{equation*}
Define
$$
m_1(k) := \min\left\{ \eps^{-\frac{1}{4\theta}} k, \eps^{\frac{1}{4}} +  \eps^{\frac{1}{4} - \frac{1}{8r}}k^{\frac{1}{4r}} + k^{\frac{1}{2}} \right\}.
$$
We will be done with proving \eqref{eq:E0spacediffbd} if we can show that for some constant $\tilde c = \tilde c(\theta)$,
\[
 m_1(k) \leq e^{\tilde c(|x|+|x'|)}k^{\theta/2}, \quad x,x' \in \RR.
\]
Consider two cases. First, if $k \leq \eps^{\frac{1}{2\theta(2 - \theta)}}$, then 
\[
m_1(k) \leq \eps^{-\frac{1}{4\theta}} k \leq k^{-\frac{2-\theta}{2}}k = k^{\frac{\theta}{2}}.
\]
Second, if $k > \eps^{\frac{1}{2\theta(2 - \theta)}}$, then 
\begin{equation} \label{eq:minbd}
m_1(k) \leq \eps^{\frac{1}{4}} + \eps^{\frac{1}{4} - \frac{1}{8r}}k^{\frac{1}{4r}} + k^{\frac{1}{2}} \leq k^{\frac{\theta(2 - \theta)}{2}} + k^{\frac{\theta(2 - \theta)}{2}(1 - \frac{1}{2r}) + \frac{1}{4r}} + k^{\frac{1}{2}}.
\end{equation}
Notice that, as $\theta \in (0,1)$, $\frac{\theta(2 - \theta)}{2} > \frac{\theta}{2}$ and
$
\frac{\theta(2 - \theta)}{2}(1 - \frac{1}{2r}) + \frac{1}{4r} \geq \min\left\{ \frac{\theta(2 - \theta)}{2}, \frac{1}{2} \right\} > \frac{\theta}{2}.
$
Hence, if $k \leq 1$, then $m_1(k)$ is bounded above by $3k^{\theta/2}$. Furthermore, if $k > 1$, then $m_1(k) \leq e^{\tilde c(|x| + |x'|)} k^{\theta/2}$ for suitably large $\tilde c = \tilde c(\theta)$, noting the polynomial growth of $m_1(k)$ in $k = |x' - x|$. This completes the proof of \eqref{eq:E0spacediffbd}.

\textit{Proof of \eqref{eq:E0timediffbd}.} Fix $t,t' \in [0,T]$ and $x \in \RR$, and assume without loss of generality that $t' \ge t$. We will estimate $\tlE^{t',x,\delta,\eps}(0) - \tlE^{t,x,\delta,\eps}(0)$ in two different ways. Using Lemma \ref{lem:Erep}, 
\[
\begin{split}
\tlE^{t',x,\delta,\eps}(0) - \tlE^{t,x,\delta,\eps}(0) & = \int_{\RR} \clG^\eps(0,u)[p_{t'+\delta}(u - x) - p_{t + \delta}(u - x)]du \\
& = \int_t^{t'}\int_{\RR} \clG^\eps(0,u)\partial_s p_{s + \delta}(u - x)du ds.
\end{split}
\]
Hence, by Minkowski's inequality and Lemma \ref{lem:Gmoments},
\begin{equation} \label{eq:difftoderiv}
\| \tlE^{t',x,\delta,\eps}(0) - \tlE^{t,x,\delta,\eps}(0) \|_r \leq \int_t^{t'} \int_{\RR} c_{10}e^{c_{11}|u|}|\partial_s p_{s + \delta}(u - x)|du ds.
\end{equation}
Note that 
\[
|\partial_s p_{s + \delta}(u - x)| \leq \left( \frac{|u - x|^2}{2(s + \delta)^2} + \frac{1}{2(s + \delta)} \right)p_{s + \delta}(u - x).
\]
Furthermore
\[
\begin{split}
e^{c_{11}|u|}p_{s + \delta}(u - x) & \leq e^{c_{11}|x|}\left(e^{c_{11}(u - x)} + e^{-c_{11}(u - x)} \right)p_{s + \delta}(u - x) \\
& \leq  e^{c_{11}|x| + \frac{c_{11}^2(s+\delta)}{2}}\left[p_{s + \delta}(u - x + c_{11}(s + \delta)) + p_{s + \delta}(u - x - c_{11}(s + \delta)) \right].
\end{split}
\]
Applying these to bounds to \eqref{eq:difftoderiv} yields the estimate
\begin{equation} \label{eq:E0diffbd1}
\begin{split}
& \| \tlE^{t',x,\delta,\eps}(0) - \tlE^{t,x,\delta,\eps}(0) \|_r \\
& \leq c_{12}e^{c_{11}|x|} \int_t^{t'} \int_{\RR} \left( \frac{|u - x|^2}{2(s + \delta)^2} + \frac{1}{2(s + \delta)} \right) \\
& \hspace{1.3in} \times \left[p_{s + \delta}(u - x + c_{11}(s + \delta)) + p_{s + \delta}(u - x - c_{11}(s + \delta)) \right] du ds \\
& \leq c_{12}e^{c_{11}|x|} \int_t^{t'} \left\{ \frac{2}{(s + \delta)^2}\int_{\RR} \left(|w|^2 + c_{11}^2(s + \delta)^2\right) p_{s + \delta}(w)dw + \frac{1}{s+\delta} \right\}ds \\
& \leq c_{13}e^{c_{11}|x|} \int_t^{t'} \left\{ 1 + \frac{1}{s + \delta} \right\}ds \leq c_{13}e^{c_{11}|x|} \delta^{-1}|t' - t|.
\end{split}
\end{equation}

On the other hand, using Lemma \ref{lem:Erep} again, with $h = t' - t$, we may write 
\[
\tlE^{t',x,\delta,\eps}(0) - \tlE^{t,x,\delta,\eps}(0) = \int_{\RR} p_{t + \delta}(u - x) \int_{\RR} p_h(z)[ \clG^\eps(0,u+z) - \clG^\eps(0,u) ]dz du.
\]
Then by Minkowski's inequality and the bound \eqref{eq:G0diffbd}, for $r \in [1,\infty)$,
\begin{equation}\label{eq:555n+}
\begin{aligned}
& \| \tlE^{t',x,\delta,\eps}(0) - \tlE^{t,x,\delta,\eps}(0) \|_r
\le \int_{\RR} p_{t+\delta}(u-x) \int_{\RR} p_h(z)\|\clG^{\eps}(0,u+z)- \clG^{\eps}(0,u)\|_r dz \, du\\
&\le c_5 \left(\int_{\RR} p_{t+\delta}(u-x) du\right) e^{c_{14}|u|} \int_{\RR} p_h(z) e^{c_{14}|z|}(\eps^{\frac{1}{4}} + \eps^{\frac{1}{4} - \frac{1}{8r}}|z|^{\frac{1}{4r}} + |z|^{\frac{1}{2}}) dz.
\end{aligned}
\end{equation}
Noting that 
$$p_h(z)e^{c_{14}|z|} \leq (e^{c_{14}z} + e^{-c_{14}z})p_h(z) \leq c_{15}[p_h(z + c_{14}h) + p_h(z - c_{14}h)],
$$
the last line in \eqref{eq:555n+} is bounded by 
\begin{equation} \label{eq:E0diffbd2}
\begin{split}
& c_{16} e^{c_{14}|x|} \int_{\RR} p_h(z)\left(\eps^{\frac{1}{4}} + \eps^{\frac{1}{4} - \frac{1}{8r}}(|z|^{\frac{1}{4r}} + h^{\frac{1}{4r}}) + (|z|^{\frac{1}{2}} + h^{\frac{1}{2}})\right)dz \\
& \leq c_{17} e^{c_{14}|x|} \left(\eps^{\frac{1}{4}} + \eps^{\frac{1}{4} - \frac{1}{8r}}(h^{\frac{1}{8r}} + h^{\frac{1}{4r}}) + (h^{\frac{1}{4}} + h^{\frac{1}{2}})\right) \leq c_{18}e^{c_{14}|x|} \left(\eps^{\frac{1}{4}} + \eps^{\frac{1}{4} - \frac{1}{8r}}h^{\frac{1}{8r}} + h^{\frac{1}{4}}\right).
\end{split}
\end{equation}

Combining the bounds \eqref{eq:E0diffbd1} and \eqref{eq:E0diffbd2}, we obtain 
\[
\| \tlE^{t',x,\delta,\eps}(0) - \tlE^{t,x,\delta,\eps}(0) \|_r \leq c_{19}e^{c_{20}|x|} \min\left\{ \eps^{-\frac{1}{2\theta}}h, (\eps^{\frac{1}{4}} + \eps^{\frac{1}{4} - \frac{1}{8r}}h^{\frac{1}{8r}} + h^{\frac{1}{4}}) \right\}.
\]
Denote the minimum on the right-hand side by $m_2(h)$. We now consider two cases: If $h \le \eps^{\frac{2}{\theta(4-\theta)}}$, then
$$
m_2(h) = \min\left\{ \eps^{-\frac{1}{2\theta}}h, (\eps^{\frac{1}{4}} + \eps^{\frac{1}{4} - \frac{1}{8r}}h^{\frac{1}{8r}} + h^{\frac{1}{4}} ) \right\} \le \eps^{-\frac{1}{2\theta}}h \le
h^{-1 + \frac{\theta}{4}} h = h^{\frac{\theta}{4}}.
$$
On the other hand, if $h > \eps^{\frac{2}{\theta(4-\theta)}}$, then
\[
\begin{split}
& m_2(h) \le \eps^{\frac{1}{4}} + \eps^{\frac{1}{4} - \frac{1}{8r}}h^{\frac{1}{8r}} + h^{\frac{1}{4}} \le h^{\frac{\theta(4-\theta)}{8}} +  h^{(\frac{\theta(4 - \theta)}{8})(1 - \frac{1}{2r}) + \frac{1}{8r}} + h^{\frac{1}{4}} \leq c_{21}h^{\frac{\theta}{4}},
\end{split}
\]
where the last line follows by noting that 
$$
\frac{\theta(4-\theta)}{8} > \frac{\theta}{4},
\quad \text{ and } \quad 
(\frac{\theta(4 - \theta)}{8})(1 - \frac{1}{2r}) + \frac{1}{8r} \geq \min\left\{ \frac{\theta(4 - \theta)}{8}, \frac{1}{4} \right\} > \frac{\theta}{4}.
$$ 
Combining the two cases, we conclude that, for all $x \in \RR$ and $t',t \in [0,T]$,
$$
\|\tlE^{t',x,\delta,\eps}(0) - \tlE^{t,x,\delta,\eps}(0)\|_r \le c_{22}e^{c_{20}|x|} |t'-t|^{\frac{\theta}{4}}.
$$
This completes the proof of \eqref{eq:E0timediffbd}.
\end{proof}

\begin{proof}[Proof of Lemma \ref{lem:continuity}.i]
Fix $r,T, \theta$ as in the statement of the lemma. For $\delta \in (0,1)$, $t \in [0,T]$ and $x, x' \in \RR$, by  Lemma \ref{lem:Eprocess}(ii), \eqref{eq:te-def}, triangle inequality and Burkholder-Gundy inequality, for some $C_1<\infty$ (depending only on $r$), with $\tau(s)=t-s+\delta$ for $s \in [0,t]$,
\begin{multline}
\|\tlE^{t,x,\delta, \eps}(t) - \tlE^{t,x',\delta, \eps}(t)\|_r 
\le \|\tlE^{t,x,\delta, \eps}(0) - \tlE^{t,x',\delta, \eps}(0)\|_r \\
 + C_1 \Big\| \int_0^t e^{\sigma^2 \eps^{-1/2}(t+\delta)}
(1 - e^{-\sigma\eps^{1/4}})^2\\
\times \sum_{i\in \ZZ}\left(e^{\sigma \eps^{1/4}i - \sigma\eps^{-1/4}x}
p_{\tau(s)}(X_i^{\eps,x,t}(s)) - e^{\sigma \eps^{1/4}i - \sigma\eps^{-1/4}x'}
p_{\tau(s)}(X_i^{\eps,x',t}(s))\right)^2 ds\Big\|^{1/2}_{r/2}, \label{eq:136mm}
\end{multline}
where, for $y \in \RR$,
$$X_i^{\eps,y,t}(s) = X_i^{\eps}(s) - y + \sigma\eps^{-1/4}(t+\delta).$$ 
Write
\begin{multline*}
\left(e^{\sigma \eps^{1/4}i - \sigma\eps^{-1/4}x}
p_{\tau(s)}(X_i^{\eps,x,t}(s)) - e^{\sigma \eps^{1/4}i - \sigma\eps^{-1/4}x'}
p_{\tau(s)}(X_i^{\eps,x',t}(s))\right)^2\\
= e^{2\sigma \eps^{1/4}i - 2\sigma\eps^{-1/4}x}
p^2_{\tau(s)}(X_i^{\eps,x,t}(s)) + e^{2\sigma \eps^{1/4}i - 2\sigma\eps^{-1/4}x'}
p^2_{\tau(s)}(X_i^{\eps,x',t}(s))\\
- 2  e^{2\sigma \eps^{1/4}i - \sigma\eps^{-1/4}(x+x')}p_{\tau(s)}(X_i^{\eps,x,t}(s))p_{\tau(s)}(X_i^{\eps,x',t}(s))\\
=: T(i,s,x) + T(i,s,x') -2 \tilde T(i,s, x, x').
\end{multline*}
From similar calculations to those leading to \eqref{mom3}, and recalling \eqref{eq:159mm}, we see that, 
\begin{multline*}
 \int_0^t e^{\sigma^2 \eps^{-1/2}(t+\delta)}
(1 - e^{-\sigma\eps^{1/4}})^2 \sum_{i\in \ZZ} T(i,s,x) ds \\
= \Lambda(\eps)\int_0^t\int_{\RR} \clG^\eps(s,z)^2\left(1 + \frac{\eps^{1/4}}{\sigma}\frac{(z-x)}{\tau(s)}\right)p_{\tau(s)}^2(x-z)\,dz ds,
\end{multline*}
and the same equality holds with $x$ replaced with $x'$.
For $\tilde T$ we will apply Lemma \ref{lem:SBPformula} with $m=2$ and
$$F(s,y) = p_{\tau(s)}(x-y)p_{\tau(s)}(x'-y).$$
Note that
$\partial_yF(s,y) = \frac{(x+x'-2y)}{\tau(s)} p_{\tau(s)}(x-y)p_{\tau(s)}(x'-y)$ and so
$$2\sigma\eps^{-1/4}F(s,z) - \partial_zF(s,z) = 2\sigma\eps^{-1/4} \left(1 + \frac{\eps^{1/4}}{2\sigma}\frac{2z-x-x'}{\tau(s)}\right)p_{\tau(s)}(x-z)p_{\tau(s)}(x'-z).$$
Thus using \eqref{eq:magic}, applying Lemma \ref{lem:SBPformula}, and recalling \eqref{eq:159mm},
\begin{multline*}
 \int_0^t e^{\sigma^2 \eps^{-1/2}(t+\delta)}
(1 - e^{-\sigma\eps^{1/4}})^2 \sum_{i\in \ZZ} \tilde T(i,s,x, x')\\
= \Lambda(\eps)\int_0^t\int_{\RR} \clG^\eps(s,z)^2   \left(1 + \frac{\eps^{1/4}}{2\sigma}\frac{2z-x-x'}{\tau(s)}\right)p_{\tau(s)}(x-z)p_{\tau(s)}(x'-z) dz ds.
\end{multline*}
Using the above identities in \eqref{eq:136mm} we obtain, for some constant $C_2$, depending only on $r$,
\begin{align*}
&\|\tlE^{t,x,\delta, \eps}(t) - \tlE^{t,x',\delta, \eps}(t)\|_r\\ 
&\le \|\tlE^{t,x,\delta, \eps}(0) - \tlE^{t,x',\delta, \eps}(0)\|_r
 + C_2 \bigg\| \int_0^t \int_{\RR} \clG^\eps(s,z)^2 \bigg[
\left(1 + \frac{\eps^{1/4}}{\sigma}\frac{(z-x)}{\tau(s)}\right)p_{\tau(s)}^2(x-z)\\
&\quad+ \left(1 + \frac{\eps^{1/4}}{\sigma}\frac{(z-x')}{\tau(s)}\right)p_{\tau(s)}^2(x'-z)
- 2 \left(1 + \frac{\eps^{1/4}}{2\sigma}\frac{2z-x-x'}{\tau(s)}\right)p_{\tau(s)}(x-z)p_{\tau(s)}(x'-z) \bigg] dz ds
\bigg\|^{1/2}_{r/2}.
\end{align*}
Thus by Minkowski's inequality,
\begin{multline*}
\|\tlE^{t,x,\delta, \eps}(t) - \tlE^{t,x',\delta, \eps}(t)\|_r 
\le \|\tlE^{t,x,\delta, \eps}(0) - \tlE^{t,x',\delta, \eps}(0)\|_r \\
+ C_2 \bigg(\int_0^t \int_{\RR} \| \clG^\eps(s,z)\|_r^2 \bigg|
\left(1 + \frac{\eps^{1/4}}{\sigma}\frac{(z-x)}{\tau(s)}\right)p_{\tau(s)}^2(x-z) \\
+ \left(1 + \frac{\eps^{1/4}}{\sigma}\frac{(z-x')}{\tau(s)}\right)p_{\tau(s)}^2(x'-z)\\
- 2 \left(1 + \frac{\eps^{1/4}}{2\sigma}\frac{2z-x-x'}{\tau(s)}\right)p_{\tau(s)}(x-z)p_{\tau(s)}(x'-z) \bigg| dz ds
\bigg)^{1/2}.
\end{multline*}
Consequently, using Lemma  \ref{lem:Gmoments}, there is a $C_3$, depending on $r,T$, and $c$, depending only on $r$, such that for all $t \le T$, 
\begin{equation}\label{eq:433mm}
\|\tlE^{t,x,\delta, \eps}(t) - \tlE^{t,x',\delta, \eps}(t)\|_r^2 \le
2 \|\tlE^{t,x,\delta, \eps}(0) - \tlE^{t,x',\delta, \eps}(0)\|_r^2
+ C_3 \int_0^t  T_4(s,x,x') ds,
\end{equation}
where
\begin{multline*}
T_4(s,x,x') = \int_{\RR} e^{c|z|} \bigg|
\left(1 + \frac{\eps^{1/4}}{\sigma}\frac{(z-x)}{\tau(s)}\right)p_{\tau(s)}^2(x-z) 
+ \left(1 + \frac{\eps^{1/4}}{\sigma}\frac{(z-x')}{\tau(s)}\right)p_{\tau(s)}^2(x'-z)\\
- 2 \left(1 + \frac{\eps^{1/4}}{2\sigma}\frac{2z-x-x'}{\tau(s)}\right)p_{\tau(s)}(x-z)p_{\tau(s)}(x'-z) \bigg| dz.
\end{multline*}
We estimate
$$T_4(s,x,x') \le T_5(s,x,x') + T_6(s,x,x'),
$$
where
$$
T_5(s,x,x') = \int_{\RR} e^{c|z|} (p_{\tau(s)}(x-z)-p_{\tau(s)}(x'-z))^2 dz,
$$
and
$$
T_6(s,x,x') = \frac{\eps^{1/4}}{\sigma{\tau(s)}} \int_{\RR} e^{c|z|}\left(|z-x|p_{\tau(s)}(z-x)
+ |z-x'|p_{\tau(s)}(z-x')\right) |p_{\tau(s)}(z-x)- p_{\tau(s)}(z-x')| dz.
$$
Observe that
$$ |p_t(x) - p_t(x')| \le \frac{|x-x'|}{t} \mbox{ for all } x, x' \in \RR, t >0.$$
Moreover, as $\sqrt{t}|p_t(x)- p_t(x')|$ is uniformly bounded by $1$, $\sqrt{t}|p_t(x)- p_t(x')| \le (\sqrt{t}|p_t(x)- p_t(x')|)^{\theta}$ for all $\theta \in (0,1)$, $t\in \ohalf$ and $x, x' \in \RR$. Consequently, for any $\theta \in (0,1)$,
$$
|p_t(x) - p_t(x')| \le \frac{|x-x'|^{\theta}}{t^{(1+\theta)/2}} \mbox{ for all } x, x'\in \RR, \; t>0.
$$
Thus, for some $\kappa_3$ depending only on $T$,
\begin{align*}
T_5(s,x,x') &\le \frac{|x-x'|^{\theta}}{\tau(s)^{(1+\theta)/2}} \int e^{c|z|} (p_{\tau(s)}(x-z)+p_{\tau(s)}(x'-z)) dz\\
&\le \kappa_3 \frac{|x-x'|^{\theta}}{\tau(s)^{(1+\theta)/2}} (e^{c|x|} + e^{c|x'|}).
\end{align*}
Also, for some $\kappa_4$ depending only on $T$,
\begin{multline*}
T_6(s,x,x') \le \frac{\eps^{1/4}}{\sigma{\tau(s)}}\frac{|x-x'|^{\theta}}{\tau(s)^{(1+\theta)/2}} \int_{\RR} e^{c|z|}\left(|z-x|p_{\tau(s)}(z-x)
+ |z-x'|p_{\tau(s)}(z-x')\right)  dz\\
\le  \kappa_4\frac{\eps^{1/4}}{\sigma{\tau(s)^{1+\frac{\theta}{2}}}}|x-x'|^{\theta} (e^{c|x|} + e^{c|x'|}).
\end{multline*}
Using the above bounds on $T_5$ and $T_6$ in \eqref{eq:433mm}, and applying Lemma \ref{lem:continuity0-bd}  we obtain
\begin{align*}
&\|\tlE^{t,x,\delta, \eps}(t) - \tlE^{t,x',\delta, \eps}(t)\|_r^2 \le
2b_1^2 e^{2b_2(|x| + x'|)} |x'-x|^\theta\\
&\quad\qquad+ C_3 (\kappa_3+\kappa_4)(e^{c|x|} + e^{c|x'|}) |x-x'|^{\theta} \int_0^t  \left[ \frac{1}{\tau(s)^{(1+\theta)/2}}  
+\frac{\eps^{1/4}}{\sigma{\tau(s)^{1+\frac{\theta}{2}}}}\right] ds\\
&= 2b_1^2 e^{2b_2(|x| + x'|)} |x'-x|^\theta\\
&\quad\qquad+ C_3 (\kappa_3+\kappa_4)(e^{c|x|} + e^{c|x'|}) |x-x'|^{\theta} \left[\frac{2}{1-\theta} T^{(1-\theta)/2} +\frac{2\eps^{1/4}}{\sigma\theta \delta^{\theta/2}}\right].
\end{align*}
Finally, taking $\delta = \eps^{1/(2\theta)}$ in the above display we complete the proof of part (i) of the lemma.
\end{proof}

\begin{proof}[Proof of Lemma \ref{lem:continuity}.ii]
Throughout this proof, the constants $c, c'$ are allowed to change value from line to line and are allowed to depend on $r$. If the constants depend on additional parameters such as $T$ or $\theta$, then we will denote them by $c(T)$, $c(T,\theta)$, etc.

Fix $t,t' \in [0,T]$ with $t' > t$,  and $\eps \in \ohalf$, and let $\delta = \eps^{\frac{1}{2\theta}}$. Recall $\tlE_\theta^\eps(t,x) = \tlE^{t,x,\delta,\eps}(t)$. We  write    
\begin{equation} \label{eq:tEsplit}
\begin{split}
& \tlE_\theta^\eps(t',x) - \tlE_\theta^\eps(t,x)
= D_1(t,t',x,\eps) + D_2(t,t',x,\eps),
\end{split}
\end{equation}
where 
\[
\begin{split}
& D_1(t,t',x,\eps) = \tlE^{t',x,\delta,\eps}(t') - \tlE^{t',x,\delta,\eps}(t) \\
& \quad\quad = (1 - e^{-\sigma\eps^{1/4}})e^{-\sigma\eps^{-1/4}x + \frac{1}{2}\eps^{-1/2}\sigma^2(t' + \delta)}\left( \clE^{t',x - \sigma\eps^{-1/4}(t' + \delta),\delta,\eps}(t') - \clE^{t',x - \sigma\eps^{-1/4}(t' + \delta),\delta,\eps}(t) \right), \\
& D_2(t,t',x,\eps) = \tlE^{t',x,\delta,\eps}(t) - \tlE^{t,x,\delta,\eps}(t).
\end{split}
\]
By Lemma \ref{lem:Eprocess}(ii), 
\begin{equation} \label{eq:tEfirsthalf}
\begin{split}
& D_1(t,t',x,\eps) \\
& = -(1 - e^{-\sigma\eps^{1/4}}) \sum_{i \in \ZZ} e^{\sigma\eps^{1/4}i} \int_t^{t'} e^{-\sigma\eps^{-1/4}x + \frac{1}{2}\eps^{-1/2}\sigma^2(t' + \delta)}p_{t' - s + \delta}(X_i^\eps(s) - x + \sigma\eps^{-1/4}(t' + \delta))dB_i^\eps(s).
\end{split}
\end{equation}
With $Y_i^\eps(s)$
as in Lemma \ref{lem:SBPformula}, using \eqref{eq:magic}, 
the quantity \eqref{eq:tEfirsthalf} is equal to 
\[
-(1 - e^{-\sigma\eps^{1/4}}) \sum_{i \in \ZZ} e^{\sigma\eps^{1/4}i} \int_t^{t'} e^{-\sigma\eps^{-1/4}Y_i^\eps(s)  + \frac{1}{2}\sigma^2\eps^{-1/2}s} p_{t' - s + \delta}(Y_i^\eps(s) - x) dB_i^\eps(s).
\]
Therefore, by the Burkholder-Davis-Gundy Inequality, 
\begin{equation} \label{eq:tE1est}
\begin{split}
& \left\| D_1(t,t',x,\eps) \right\|_r \\
& \leq c( 1 - e^{-\sigma\eps^{1/4}}) \left\| \int_t^{t'} \sum_{i \in \ZZ} e^{2\sigma\eps^{1/4}i} e^{-2\sigma\eps^{-1/4}Y_i^\eps(s)  + \sigma^2\eps^{-1/2}s} p_{t' - s + \delta}(Y_i^\eps(s) - x)^2 ds \right\|_{r/2}^{1/2}.
\end{split}
\end{equation}
By the summation by parts formula in Lemma \ref{lem:SBPformula}
applied to $F(s,z) = (p_{t'-s+\delta}(z-x))^2$ with $m=2$,
 the sum appearing in \eqref{eq:tE1est} is equal to 
\begin{align*}
& (1 - e^{-2\sigma\eps^{1/4}})^{-1} \int_{\RR} \clG^\eps(s,z)^2 \bigg[ 2\sigma\eps^{-1/4} p_{t' - s + \delta}(z - x)^2 - \frac{\partial}{\partial z} [p_{t' - s + \delta}(z - x)^2] \bigg] dz \\
& \quad=\frac{2\sigma\eps^{-1/4}}{1 - e^{-2\sigma\eps^{1/4}}} \int_{\RR} \clG^\eps(s,z)^2 \left( 1 + \frac{\eps^{1/4}}{\sigma}\frac{(z - x)}{(t' - s + \delta)} \right) p_{t' - s + \delta}(z - x)^2 dz.
\end{align*}
Replacing the sum with this expression, and recalling \eqref{eq:159mm}, yields
\begin{align} \label{eq:tE1est2}
& \left\| D_1(t,t',x,\eps) \right\|_r \nonumber\\
& \leq \Lambda(\eps)^{1/2} \left\| \int_t^{t'} \int_{\RR} \clG^\eps(s,z)^2 \left( 1 + \frac{\eps^{1/4}}{\sigma}\frac{(z - x)}{(t' - s + \delta)} \right) p_{t' - s + \delta}(z - x)^2 dz ds \right\|_{r/2}^{1/2} \nonumber\\
& \leq c' \left(\int_t^{t'} \int_{\RR} \|\clG^\eps(s,z)\|_r^{2} \left|  1 + \frac{\eps^{1/4}}{\sigma}\frac{(z - x)}{(t' - s + \delta)} \right| p_{t' - s + \delta}(z - x)^2 dz ds \right)^{1/2},\nonumber\\
& \leq c(T) \left(\int_t^{t'} \int_{\RR} e^{c|z|} \left|  1 + \frac{\eps^{1/4}}{\sigma}\frac{(z - x)}{(t' - s + \delta)} \right| p_{t' - s + \delta}(z - x)^2 dz ds \right)^{1/2} \nonumber\\
& \leq c(T) e^{c|x|} \left(\int_t^{t'} \int_{\RR} \left( e^{c|z|} p_{t' - s + \delta}(z)^2 + \frac{\eps^{1/4}|z| e^{c|z|}p_{t' - s + \delta}(z)^2}{\sigma(t' - s + \delta)} \right) dz ds \right)^{1/2} \nonumber\\
& \leq c(T)e^{c|x|} \left(\int_t^{t'} \left( \frac{1}{\sqrt{t' - s + \delta}} +   \frac{\eps^{1/4}}{t' - s + \delta}  \right) ds \right)^{1/2}. 
\end{align}
Here the third inequality follows from Lemma \ref{lem:Gmoments}. Recalling that $\delta = \eps^{\frac{1}{2\theta}}$, the last line above is bounded by 
\begin{equation} \label{eq:tE1est3}
\begin{split}
& c(T)e^{c|x|} \left(\int_t^{t'} \left( \frac{1}{\sqrt{t' - s + \delta}} +    \frac{1}{(t' - s + \delta)^{1 - \frac{\theta}{2}}}\right)   ds \right)^{1/2} \\
& \leq c(T)e^{c|x|} \left( 2|t'-t|^{1/2} + \frac{2}{\theta}|t'-t|^{\theta/2}\right)^{1/2} \le  c(T,\theta)e^{c|x|} |t'-t|^{\theta/4}.
\end{split}
\end{equation}
Now we turn to estimating $D_2(t,t',x,\eps)$. By \eqref{eq:mart_rep},
 we have
\begin{align}\label{eq:619nn}
& D_2(t,t',x,\eps) \nonumber\\
&= \tlE^{t',x,\delta,\eps}(0) - \tlE^{t,x,\delta,\eps}(0)\nonumber\\
& \quad - (1 - e^{-\sigma\eps^{1/4}}) \sum_{i \in \ZZ} e^{\sigma\eps^{1/4}i} \int_0^t \bigg\{ e^{-\sigma\eps^{-1/4}x + \frac{1}{2}\sigma^2\eps^{-1/2}(t' + \delta)}p_{t' - s + \delta}\left(X_i^\eps(s) - x + \sigma\eps^{-1/4}(t' + \delta)\right) \nonumber\\
& \hspace{2in} - e^{-\sigma\eps^{-1/4}x + \frac{1}{2}\sigma^2\eps^{-1/2}(t + \delta)}p_{t - s + \delta}\left(X_i^\eps(s) - x + \sigma\eps^{-1/4}(t + \delta)\right) \bigg\}dB_i^\eps(s) \nonumber\\
&= \tlE^{t',x,\delta,\eps}(0) - \tlE^{t,x,\delta,\eps}(0) + \tilde D_2(t,t',x,\eps),
\end{align}
where  using
\eqref{eq:magic} again, we see that
\[
\begin{split}
& \tilde D_2(t,t',x,\eps) \\
& = (1 - e^{-\sigma\eps^{1/4}}) \sum_{i \in \ZZ} e^{\sigma\eps^{1/4}i} \int_0^t e^{-\sigma\eps^{-1/4}Y_i^\eps(s) + \frac{1}{2}\sigma^2\eps^{-1/2}s} \Big\{ p_{t' - s + \delta}(Y_i^\eps(s) - x) - p_{t - s + \delta}(Y_i^\eps(s) - x) \Big\}dB_i^\eps(s).
\end{split}
\]
By the Burkholder-Davis-Gundy inequality, 
\begin{equation} \label{eq:D2est1}
\begin{split}
& \| \tilde D_2(t,t',x,\eps) \|_r \\
& \leq  c(1 - e^{-\sigma\eps^{1/4}}) \Bigg\| \int_0^t \sum_{i \in \ZZ} e^{2\sigma\eps^{1/4}i -2\sigma\eps^{-1/4}Y_i^\eps(s) + \sigma^2\eps^{-1/2}s} \bigg\{ p_{t' - s + \delta}(Y_i^\eps(s) - x) - p_{t - s + \delta}(Y_i^\eps(s) - x) \bigg\}^2 ds \Bigg\|_{r/2}^{1/2}.
\end{split}
\end{equation}
Applying Lemma \ref{lem:SBPformula} with $m=2$ and
\[
F(s,z) = \Big\{ p_{t' - s + \delta}(z - x) - p_{t - s + \delta}(z - x) \Big\}^2, 
\]
we see that the sum appearing in \eqref{eq:D2est1} is equal to 
\[
\frac{2\sigma\eps^{-1/4}}{1 - e^{-2\sigma\eps^{1/4}}}\int_{\RR} \clG^\eps(s,z)^2 \left[ F(s,z) -\frac{\eps^{1/4}}{2\sigma}\frac{\partial}{\partial z} F(s,z) \right]dz.
\]
Hence, as in \eqref{eq:tE1est2},
\begin{equation} \label{eq:D2est2}
\begin{split}
 \| \tilde D_2(t,t',x,\eps) \|_r  \leq c(T) \left( \int_0^t \int_{\RR} e^{c|z|} \left| F(z) -\frac{\eps^{1/4}}{2\sigma}\frac{\partial}{\partial z} F(z)\right| dz ds  \right)^{1/2}.
\end{split}
\end{equation}
Let
$h:=t'-t$, $u:=t-s+\delta$.
Write
$
\Delta_u(y):=p_{u+h}(y)-p_u(y)$.
Then
$
F(z)=\Delta_u(z-x)^2$
and therefore
\[
\partial_z F(z)
=
2\Delta_u(z-x)\partial_z\Delta_u(z-x).
\]
Hence, using $
e^{c|z|}\le e^{c|x|}e^{c|z-x|}$, and making the substitution  $y=z-x$,
\[
\begin{aligned}
e^{-c|x|}\| \tilde D_2(t,t',x,\eps) \|_r^2
&\le
c(T)\int_0^t\int_{\mathbb R}
e^{c|y|}\Delta_u(y)^2\,dy\,ds \\
&\quad
+
c(T)\frac{\varepsilon^{1/4}}{\sigma}
\int_0^t\int_{\mathbb R}
e^{c|y|}
|\Delta_u(y)|\,|\partial_z\Delta_u(y)|\,dy\,ds .
\end{aligned}
\]
Using
\[
\Delta_u(y)=\int_u^{u+h}\partial_m p_m(y)\,dm
\]
and the standard Gaussian estimates
\[
\left(\int_{\mathbb R}e^{c|y|}
|\partial_m p_m(y)|^2\,dy\right)^{1/2}
\le c(T)m^{-5/4}, \qquad 
\left(\int_{\mathbb R}e^{c|y|}
|\partial_y\partial_m p_m(y)|^2\,dy\right)^{1/2}
\le c(T)m^{-7/4},
\]
we get
\[
\left(
\int_{\mathbb R}e^{c|y|}\Delta_u(y)^2\,dy
\right)^{1/2}
\le
c(T)\int_u^{u+h}m^{-5/4}\,dm
\]
and
\[
\left(
\int_{\mathbb R}e^{c|y|}
|\partial_y\Delta_u(y)|^2\,dy
\right)^{1/2}
\le
c(T)\int_u^{u+h}m^{-7/4}\,dm .
\]
Thus by Cauchy-Schwarz inequality,
\[
\int_{\mathbb R}e^{c|y|}
|\Delta_u(y)|\,|\partial_y\Delta_u(y)|\,dy
\le
c(T)^2
\left(\int_u^{u+h}m^{-5/4}\,dm\right)
\left(\int_u^{u+h}m^{-7/4}\,dm\right).
\]
Thus, recalling the substitution $u=t-s+\delta$,
\begin{equation}\label{eq:913nn}
\begin{aligned}
e^{-c|x|}\| \tilde D_2(t,t',x,\eps) \|_r^2
&\le
c'(T)
\int_\delta^{t+\delta}
\left(\int_u^{u+h}m^{-5/4}\,dm\right)^2du \\
&\quad
+
c'(T)\varepsilon^{1/4}
\int_\delta^{t+\delta}
\left(\int_u^{u+h}m^{-5/4}\,dm\right)
\left(\int_u^{u+h}m^{-7/4}\,dm\right)du .
\end{aligned}
\end{equation}
Making the substitutions
$u=hr$, $m=hw$, we get
\begin{equation}\label{eq:925nn}
\begin{aligned}
\int_\delta^{t+\delta}
\left(\int_u^{u+h}m^{-5/4}\,dm\right)^2du &\le 
\int_0^\infty
\left(\int_u^{u+h}m^{-5/4}\,dm\right)^2du\\
&=
h^{3-2(5/4)}
\int_0^\infty
\left(\int_r^{r+1}w^{-5/4}\,dw\right)^2dr \\
&\le c(T) h^{1/2} \le c(T,\theta)h^{\theta/2},
\end{aligned}
\end{equation}
where the last inequality uses  $0<\theta<1$ and $h\le T$.

For the second term on the right side of \eqref{eq:913nn}, define
\[
I_{\delta,h}:=
\int_\delta^\infty
\left(\int_u^{u+h}m^{-5/4}\,dm\right)
\left(\int_u^{u+h}m^{-7/4}\,dm\right)du .
\]
Again using $u=hr$, $m=hw$,
\[
I_{\delta,h}
=
\int_{\delta/h}^\infty
\left(\int_r^{r+1}w^{-5/4}\,dw\right)
\left(\int_r^{r+1}w^{-7/4}\,dw\right)dr .
\]

We split into two cases.

First suppose \(\delta\le h\). Then \(\delta/h\le1\). For
\(0<r\le1\),
\[
\left(\int_r^{r+1}w^{-5/4}\,dw\right)
\left(\int_r^{r+1}w^{-7/4}\,dw\right)
\le c r^{-1}.
\]
For \(r\ge1\),
\[
\left(\int_r^{r+1}w^{-5/4}\,dw\right)
\left(\int_r^{r+1}w^{-7/4}\,dw\right)\le c r^{-3}.
\]
 Therefore
\[
I_{\delta,h}
\le
c\int_{\delta/h}^1\frac{dr}{r}
+
c\int_1^\infty r^{-3}\,dr
\le
c\left(1+\log\frac{h}{\delta}\right).
\]
Consequently,
\[
\varepsilon^{1/4}I_{\delta,h}
=
\delta^{\theta/2}I_{\delta,h}
\le
C\delta^{\theta/2}
\left(1+\log\frac{h}{\delta}\right).
\]
Since \(\delta\le h\), write \(a=\delta/h\in(0,1]\). Then, since
$
\sup_{0<a\le1}a^{\theta/2}\left(1+\log\frac1a\right)<\infty,
$
we have 
\[
\delta^{\theta/2}
\left(1+\log\frac{h}{\delta}\right)
=
h^{\theta/2}
a^{\theta/2}\left(1+\log\frac1a\right)
\le
c(\theta) h^{\theta/2}.
\]
Thus, in the case \(\delta\le h\),
\[
\varepsilon^{1/4}I_{\delta,h}
\le c(\theta) h^{\theta/2}.
\]

Now suppose \(h<\delta\).  For $r\ge \delta/h> 1$,
\[
\int_r^{r+1}w^{-5/4}\,dw\le c r^{-5/4},
\qquad
\int_r^{r+1}w^{-7/4}\,dw\le c r^{-7/4}.
\]
Therefore
\[
\varepsilon^{1/4}I_{\delta,h}
\le
c\varepsilon^{1/4}\int_{\delta/h}^\infty r^{-3}\,dr
\le
c\varepsilon^{1/4}\left(\frac{h}{\delta}\right)^2 =c\delta^{\theta/2}\left(\frac{h}{\delta}\right)^2
=
c h^{\theta/2}
\left(\frac{h}{\delta}\right)^{2-\theta/2} \le c h^{\theta/2}.
\]
Thus we have shown that for all $\delta,h$
\[
\varepsilon^{1/4}I_{\delta,h} \le c(\theta) h^{\theta/2}.\]
Using this fact along with \eqref{eq:925nn} in \eqref{eq:913nn}, we obtain
\begin{equation}\label{eq:617nn}
\left\| \tilde D_2(t,t',x,\eps) \right\|_r \le c(T,\theta)e^{c|x|} |t'-t|^{\theta/4}.
\end{equation}
Moreover, by the second bound in Lemma \ref{lem:continuity0-bd}, we conclude that for all $x\in \RR$, $t, t' \in [0,T]$,
$$
\|\tlE^{t',x,\delta,\eps}(0) - \tlE^{t,x,\delta,\eps}(0)\|_r \le b_3e^{b_4|x|} |t'-t|^{\theta/4}.
$$
Combining this estimate with \eqref{eq:617nn}, \eqref{eq:619nn}, \eqref{eq:tE1est3} and \eqref{eq:tEsplit}
we conclude the proof of part (ii) of the lemma.
\end{proof}

\subsubsection{{Proof of tightness of $\{\tlE_\theta^\eps\}$}}

\begin{proof}[Proof of Corollary \ref{cor:tightcont}]
Fix $T, K \in (0,\infty)$ and consider the rectangle
$[0,T]\times[-K,K]$. First, by Lemma~\ref{lem:zerobd},
\[
\sup_{\eps\in(0,1]}
\|\tlE_\theta^\eps(0,0)\|_r
=
\sup_{\eps\in(0,1]}
\|\tlE^{0,0,\eps^{1/(2\theta)},\eps}(0)\|_r
<\infty .
\]
Together with the {continuity} estimates in Lemma~\ref{lem:continuity}, this gives
tightness of the one-point marginals on each compact rectangle.

We next prove a uniform modulus estimate. Fix \(r\in[1,\infty)\). 
By Lemma \ref{lem:continuity}, for \((t,x),(s,y)\in[0,T]\times[-K,K]\),
\[
\begin{split}
\|\tlE_\theta^\eps(t,x)-\tlE_\theta^\eps(s,y)\|_r
&\leq
\|\tlE_\theta^\eps(t,x)-\tlE_\theta^\eps(t,y)\|_r
+
\|\tlE_\theta^\eps(t,y)-\tlE_\theta^\eps(s,y)\|_r  \\
&\leq
C'\left(|x-y|^{\theta/2}+|t-s|^{\theta/4}\right) \le C\left(|x-y|+|t-s|\right)^{\theta/4},
\end{split}
\]
for finite constants $C,C'$ depending on $\theta,r,T,K$.
Consequently,
\[
\EE\left[
|\tlE_\theta^\eps(t,x)-\tlE_\theta^\eps(s,y)|^r
\right]
\leq
C^r\left(|x-y|+|t-s|\right)^{r\theta/4}.
\]
Choose \(r\) large enough so that \(r\theta/4>2\). Since the parameter
set \([0,T]\times[-K,K]\) is two-dimensional, the Kolmogorov--Chentsov
tightness criterion \cite[Theorem 1.4.1]{kunita1990stochastic} implies that
$
\{\tlE_\theta^\eps:\eps\in(0,1]\}
$
is tight in \(C([0,T]\times[-K,K]:\RR)\).

Finally, since \(T,K\in\ohalf\) were arbitrary, tightness holds in
\(C(\half \times \RR:\RR)\), equipped with the topology of uniform convergence
on compact subsets.
\end{proof}

\subsection{Uniform closeness of $\tlE_\theta^\eps$ and $\clG^\eps$} \label{ssec:unifclose}


The following result will allow us to prove convergence of $\clG^\eps$ via that of $\tilde \clE^{\eps}_{\theta}$.
Once again we assume throughout the section that the initial distribution is $\gamma= \gamma_1$.

\begin{proposition} \label{lem:unifclose}
Fix  $T \in \ohalf$, $\theta \in (0,1)$ and a compact set $K \subset \RR$.
Then, for any $r \ge 1$,
\[
\left\|\sup_{t \in [0,T], x \in K} \left|\clG^\eps(t,x) - \tilde \clE_{\theta}^\eps(t,x) \right| \right\|_r \to 0 \text{ as } \eps \to 0.
\]
\end{proposition}

For $(a,t) \in \RR \times \RR_+$, and $\theta \in (0,1)$ let
$$B(a,t, \theta) := \left[a- \frac{1}{2}\eps^{1/(4\theta)}, a+ \frac{1}{2}\eps^{1/(4\theta)}\right] \times [t, t+\eps].
$$
For some $C= C(K, T)$, let
$\{B(a_i, t_i,\theta), 1 \le i \le C \lfloor\eps^{-\frac{1}{4\theta}-1}\rfloor =: M_{\eps}(\theta)\}$ be a collection of such rectangles such that
$K\times [0,T] \subset \bigcup_{i} B(a_i, t_i, \theta)$. Let $K_1\supset K$ be a compact set such that $[a_i-1, a_i+1] \subset K_1$ for all $1 \le i \le M_{\eps}$.

The following lemma will be key in the proof of Proposition \ref{lem:unifclose}.
\begin{lemma} \label{lem:unifclose-assist}
Let 
$T,  K, K_1$, be as above. Fix $\theta \in (0,1)$ and let $\theta_0 \in (\theta,1)$. Then there exists a $\gamma= \gamma(\theta_0) \in (0,\infty)$ such that for any $r\ge 1$,
$$\sup_{(a_0, t_0) \in K_1\times [0,T]} \sup_{\eps \in (0,1)}
\eps^{-\gamma} \left\|\sup_{(x,t) \in B(a_0, t_0, \theta_0)} |\clG^\eps(t,x) - \tilde \clE_{\theta}^\eps(t,x)| \right\|_r <\infty.
$$
\end{lemma}
We first provide the proof of Proposition \ref{lem:unifclose}, assuming Lemma \ref{lem:unifclose-assist}.
Proof of the lemma will be given in the next subsection.\\

\begin{proof}[Proof of Proposition \ref{lem:unifclose}]
Fix $T,\theta,K$ as in the statement of the proposition, and fix
$r\ge1$. Choose $\theta_0\in(\theta,1)$. Let
$
\{B(a_i,t_i,\theta_0):1\le i\le M_\eps(\theta_0)\}
$
be a cover of $K\times[0,T]$ as above. By Lemma~\ref{lem:unifclose-assist},
there exist $\gamma>0$ and, for every $q\ge1$, a constant
$c_q=c_q(T,K_1,\theta,\theta_0)<\infty$ such that
\[
\sup_{1\le i\le M_\eps(\theta_0)}
\left\|
\sup_{(x,t)\in B(a_i,t_i,\theta_0)}
\left|\clG^\eps(t,x)-\tlE^\eps_\theta(t,x)\right|
\right\|_q
\le c_q\eps^\gamma \text{ for all } \eps \in (0,1).
\]
Choose $q>r$ so large that
$
\gamma-\frac{1+1/(4\theta_0)}{q}>0 .
$
Since $q>r$, monotonicity of $L^p$-norms and the inequality
$\max_i a_i\le(\sum_i a_i^q)^{1/q}$ for nonnegative numbers give
\begin{align*}
&\left\|
\sup_{0\le t\le T}\sup_{x\in K}
\left|\clG^\eps(t,x)-\tlE^\eps_\theta(t,x)\right|
\right\|_r \le \left\|
\sup_{0\le t\le T}\sup_{x\in K}
\left|\clG^\eps(t,x)-\tlE^\eps_\theta(t,x)\right|
\right\|_q\\
&\qquad\le
\left\|
\max_{1\le i\le M_\eps(\theta_0)}
\sup_{(x,t)\in B(a_i,t_i,\theta_0)}
\left|\clG^\eps(t,x)-\tlE^\eps_\theta(t,x)\right|
\right\|_q \\
&\qquad\le
\left\|
\left(
\sum_{i=1}^{M_\eps(\theta_0)}
\left[
\sup_{(x,t)\in B(a_i,t_i,\theta_0)}
\left|\clG^\eps(t,x)-\tlE^\eps_\theta(t,x)\right|
\right]^q
\right)^{1/q}
\right\|_q \\
&\qquad= \left(
\sum_{i=1}^{M_\eps(\theta_0)}
\left\|
\sup_{(x,t)\in B(a_i,t_i,\theta_0)}
\left|\clG^\eps(t,x)-\tlE^\eps_\theta(t,x)\right|
\right\|_q^q
\right)^{1/q}.
\end{align*}
Therefore,
\[
\begin{split}
\left\|
\sup_{0\le t\le T}\sup_{x\in K}
\left|\clG^\eps(t,x)-\tlE^\eps_\theta(t,x)\right|
\right\|_r
&\le
c_q M_\eps(\theta_0)^{1/q}\eps^\gamma
\le
c_q \eps^{\gamma-\{1+1/(4\theta_0)\}/q}.
\end{split}
\]
By our choice of $q$, the exponent on the right-hand side is positive.
Hence the right-hand side converges to zero as $\eps\to0$, proving the
proposition.
\end{proof}

\subsubsection{Proof of Lemma \ref{lem:unifclose-assist}.}
We begin with the following lemma. Recall the definition of $\Theta(a,t)$, for $(a,t) \in \RR \times \RR_+$,
from \eqref{eq:thetadefn}.
\begin{lemma}\label{lem:G}
Fix $\theta \in (0,1)$. Then, with $\beta = \beta(\theta) = \frac{1}{4}\left[\left(\frac{1}{\theta}-1\right) \wedge \frac{1}{2}\right]$, for any $r\ge 1$,
$$\sup_{(a_0, t_0) \in K_1\times [0,T]}\sup_{\eps \in (0,1)}
\eps^{-\beta} \left\| \sup_{(x,t) \in \Theta(a_0, t_0)}|\clG^{\eps}(t,x) - \clG^{\eps}(t_0,a_0)| \right\|_r <\infty.
$$
\end{lemma}
\begin{proof}
Note that for $(x,t) \in \Theta(a_0, t_0)$
\begin{align*}
\clG^{\eps}(t,x) - \clG^{\eps}(t_0,a_0) &= \exp\left( \sigma\eps^{1/4}N_\eps(t,x-\sigma\eps^{-1/4}t) - \sigma\eps^{-1/4}x + \frac{1}{2}\sigma^2\eps^{-1/2}t \right)\\
&\quad - \exp\left( \sigma\eps^{1/4}N_\eps(t_0,a_0-\sigma\eps^{-1/4}t_0) - \sigma\eps^{-1/4}a_0 + \frac{1}{2}\sigma^2\eps^{-1/2}t_0 \right)\\
&= \clG^{\eps}(t_0,a_0)
\Big[\exp\left( \sigma\eps^{1/4}(N_\eps(t,x-\sigma\eps^{-1/4}t)- N_\eps(t_0,a_0-\sigma\eps^{-1/4}t_0))\right)\\
&\quad\quad \times 
e^{\frac{1}{2}\sigma^2\eps^{-1/2}(t-t_0)}e^{\sigma\eps^{-1/4}(a_0-x)}-1\Big].
\end{align*}
Using the inequality $|e^z-1| \le e^{|z|}-1$ for $z \in \RR$,  and letting $\lambda = \left[\frac{1}{4}\left(\frac{1}{\theta}-1\right)\right] \wedge \frac{1}{2}$, we get that
\begin{multline*}
\sup_{(x,t) \in \Theta(a_0, t_0)}|\clG^{\eps}(t,x) - \clG^{\eps}(t_0,a_0)|\\
\le \clG^{\eps}(t_0,a_0) e^{c\eps^{\lambda}}
\left[\exp\left( \sigma\eps^{1/4}\sup_{(x,t) \in \Theta(a_0, t_0)}|N_\eps(t,x-\sigma\eps^{-1/4}t)- N_\eps(t_0,a_0-\sigma\eps^{-1/4}t_0)|\right) - 1\right]\\
 + \clG^{\eps}(t_0,a_0)\sup_{(x,t) \in \Theta(a_0, t_0)}\left|
 e^{\frac{1}{2}\sigma^2\eps^{-1/2}(t-t_0)}e^{\sigma\eps^{-1/4}(a_0-x)}-1\right|.
 \end{multline*}
 Using the second statement in Lemma \ref{lem:countfluc}, and Lemma \ref{lem:Gmoments}, together with 
 Cauchy-Schwarz inequality gives that
\begin{align*}
&\sup_{(a_0, t_0) \in K_1\times [0,T]}\sup_{\eps \in (0,1)}
\eps^{-\beta} \Bigg \|\clG^{\eps}(t_0,a_0) \\
&\quad \times \left[\exp\left( \sigma\eps^{1/4}\sup_{(x,t) \in \Theta(a_0, t_0)}|N_\eps(t,x-\sigma\eps^{-1/4}t)- N_\eps(t_0,a_0-\sigma\eps^{-1/4}t_0)|\right) - 1\right]\Bigg\|_r <\infty.
\end{align*}
Also, from the definition of $\beta$,
\begin{align*}
\sup_{(a_0, t_0) \in K_1\times [0,T]}\sup_{\eps \in (0,1)}\eps^{-\beta}\sup_{(x,t) \in \Theta(a_0, t_0)}\left|
 e^{\frac{1}{2}\sigma^2\eps^{-1/2}(t-t_0)}e^{\sigma\eps^{-1/4}(a_0-x)}-1\right|<\infty.
\end{align*}
Thus applying Lemma \ref{lem:Gmoments} once more we see that
\begin{align*}
&\sup_{(a_0, t_0) \in K_1\times [0,T]}\sup_{\eps \in (0,1)}
\eps^{-\beta}
 \|\clG^{\eps}(t_0,a_0)\|_r \left|
 e^{\frac{1}{2}\sigma^2\eps^{-1/2}(t-t_0)}e^{\sigma\eps^{-1/4}(a_0-x)}-1\right| <\infty.
\end{align*}
Combining the last two estimates, we have the result.
\end{proof}

From Lemma \ref{lem:Gmoments}, for each $r\ge 1$, there is a $c^*=c^*(r)<\infty$, such that for all $T<\infty$, 
\begin{equation}\label{eq:1150n}
\sup_{\eps \in (0,1)} \sup_{s \in [0,T+1]} \sup_{y \in \RR} e^{-c^*|y|/2} \|\clG^{\eps}(s,y)\|_r <\infty.
\end{equation}
The following lemma will be the second ingredient in the proof of Lemma \ref{lem:unifclose-assist}.
\begin{lemma}\label{lem:RR}
Fix $\theta \in (0,1)$ and let $T, K_1$ be as above.
For $a_0\in \RR,  t_0 \ge 0$, $(x,t) \in B(a_0, t_0)$ and $\eps, \delta \in (0,1)$, define $H(x,\eps,\delta):= [x-\sigma\eps^{-1/4}\delta-(\delta \log \delta^{-1})^{1/2}, x-\sigma\eps^{-1/4}\delta+(\delta \log \delta^{-1})^{1/2}]^c$, and
$$
R(x,t, t_0, a_0, \eps, \delta):= \int_{H(x,\eps,\delta)} (\clG^{\eps}(t,u)- \clG^{\eps}(t_0,a_0)) p_{\delta}(u-x) du.
$$
Then, for each $r\ge 1$, there exists $C \in (0,\infty)$ such that, for all  $\eps, \delta \in (0,1)$ satisfying $\sigma \eps^{-1/4}\delta + c^*\delta \le \frac{1}{4}(\delta \log \delta^{-1})^{1/2}$,
$$\sup_{(a_0, t_0) \in K_1\times [0,T]} \left\|\sup_{(x,t) \in B(a_0, t_0, \theta)} |R(x,t, t_0, a_0, \eps, \delta)|\right\|_r \le C\delta^{1/8}.
$$
\end{lemma}
\begin{proof}
Fix $\eps, \delta \in (0,1)$ as in the statement of the lemma and let
\begin{equation}\label{eq:636n}
\alpha:= (\delta \log \delta^{-1})^{1/2} - \sigma\eps^{-1/4}\delta, \; \alpha':= (\delta \log \delta^{-1})^{1/2} + \sigma\eps^{-1/4}\delta.\end{equation}
Then, with $x,t, t_0, a_0$ as in the statement of the lemma,
\begin{equation}\label{eq:1139}
\begin{aligned}
|R(x,t, t_0, a_0, \eps, \delta)| &\le \int_{-\infty}^{x-\alpha'} (\clG^{\eps}(t,u)+ \clG^{\eps}(t_0,a_0)) p_{\delta}(u-x) du\\
&\quad 
+ \int_{x+\alpha}^{\infty} (\clG^{\eps}(t,u)+ \clG^{\eps}(t_0,a_0)) p_{\delta}(u-x) du.
\end{aligned}
\end{equation}
Note that
\begin{align*}
\int_{x+\alpha}^{\infty} \clG^{\eps}(t,u) p_{\delta}(u-x) du &= \int_{\alpha+a_0}^{\infty} \clG^{\eps}(t,w+x-a_0) p_{\delta}(w-a_0) dw\\
&= \int_{\alpha+a_0}^{\infty} \frac{\clG^{\eps}(t,w+x-a_0)}{\clG^{\eps}(t_0,w)} \clG^{\eps}(t_0,w) p_{\delta}(w-a_0) dw.
\end{align*}
Thus, with $B= B(a_0, t_0, \theta)$,
\begin{multline}\label{eq:R1}
\sup_{(x,t) \in B} \int_{x+\alpha}^{\infty} \clG^{\eps}(t,u) p_{\delta}(u-x) du\\
\le \int_{\alpha+a_0}^{\infty} \left(\sup_{(x,t) \in B}\frac{\clG^{\eps}(t,w+x-a_0)}{\clG^{\eps}(t_0,w)}\right) \clG^{\eps}(t_0,w) p_{\delta}(w-a_0) dw.
\end{multline}
Next note that,
\begin{multline*}
\frac{\clG^{\eps}(t,w+x-a_0)}{\clG^{\eps}(t_0,w)} = 
\exp\left( \sigma\eps^{1/4} (N_\eps(t,w+x-a_0-\sigma\eps^{-1/4}t)- N_\eps(t_0,w-\sigma\eps^{-1/4}t_0)\right)\\
\times e^{\frac{1}{2}\sigma^2\eps^{-1/2}(t-t_0)}e^{-\sigma\eps^{-1/4}(x-a_0)}\\
\le c_1 \exp\left( \sigma\eps^{1/4} (N_\eps(t,w+x-a_0-\sigma\eps^{-1/4}t)- N_\eps(t_0,w-\sigma\eps^{-1/4}t_0)\right).
\end{multline*}
Note that, for $(x,t)\in B$, and $w \in \RR$, $(w+x-a_0, t) \in B(w, t_0, \theta) \subset \Theta(w, t_0, \theta)$.
Thus, for any $w \in \RR$, with $\Theta = \Theta(w, t_0, \theta)$,
\begin{align*}
&\left\|\sup_{(x,t) \in B}\frac{\clG^{\eps}(t,w+x-a_0)}{\clG^{\eps}(t_0,w)}\right\|_r\\
&\quad \le c_1 \left\|\sup_{(x,t) \in \Theta} \exp\left( \sigma\eps^{1/4} (N_\eps(t,w+x-a_0-\sigma\eps^{-1/4}t)- N_\eps(t_0,w-\sigma\eps^{-1/4}t_0)\right)\right\|_r.
\end{align*}
Thus, from the second statement in Lemma \ref{lem:countfluc}, for each $r\ge1$, there exists a $c_2(r) <\infty$ such that
$$
\sup_{\eps \in (0,1)} \left\|\sup_{(x,t) \in B}\frac{\clG^{\eps}(t,w+x-a_0)}{\clG^{\eps}(t_0,w)}\right\|_r = c_2(r)(|w|+1)^d <\infty,
$$
where $d= d(r)$ is as in Lemma \ref{lem:countfluc}.
Using this estimate in \eqref{eq:R1}, for some $c_3(r)<\infty$,
$$
\left\|\sup_{(x,t) \in B} \int_{x+\alpha}^{\infty} \clG^{\eps}(t,u) p_{\delta}(u-x) du\right\|_r
\le c_3(r) \int_{\alpha+a_0}^{\infty} \|\clG^{\eps}(t_0,w)\|_{2r} (|w|+1)^d p_{\delta}(w-a_0) dw.
$$
Similarly, for some $c_4(r)<\infty$,
$$
\left\|\sup_{(x,t) \in B} \int_{-\infty}^{x-\alpha'} \clG^{\eps}(t,u) p_{\delta}(u-x) du\right\|_r
\le c_4(r) \int_{-\infty}^{a_0-\alpha'}\|\clG^{\eps}(t_0,w)\|_{2r} (|w|+1)^d p_{\delta}(w-a_0) dw.
$$
Using the last two estimates in \eqref{eq:1139} we obtain, for some $c_5(r)<\infty$,
\begin{multline*}
\left\|\sup_{(x,t) \in B} |R(x,t, t_0, a_0, \eps, \delta)|\right\|_r
\le
c_5(r) \int_{\alpha+a_0}^{\infty} \|\clG^{\eps}(t_0,w)\|_{2r} (|w|+1)^d p_{\delta}(w-a_0) dw\\
+ c_5(r)  \int_{-\infty}^{a_0-\alpha'}\|\clG^{\eps}(t_0,w)\|_{2r} (|w|+1)^d p_{\delta}(w-a_0) dw\\
+ c_5(r) \|\clG^{\eps}(t_0,a_0)\|_{r} \left(\int_{\alpha+a_0}^{\infty} p_{\delta}(w-a_0) dw + 
\int_{-\infty}^{a_0-\alpha'} p_{\delta}(w-a_0) dw\right).
\end{multline*}
Recalling the bound in \eqref{eq:1150n}, we now have, for some $c_6(T,r)<\infty$,
\begin{equation}\label{eq:R2}
\left\|\sup_{(x,t) \in B} |R(x,t, t_0, a_0, \eps, \delta)|\right\|_r
\le c_6(T,r) \left(\int_{\alpha+a_0}^{\infty} e^{c^*|w|} p_{\delta}(w-a_0) dw
+  \int_{-\infty}^{a_0-\alpha'} e^{c^*|w|} p_{\delta}(w-a_0) dw\right).
\end{equation}
For any $x\in \RR$,
$$
\int_{\alpha+x}^{\infty} e^{c^*|w|} p_{\delta}(w-x) dw \le
\int_{\alpha+x}^{|\alpha+x|} e^{c^*|w|} p_{\delta}(w-x) dw + \int_{|\alpha+x|}^{\infty} e^{c^*|w|} p_{\delta}(w-x) dw.
$$
If $x+\alpha\ge 0$, we have, with $\bar \Phi (z)= 1- \Phi(z)$, where $\Phi$ is the cumulative distribution function of a standard normal random variable, for some constant $c' \in (0,\infty)$,
\begin{multline*}
\int_{\alpha+x}^{\infty} e^{c^*|w|} p_{\delta}(w-x) dw = \int_{\alpha+x}^{\infty} e^{c^*x+ (c^*)^2\delta/2} p_{\delta}(x-w+c^*\delta) dw\\
= e^{c^*x+ (c^*)^2\delta/2} \bar \Phi\left(\frac{\alpha-c^*\delta}{\sqrt{\delta}}\right) \le 
c' e^{c^*x+ (c^*)^2\delta/2} e^{-(\alpha-c^*\delta)^2/(2\delta)} = c' e^{c^*x - \alpha^2/(2\delta) +\alpha c^*}.
\end{multline*}
Next note that, using the assumption that $\sigma \eps^{-1/4}\delta + c^*\delta \le \frac{1}{4}(\delta \log \delta^{-1})^{1/2}$, 
\begin{multline*}
- \alpha^2/(2\delta) +\alpha c^* = \frac{\alpha}{\delta}\left(c^*\delta + \frac{1}{2} \sigma \eps^{-1/4}\delta - 
\frac{1}{2} (\delta \log \delta^{-1})^{1/2}\right)\\
\le \frac{1}{\delta} \left(\frac{3}{4}(\delta \log \delta^{-1})^{1/2} + c^*\delta\right) \left( - \frac{1}{4} (\delta \log \delta^{-1})^{1/2}\right) \le - \frac{3}{16\delta}(\delta \log \delta^{-1})\le - \frac{1}{8}\log \delta^{-1}.
\end{multline*}
Thus combining the last two estimates,
$$
\int_{\alpha+x}^{\infty} e^{c^*|w|} p_{\delta}(w-x) dw \le e^{c^*x} e^{- \frac{1}{8}\log \delta^{-1}}
\le c' e^{c^*x}\delta^{1/8}.
$$
Now consider the case $x+\alpha <0$.
Then
\begin{multline*}
\int_{\alpha+x}^{\infty} e^{c^*|w|} p_{\delta}(w-x) dw \le e^{-c^*(x+\alpha)}\bar\Phi(\alpha/\sqrt{\delta})+
\int_{-x-\alpha}^{\infty} e^{c^*w}p_{\delta}(w-x) dw\\
= e^{-c^*(x+\alpha)}\bar\Phi(\alpha/\sqrt{\delta})
+ \int_{-x-\alpha}^{\infty} e^{c^*x+ (c^*)^2\delta/2} p_{\delta}(x-w+c^*\delta) dw\\
= e^{-c^*(x+\alpha)}\bar\Phi(\alpha/\sqrt{\delta}) + e^{c^*x+ (c^*)^2\delta/2}\bar \Phi\left(\frac{-2x-\alpha-c^*\delta}{\sqrt{\delta}}\right)\\
\le e^{-c^*(x+\alpha)}\bar\Phi(\alpha/\sqrt{\delta}) + e^{c^*x+ (c^*)^2\delta/2}\bar \Phi\left(\frac{\alpha-c^*\delta}{\sqrt{\delta}}\right)
\le 2c'  e^{c^*|x|} \delta^{1/8}.
\end{multline*}
Thus we have shown that for all $x \in \RR$
$$
\int_{\alpha+x}^{\infty} e^{c^*|w|} p_{\delta}(w-x) dw \le 2c'  e^{c^*|x|} \delta^{1/8}.
$$
Similarly, we see using $\alpha' > \alpha$ that
$$
\int_{-\infty}^{x-\alpha'} e^{c^*|w|} p_{\delta}(w-x) dw = \int_{-\infty}^{-x+\alpha'} e^{c^*|w|} p_{\delta}(w-(-x)) dw \le 2c'  e^{c^*|x|} \delta^{1/8}.
$$
The result follows on using the last two estimates in \eqref{eq:R2}.
\end{proof}

We now complete the proof of Lemma \ref{lem:unifclose-assist}.

\begin{proof}[Proof of Lemma \ref{lem:unifclose-assist}] Let $T, K, K_1, \theta, \theta_0$ be as in the statement of the lemma. Fix $r\ge 1$.
Let $\delta = \eps^{1/(2\theta)}$. Note that, with this choice of $\delta$, for $\eps$ sufficiently small
$\sigma \eps^{-1/4}\delta + c^*\delta \le \frac{1}{4}(\delta \log \delta^{-1})^{1/2}$. Thus we assume without loss of generality that this inequality is satisfied. Write $R(x,t, t_0, a_0, \eps, \eps^{1/(2\theta)})= \tilde R(x,t, t_0, a_0, \eps)$. 
Consider for $a_0\in \RR,  t_0 \ge 0$,  the rectangle $B(a_0, t_0, \theta_0)$ and consider $(x,t) \in B(a_0, t_0, \theta_0)$.
Using Lemma \ref{lem:Erep},
\begin{align*}
\tilde \clE_{\theta}^\eps(t,x)-\clG^\eps(t_0,a_0) &=
\int_{-\infty}^{\infty} (\clG^\eps(t,u) - \clG^\eps(t_0,a_0))p_{\delta}(u-x) du\\
&= \int_{x-\alpha'}^{x+\alpha} (\clG^\eps(t,u) - \clG^\eps(t_0,a_0))p_{\delta}(u-x) du +  \tilde R(x,t, t_0, a_0, \eps),
\end{align*}
where $\alpha', \alpha$ are as in \eqref{eq:636n}.
Thus, with $B= B(a_0, t_0, \theta_0)$, 
\begin{align*}
\sup_{(x,t) \in B} |\tilde \clE_{\theta}^\eps(t,x)-\clG^\eps(t_0,a_0)|
&\le \sup_{(x,t) \in B} \int_{x-\alpha'}^{x+\alpha} |\clG^\eps(t,u) - \clG^\eps(t_0,a_0)| p_{\delta}(u-x) du\\
&\quad + \sup_{(x,t) \in B} |\tilde R(x,t, t_0, a_0, \eps)|\\
&=  \int_{a_0-\alpha'}^{a_0+\alpha} \sup_{(x,t) \in B}|\clG^\eps(t,w+x-a_0) - \clG^\eps(t_0,a_0)| 
p_{\delta}(w-a_0) dw\\
&\quad + \sup_{(x,t) \in B} |\tilde R(x,t, t_0, a_0, \eps)|.
\end{align*}
Using Lemma \ref{lem:RR} with $\theta = \theta_0$ we obtain, for some $C = C(\theta, T, K_1,r)$,
\begin{align*}
&\|\sup_{(x,t) \in B} |\tilde \clE_{\theta}^\eps(t,x)-\clG^\eps(t_0,a_0)|\|_r\\
& \le \int_{a_0-\alpha'}^{a_0+\alpha} 
\|\sup_{(x,t) \in B}|\clG^\eps(t,w+x-a_0) - \clG^\eps(t_0,a_0)| \|_r p_{\delta}(w-a_0) dw +
C \eps^{1/(16\theta_0)}.
\end{align*}
Note that for $w \in [a_0-\alpha', a_0+\alpha]$, and $\eps$ sufficiently small,
\begin{align*}
|(w+x-a_0) -a_0| &\le |x-a_0| + |w-a_0| \le |x-a_0| + \alpha \vee \alpha'\\
&\le \frac{1}{2}\eps^{1/(4\theta_0)} + \eps^{1/(4\theta)}\frac{1}{\sqrt{2\theta}} \sqrt{\log \eps^{-1}}
+ \sigma \eps^{\frac{1}{2\theta}-\frac{1}{4}} \le \eps^{1/(4\theta_0)},
\end{align*}
where the last inequality follows on noting that  $\frac{1}{4\theta} < \frac{1}{4\theta_0} \le \frac{1}{2\theta_0} - \frac{1}{4}$.
Thus using Lemma \ref{lem:G} with $\theta$ there replaced by $\theta_0$, we have that with
$\beta = \beta(\theta_0)$ as in that lemma, for some $c_1= c_1(T, K_1, \theta_0, r)$, for all $(a_0, t_0) \in K_1 \times [0,T]$ and $\eps$ sufficiently small,
$$
\|\sup_{(x,t) \in B} |\tilde \clE_{\theta}^\eps(t,x)-\clG^\eps(t_0,a_0)|\|_r
\le c_1 (\eps^{\beta} + \eps^{1/(16\theta_0)}).
$$
Using Lemma \ref{lem:G} once more, we also have that for all $(a_0, t_0) \in K_1 \times [0,T]$ and $\eps$ sufficiently small,
$$
\|\sup_{(x,t) \in B} |\clG^\eps(t,x)-\clG^\eps(t_0,a_0)|\|_r \le c_1 \eps^{\beta}.
$$
Combining the last two estimates we have the lemma.
\end{proof}

\subsection{Characterization and Convergence}\label{ssec:martprob}

In this section we complete the proof of Theorem \ref{thm:maincgce}. As a consequence of the tightness established in 
Corollary \ref{cor:tightcont} and the uniform closeness shown in Proposition \ref{lem:unifclose}, it remains to 
characterize the distributional limit points of $\tlE_\theta^\eps$ as $\eps \to 0$.

Recall the random field given as a solution of the stochastic heat equation \eqref{eq:mshe}, whose distribution on the path space $C([0, \infty); \clc(\RR))$ is denoted as $\clQ$.
From \cite{BertiniGiacomin}, $\clQ$ can be characterized through a suitable martingale problem. To describe this, we introduce the following notation.

We denote the coordinate process on $C([0, \infty);\clc(\RR))$ by $\{w_t\}_{t\ge 0}$. Let
$\clD(\RR)$ be the collection of infinitely differentiable functions on $\RR$ with compact support,  equipped with the inductive limit topology, constructed from the
Fr\'{e}chet spaces $C_K^{\infty}$ of functions in $\clD(\RR)$ with support in $K$, equipped with the seminorms $p_{m,K}(\varphi) = \sup_{x \in K}|\partial^m\varphi(x)/\partial x^m|$, $m \in \NN_0$, where $K$ ranges over all compact sets in $\RR$.
 We denote its strong topological dual by $\clD'(\RR)$. If $f \in \clD'(\RR)$ and $\vp \in \clD(\RR)$, we let $\langle f, \vp \rangle$ denote $f$ applied to $\vp$. In particular, if $f$ is a locally integrable function on $\RR$ (and therefore in $\clD'(\RR)$),  $\langle f, \vp \rangle = \int_{\RR} f(x)\vp(x)dx$.

Then, from \cite[Proposition 4.11]{BertiniGiacomin}, the probability law $\clQ$ can be characterized as the unique probability measure on $C([0, \infty);\clc(\RR))$, such that
\begin{enumerate}[label = (\alph*)]
\item $\clQ(w_0 \in A) = \PP(e^{\sigma B} \in A)$ for every $A \in \clb(\clc(\RR))$, where $\{B(x), x \in \RR\}$ is a Brownian motion indexed by $\RR$. \label{it:mpa}
\item For every $T>0$, there is an $a>0$ such that
$$\sup_{t \in [0,T]} \sup_{r \in \RR} e^{-a|r|} \int (w_t(r))^2 \clQ(dw) <\infty.$$ \label{it:mpb}
\item Let $\clf_t := \sigma\{w_s: 0 \le s \le t\}$. For all $\varphi \in \clD(\RR)$, under $\clQ$,
\begin{equation}\label{eq:mtphi}
M_{t,\varphi}(w):= \langle w_t, \varphi\rangle - \langle w_0, \varphi\rangle
-\frac{1}{2} \int_0^t \langle w_s, \varphi''\rangle ds,
\end{equation}
\begin{equation}\label{eq:ltphi}
\Lambda_{t,\varphi}(w):= (M_{t,\varphi}(w))^2
-\sigma^2 \int_{(0,t)\times\RR} (w_s(x))^2 \varphi^2(x) ds\, dx,
\end{equation}
are continuous $\clf_t$-local martingales. \label{it:mpc}
\end{enumerate}

Recall the notation $\tlE_\theta^\eps(t,x)$ introduced in \eqref{eq:tEdef0}  and the identity below \eqref{eq:te-def}.  Below we use the notation 
\[
\tlE_t^{\theta,\eps}(x) = \tlE_\theta^\eps(t,x) = \tlE^{t,x,\eps^{1/2\theta},\eps}(t), \quad t \in \half, \quad x \in \RR.
\]
The first result describes the initial data of the limiting field when the initial data of the prelimiting particle system has i.i.d. exponential gaps.

\begin{lemma} \label{lem:initlimit}
Assume that $\{X_i(0)\}_{i \in \ZZ} \sim \gamma_1$. Then $\tlE_\theta^\eps(0,\cdot)$ converges uniformly on compacts, in distribution to $e^{\sigma \tilde B(\cdot)}$, where $\tilde B$ is a standard two-sided Brownian motion.
\end{lemma}

The following lemma shows that prelimit versions of the processes $M_{t,\varphi}$ and $\Lambda_{t,\varphi}$ in \eqref{eq:mtphi} and \eqref{eq:ltphi}, respectively, with {$w_t = \tlE_t^{\theta,\eps}$}, $\theta \in (0,1)$, are martingales. The square integrability of the quadratic variations that ensure that the local martingales are in fact martingales follow from the compact support property of $\varphi$ and arguments analogous to those in the proof of Lemma \ref{lem:discToCont} below. We omit the details.
Recall $Y_i^\eps(s) := X_i^\eps(s) + \sigma\eps^{-1/4}s$, introduced in Lemma \ref{lem:SBPformula}.

\begin{lemma} \label{lem:martprelim} Suppose Assumption \ref{assump:initdiff} holds.
For all $\eps > 0$, $t \in \half$, and $ \vp \in \clD(\RR)$, the following equalities hold:
\begin{equation} \label{eq:prelim1}
\begin{split}
\clM_t^\eps(\vp) & := \langle \tlE_t^{\theta,\eps}, \vp \rangle - \langle \tlE_0^{\theta,\eps}, \vp \rangle - \frac{1}{2} \int_0^t \langle \tlE_s^{\theta,\eps},\vp'' \rangle ds \\
& = -(1-e^{-\sigma\eps^{1/4}})\sum_{i \in \ZZ} \int_0^t \clG^\eps(s,Y_i^\eps(s))(p_{\eps^{1/2\theta}} * \vp)(Y_i^\eps(s))dB_i^\eps(s),
\end{split}
\end{equation}
and 
\begin{equation} \label{eq:prelim2}
\begin{split}
& \clW_t^\eps(\vp) := \clM_t^\eps(\vp)^2 - (1-e^{-\sigma\eps^{1/4}})^2\int_0^t \left( \sum_{i \in \ZZ} \clG^\eps(s,Y_i^\eps(s))^2(p_{\eps^{1/2\theta}} * \vp)(Y_i^\eps(s))^2  \right) ds \\
& = 2(1-e^{-\sigma\eps^{1/4}})^2\sum_{i \in \ZZ} \int_0^t \Bigg( \sum_{j \in \ZZ} \int_0^s \clG^\eps(u,Y_i^\eps(u)) \clG^\eps(s,\tilde{X}_j^\eps(s)) \\
& \hspace{2in}(p_{\eps^{1/2\theta}} * \vp)(Y_i^\eps(u))(p_{\eps^{1/2\theta}} * \vp)(Y_j^\eps(s)) dB_j^\eps(u) \Bigg) dB_i^\eps(s).
\end{split}
\end{equation}
\end{lemma}

In an appropriate sense, the  sum appearing on left side of  \eqref{eq:prelim2} can be replaced by an integral in the limit as $\eps \to 0$. More precisely, we have

\begin{lemma} \label{lem:discToCont}
Assume $\{X_i(0)\}_{i \in \ZZ} \sim \gamma_1$. Fix
$\theta\in(0,1)$, $T<\infty$, and $\vp\in\clD(\RR)$. Then
\begin{equation*}
\EE\Bigg[
\int_0^T
\Bigg|
(1 - e^{-\sigma\eps^{1/4}})^2
\sum_{i \in \ZZ}
\clG^\eps(s,Y_i^\eps(s))^2
(p_{\eps^{1/2\theta}} * \vp)(Y_i^\eps(s))^2
-
\sigma^2\int_{\RR}
\tlE_\theta^\eps(s,z)^2\vp(z)^2\,dz
\Bigg|\,ds
\Bigg]\to 0
\end{equation*}
as $\eps\to0$. 
\end{lemma}

\begin{remark} \label{rem:genconvE} \normalfont
{We briefly comment on the possibility of 
proving convergence of $\tlE_\theta^\eps$ to SHE  for more general initial distributions than $\gamma_1$. As noted in Remark \ref{rem:gencont}, for the collection $\{\tlE_\theta^\eps, \eps > 0\}$ to be tight, it suffices for conditions \ref{it:weaker1}-\ref{it:weaker3} to hold. In addition, to verify the conditions of the martingale characterization for the limit point, the more general initial distribution $\gamma$ would need to satisfy the conclusion analogous to Lemma \ref{lem:initlimit}, namely, D4: $\tlE_\theta^\eps(0,\cdot)$ converges to a suitable continuous random field uniformly on compacts in distribution. It turns out that the proof of Lemma \ref{lem:martprelim} does not require any assumptions on the initial data other than the condition \ref{it:weaker1} already noted (which ensures that the field $\tlE_\theta^\eps$ is well-defined). Thus, the only other non-trivial extension of our arguments needed would be to prove Lemma \ref{lem:discToCont} for more general $\gamma$. This amounts to showing that the quantity $D_3^\eps$, defined by \eqref{eq:D3def}, satisfies D5: $\int_{[0,T]} \EE[|D^{\eps}_3(s)|] ds \to 0$, as $\eps \to 0$, since
the rest of the proof of Lemma \ref{lem:discToCont} only requires  assumptions \ref{it:weaker1}-\ref{it:weaker3}. Note that the above convergence of $D_3^\eps$ to zero says that $\tlE_\theta^\eps$ approximates $\clG^\eps$ as $\eps \to 0$, but in a substantially weaker sense than in Proposition \ref{lem:unifclose}. We leave the topic of providing general conditions on $\gamma$, where the key properties D1-D5 hold, as an open problem.}
\end{remark}

We now give the proofs of the three lemmas presented above.

\begin{proof}[Proof of Lemma \ref{lem:initlimit}]
Let
\begin{align*}
 \clH^{\eps}(x) := \sigma^{-1} \log \clG^{\eps}(0,x)
 = \eps^{1/4}[N_\eps(0,x) - \eps^{-1/2}x], \; x \in \RR.
\end{align*}
Then, recalling the definition of $\gamma_1$, by the functional central limit theorem for scaled Poisson processes one has that
$\clH^{\eps} \to \tilde B$ in distribution, uniformly on compacts, where $\{\tilde B(x): x\in \RR\}$ is a Brownian motion indexed by $\RR$. Thus we have that $\clG^{\eps}(0,\cdot) \to e^{\sigma \tilde B(\cdot)}$, in distribution, uniformly on compacts.
The result now follows from Proposition \ref{lem:unifclose}.
\end{proof}

\begin{proof}[Proof of Lemma \ref{lem:martprelim}]
Fix $\eps > 0$, $\theta \in (0,1)$. Then, with $\delta := \eps^{1/2\theta}$ and $t>0$, from \eqref{eq:Edef}
and \eqref{eq:te-def},  we may write 
\[
\begin{split}
\tlE_\theta^\eps(t,x) & = (1-e^{-\sigma\eps^{1/4}})\sum_{i \in \ZZ} e^{\sigma\eps^{1/4}i - \sigma\eps^{-1/4}x + \frac{1}{2}\sigma^2\eps^{-1/2}(t + \delta)} \psi^{x-\sigma\eps^{-1/4}(t + \delta), \delta}(X_i^\eps(t)),
\end{split}
\]
where 
\[
\psi^{x,\delta}(y) = \psi^{t,x,\delta}(t,y) = \int_{y}^\infty p_{\delta}(z - x)dz.
\]
Fix $\vp \in \clD(\RR)$. By the equality above, 
\begin{equation} \label{eq:Gweak}
\begin{split}
\langle \tlE^{\theta,\eps}_t, \vp \rangle & = (1-e^{-\sigma\eps^{1/4}})\int_{\RR} \left( \sum_{i \in \ZZ} e^{\sigma\eps^{1/4}i-\sigma\eps^{-1/4}x + \frac{1}{2}\sigma^2\eps^{-1/2}(t+\delta)} \psi^{x-\sigma\eps^{-1/4}(t + \delta), \delta}(X_i^\eps(t)) \right)\vp(x)dx \\
& = (1-e^{-\sigma\eps^{1/4}})\sum_{i \in \ZZ} e^{\sigma\eps^{1/4}i} \clL^\eps\vp(t,X_i^\eps(t)),
\end{split}
\end{equation}
where 
\[
\clL^\eps\vp(t,y) := \int_{\RR} \vp(x) e^{-\sigma\eps^{-1/4}x + \frac{1}{2}\sigma^2\eps^{-1/2}(t + \delta)} \psi^{x-\sigma\eps^{-1/4}(t + \delta), \delta}(y) dx.
\]
By It\^{o}'s formula applied to \eqref{eq:Gweak}, 
\begin{equation} \label{eq:GIto1}
\begin{split}
\langle \tlE_t^{\theta,\eps}, \vp \rangle & = \langle \tlE_0^{\theta,\eps}, \vp \rangle + (1-e^{-\sigma\eps^{1/4}})\sum_{i \in \ZZ} e^{\sigma\eps^{1/4}i} \bigg\{ \int_0^t \partial_y\clL^\eps\vp(s,X_i^\eps(s)) dX_i^\eps(s) \\
& \hspace{1.6in} + \int_0^t \left(\partial_t\clL^\eps\vp(s,X_i^\eps(s)) + \frac{1}{2}\partial_{yy}\clL^\eps\vp(s,X_i^\eps(s))\right)ds \bigg\}.
\end{split}
\end{equation}
We claim that
\begin{equation} \label{eq:heat_equals_opdd}
\partial_t \clL^\eps\vp(t,y) + \frac{1}{2}\partial_{yy}\clL^\eps\vp(t,y) = \frac{1}{2}\clL^\eps\vp''(t,y).
\end{equation} 
To prove the claim, write
$$
\clL^\eps\vp(t,y) := \int_{\RR} \vp(x) F(t,x,y) dx.
$$
where
$$F(t,x,y)=e^{-\sigma\eps^{-1/4}x + \frac{1}{2}\sigma^2\eps^{-1/2}(t + \delta)} \psi^{x-\sigma\eps^{-1/4}(t + \delta), \delta}(y).$$
Then the claim will follow using integration by parts and the identity
\begin{equation}\label{eq:1153n}\partial_t F(t,x,y) + \frac{1}{2}\partial_{yy} F(t,x,y) = \frac{1}{2}\partial_{xx} F(t,x,y).\end{equation}
To establish the identity, by a direct calculation
$$
\partial_t F(t,x,y) = e^{-\sigma\eps^{-1/4}x + \frac{1}{2}\sigma^2\eps^{-1/2}(t + \delta)}\left[
\frac{1}{2}\sigma^2\eps^{-1/2}\psi^{x-\sigma\eps^{-1/4}(t + \delta), \delta}(y) - \sigma\eps^{-1/4}
p_{\delta}(y-x+\sigma\eps^{-1/4}(t+\delta))\right
],$$
$$
\frac{1}{2}\partial_{yy} F(t,x,y) = \frac{1}{2}e^{-\sigma\eps^{-1/4}x + \frac{1}{2}\sigma^2\eps^{-1/2}(t + \delta)}
\left(\frac{y-x+\sigma\eps^{-1/4}(t+\delta)}{\delta}\right) p_{\delta}(y-x+\sigma\eps^{-1/4}(t+\delta)),$$
and
\begin{multline*}
\frac{1}{2}\partial_{xx} F(t,x,y) = 
e^{-\sigma\eps^{-1/4}x + \frac{1}{2}\sigma^2\eps^{-1/2}(t + \delta)}\bigg[
\frac{1}{2}\sigma^2\eps^{-1/2}\psi^{x-\sigma\eps^{-1/4}(t + \delta), \delta}(y) - 
\sigma\eps^{-1/4}
p_{\delta}(y-x+\sigma\eps^{-1/4}(t+\delta))\\
+ \frac{(y-x+\sigma\eps^{-1/4}(t+\delta))}{\delta} p_{\delta}(y-x+\sigma\eps^{-1/4}(t+\delta))
\bigg].
\end{multline*}
Then \eqref{eq:1153n} is immediate from the last three formulae.

Using \eqref{eq:heat_equals_opdd},
\begin{equation} \label{eq:pred_term}
\begin{split}
& (1-e^{-\sigma\eps^{1/4}})\sum_{i \in \ZZ} e^{\sigma\eps^{1/4}i} \int_0^t \left(\partial_t\clL^\eps\vp(s,X_i^\eps(s)) + \frac{1}{2}\partial_{yy}\clL^\eps\vp(s,X_i^\eps(s))\right)ds \\
& \quad\quad = (1-e^{-\sigma\eps^{1/4}})\sum_{i \in \ZZ} e^{\sigma\eps^{1/4}i} \int_0^t \frac{1}{2}\clL^\eps\vp''(s,X_i^\eps(s)) ds \\
& \quad\quad = \frac{1}{2} \int_0^t \langle \tlE_s^{\theta,\eps}, \vp'' \rangle ds,
\end{split}
\end{equation}
where the second equality is from \eqref{eq:Gweak}. Furthermore, using the SDE \eqref{eq:asym_eq}, and recalling that $e^{\sigma\eps^{1/4}} = q/p$, 
\begin{align} \label{eq:mart_term}
& \sum_{i \in \ZZ} e^{\sigma\eps^{1/4}i} \int_0^t \partial_y\clL^\eps\vp(s,X_i^\eps(s)) dX_i^\eps(s)\nonumber \\
& \quad = \sum_{i \in \ZZ} e^{\sigma\eps^{1/4}i} \int_0^t \partial_y\clL^\eps\vp(s,X_i^\eps(s)) dB_i^\eps(s)\nonumber \\
& \quad\quad+ \sum_{i \in \ZZ} \int_0^t \partial_y\clL^\eps\vp(s,X_i^\eps(s)) e^{\sigma\eps^{1/4}i} (p dL_{(i-1,i)}^\eps(s) - q dL_{(i,i+1)}^\eps(s))\nonumber \\
& \quad = \sum_{i \in \ZZ} e^{\sigma\eps^{1/4}i} \int_0^t \partial_y\clL^\eps\vp(s,X_i^\eps(s)) dB_i^\eps(s)\nonumber \\
& \quad\quad + \sum_{i \in \ZZ} \int_0^t \Big\{\partial_y\clL^\eps\vp(s,X_{i+1}^\eps(s)) - \partial_y\clL^\eps\vp(s,X_i^\eps(s)) \Big\} e^{\sigma\eps^{1/4}i} q dL_{(i,i+1)}^\eps(s)\nonumber \\ 
& \quad = \sum_{i \in \ZZ} e^{\sigma\eps^{1/4}i} \int_0^t \partial_y\clL^\eps\vp(s,X_i^\eps(s)) dB_i^\eps(s).
\end{align}
Here the second equality follows by summation by parts, noting that $e^{\sigma\eps^{1/4}i}p = e^{\sigma\eps^{1/4}(i-1)}q$. The last line follows by observing the second sum in the preceding expression vanishes, since $L^\eps_{(i,i+1)}(s)$ can only increase at times $s$ when $X_{i+1}^\eps(s) = X_i^\eps(s)$. Directly computing $\partial_y \clL^\eps\vp$, we obtain that, with $Y_i^\eps$ as introduced in the statement of the lemma,  the last line in \eqref{eq:mart_term} is equal to
\begin{align} \label{eq:mart_term2}
& - \sum_{i \in \ZZ} e^{\sigma\eps^{1/4}i} \int_0^t \int_{\RR} \vp(x) e^{-\sigma\eps^{-1/4}x + \frac{1}{2}\sigma^2\eps^{-1/2}(s + \delta)} p_\delta(X_i^\eps(s) - x + \sigma\eps^{-1/4}(s + \delta)) dx dB_i^\eps(s)\nonumber \\
& = - \sum_{i \in \ZZ} e^{\sigma\eps^{1/4}i} \int_0^t e^{-\sigma\eps^{-1/4}X_i^\eps(s) - \frac{1}{2}\sigma^2\eps^{-1/2}s} (p_\delta * \vp)(X_i^\eps(s) + \sigma\eps^{-1/4}s) dB_i^\eps(s)\nonumber \\
& = - \sum_{i \in \ZZ} \int_0^t e^{\sigma\eps^{1/4}N_\eps(s,Y_i^\eps(s) - \sigma\eps^{-1/4}s) - \sigma\eps^{-1/4}Y_i^\eps(s) + \frac{1}{2}\sigma^2\eps^{-1/2}s} (p_\delta * \vp)(Y_i^\eps(s)) dB_i^\eps(s)\nonumber \\
& = - \sum_{i \in \ZZ} \int_0^t \clG^\eps(s,Y_i^\eps(s)) (p_\delta * \vp)(Y_i^\eps(s)) dB_i^\eps(s),
\end{align}
where in the second line we have used \eqref{eq:magic} and in the third line, the fact that $N_\eps(s,Y_i^\eps(s) - \sigma\eps^{-1/4}s)=i$.
In view of \eqref{eq:GIto1}, \eqref{eq:pred_term}, \eqref{eq:mart_term}, and \eqref{eq:mart_term2}, we conclude that
\begin{equation} \label{eq:GIto2}
\begin{split}
\langle \tlE_t^{\theta,\eps}, \vp \rangle & = \langle \tlE_0^{\theta,\eps}, \vp \rangle - (1-e^{-\sigma\eps^{1/4}})\sum_{i \in \ZZ} \int_0^t \clG^\eps(s,Y_i^\eps(s)) (p_\delta * \vp)(Y_i^\eps(s)) dB_i^\eps(s) \\
& \hspace{1in} + \frac{1}{2} \int_0^t \langle \tlE_s^{\theta,\eps}, \vp'' \rangle ds.
\end{split}
\end{equation}
This proves the identity in \eqref{eq:prelim1}.

Finally,  \eqref{eq:prelim2} is a straightforward consequence of \eqref{eq:prelim1}, by applying It\^{o}'s formula to  $\clM_t^\eps(\vp)^2$.

\end{proof}

\begin{proof}[Proof of Lemma \ref{lem:discToCont}]
Set
$
\psi_\eps:=p_{\eps^{1/(2\theta)}}*\vp .
$
For $s\in[0,T]$, define
\[
A_\eps(s)
:=
(1-e^{-\sigma\eps^{1/4}})^2
\sum_{i\in\ZZ}
\clG^\eps(s,Y_i^\eps)^2
\psi_\eps(Y_i^\eps)^2
-
\sigma^2
\int_{\RR}\tlE_\theta^\eps(s,z)^2\vp(z)^2\,dz .
\]
We write
\[
A_\eps(s)=D_1^\eps(s)+D_2^\eps(s)+D_3^\eps(s),
\]
where
\begin{align}
D_1^\eps(s)
&:=
(1-e^{-\sigma\eps^{1/4}})^2
\sum_{i\in\ZZ}
\clG^\eps(s,Y_i^\eps)^2
\psi_\eps(Y_i^\eps)^2
-
\sigma^2\int_{\RR}\clG^\eps(s,z)^2\psi_\eps(z)^2\,dz,\\
D_2^\eps(s)
&:=
\sigma^2\int_{\RR}
\clG^\eps(s,z)^2\big(\psi_\eps(z)^2-\vp(z)^2\big)\,dz,\\
D_3^\eps(s)
&:=
\sigma^2\int_{\RR}
\big(\clG^\eps(s,z)^2-\tlE_\theta^\eps(s,z)^2\big)\vp(z)^2\,dz . \label{eq:D3def}
\end{align}
It suffices to prove that, for $j=1,2,3$,
\begin{equation}\label{eq:Djgoal}
\int_0^T \EE\left[|D_j^\eps(s)|\right]\,ds\to0 .
\end{equation}
In the following, the constants $c_i, i \in \NN,$ are independent of $\eps \in (0,1]$ and $s \in [0,T]$.

We first consider $D_1^\eps$. 
Observe that
\[
\begin{split}
& \sum_{i \in \ZZ} \clG^\eps(s,Y_i^\eps(s))^2 \psi_\eps(Y_i^\eps(s))^2 \\
& = \sum_{i \in \ZZ} e^{2\sigma\eps^{1/4}N_\eps(s,Y_i^\eps(s) - \sigma\eps^{-1/4}s) - 2\sigma\eps^{-1/4}Y_i^\eps(s) + \sigma^2 \eps^{-1/2} s} \psi_\eps(Y_i^\eps(s))^2 \\
& = \sum_{i \in \ZZ} e^{2\sigma\eps^{1/4}i - 2\sigma\eps^{-1/4}Y_i^\eps(s) + \sigma^2 \eps^{-1/2} s} \psi_\eps(Y_i^\eps(s))^2,
\end{split}
\]
here noting that $N_\eps(s,Y_i^\eps(s) - \sigma\eps^{-1/4}s) = N_\eps(s,X_i^\eps(s)) = i$. Applying Lemma \ref{lem:SBPformula} with $F(s,z) = \psi_\eps(z)^2$, we obtain 
\[
(1 - e^{-2\sigma\eps^{1/4}}) \sum_{i \in \ZZ} \clG^\eps(s,Y_i^\eps(s))^2 \psi_\eps(Y_i^\eps(s))^2 = \int_{\RR} \clG^\eps(s,z)^2(2\sigma\eps^{-1/4} \psi_\eps(z)^2 - \partial_z [\psi_\eps(z)^2])dz.
\]
Hence,
\[
\begin{split}
& (1 - e^{-\sigma\eps^{1/4}})^2\sum_{i \in \ZZ} \clG^\eps(s,Y_i^\eps(s))^2 \psi_\eps(Y_i^\eps(s))^2 \\
& = \frac{(1 - e^{-\sigma\eps^{1/4}})^2 2\sigma\eps^{-1/4}}{(1 - e^{-2\sigma\eps^{1/4}})} \int_{\RR} \clG^\eps(s,z)^2 \psi_\eps(z)^2 dz - \frac{(1 - e^{-\sigma\eps^{1/4}})^2}{(1 - e^{-2\sigma\eps^{1/4}})} \int_{\RR} \clG^\eps(s,z)^2 \partial_z [\psi_\eps(z)^2] dz.
\end{split}
\]
Therefore, 
\begin{equation} \label{eq:D1bd}
|D_1^\eps(s)| \leq \alpha_1(\eps) \int_{\RR} \clG^\eps(s, z)^2 \psi_\eps(z)^2 dz + \alpha_2(\eps) \int_{\RR} \clG^\eps(s,z)^2 \partial_z [\psi_\eps(z)^2] dz,
\end{equation}
where
\[
\alpha_1(\eps) = \left| \frac{(1 - e^{-\sigma\eps^{1/4}})^2 2\sigma\eps^{-1/4}}{(1 - e^{-2\sigma\eps^{1/4}})} - \sigma^2 \right|, \quad\quad \alpha_2(\eps) = \frac{(1 - e^{-\sigma\eps^{1/4}})^2}{(1 - e^{-2\sigma\eps^{1/4}})}.
\]
Note that $\alpha_1(\eps)$ and $\alpha_2(\eps)$ both converge to zero as $\eps \to 0$. By Lemma \ref{lem:Gmoments}, followed by Jensen's inequality, for any $s \in [0,T]$,
\begin{multline}
\EE\left[\int_{\RR} \clG^\eps(s, z)^2 \psi_\eps(z)^2 dz\right]  \leq \int_{\RR} c_1 e^{c_2|z|} \psi_\eps(z)^2 dz \\
 = \int_{\RR} c_1 e^{c_2|z|} \left( \int_{\RR} p_{\eps^{1/2\theta}}(u - z)\vp(u)du \right)^2 dz \\ 
 \leq \int_{\RR} \vp(u)^2 \left( \int_{\RR} c_1 e^{c_2|z|}p_{\eps^{1/2\theta}}(u - z)dz \right) du \leq \int_{\RR} \vp(u)^2 c_3 e^{c_4|u|} du.\label{eq:1136q}
\end{multline}
The last bound is clearly finite since $\vp$ is smooth and compactly supported. Additionally, note that 
$$
\partial_z[\psi_\eps(z)^2] = 2\psi_\eps(z) \partial_z\psi_\eps(z) = 2(p_{\eps^{1/2\theta}} * \vp(z)) (p_{\eps^{1/2\theta}} * \partial_z\vp(z)).
$$
Since $\partial_z \vp$ is bounded and compactly supported, a similar estimate as \eqref{eq:1136q} holds for the second integral in \eqref{eq:D1bd}. From these observations, we conclude
\[
\int_0^T\EE\left[|D_1^\eps(s)|\right]\,ds\to0 .
\]

Next consider $D_2^\eps$. By Lemma \ref{lem:Gmoments} again, noting that $\psi_\eps$ and $\vp$ are uniformly bounded in $\eps$, we have,
\[
\int_0^T\EE\left[|D_2^\eps(s)|\right]\,ds \leq T\,\int_{\RR} c_5 e^{c_6|z|} |\psi_\eps(z) - \vp(z)|\,dz.
\]
The last line converges to zero, by standard computations on approximate identities obtained from the heat kernel.

It remains to treat $D_3^\eps$. Choose a compact set $K\subset\RR$ such that
$\vp(z)=0$ for $z\notin K$. Since
\[
\clG^\eps(s,z)^2-\tlE_\theta^\eps(s,z)^2
=
\big(\clG^\eps(s,z)-\tlE_\theta^\eps(s,z)\big)
\big(\clG^\eps(s,z)+\tlE_\theta^\eps(s,z)\big),
\]
the Cauchy-Schwarz inequality gives
\[
\begin{split}
\EE\left[|D_3^\eps(s)|\right]
&\le
\sigma^2\|\vp\|_\infty^2
\int_K
\left\|\clG^\eps(s,z)-\tlE_\theta^\eps(s,z)\right\|_2
\left\|\clG^\eps(s,z)+\tlE_\theta^\eps(s,z)\right\|_2\,dz .
\end{split}
\]
By Lemma~\ref{lem:Gmoments}, and Jensen's inequality applied to the following identity (from Lemma \ref{lem:Erep})
\[
\tlE_\theta^\eps(s,z)
=
\int_{\RR}\clG^\eps(s,u)p_{\eps^{1/(2\theta)}}(z-u)\,du,
\]
there is a constant $c<\infty$ such that
\[
\sup_{\eps\in(0,1]}\sup_{0\le s\le T}\sup_{z\in K}
\left\|\clG^\eps(s,z)+\tlE_\theta^\eps(s,z)\right\|_2
\le c .
\]
Consequently,
\[
\EE\left[|D_3^\eps(s)|\right]
\le
c\int_K
\left\|\clG^\eps(s,z)-\tlE_\theta^\eps(s,z)\right\|_2\,dz .
\]
By Proposition~\ref{lem:unifclose}, with $r=2$,
\[
\left\|
\sup_{0\le s\le T}\sup_{z\in K}
|\clG^\eps(s,z)-\tlE_\theta^\eps(s,z)|
\right\|_2\to0 .
\]
Therefore
\[
\begin{split}
\int_0^T\EE\left[|D_3^\eps(s)|\right]\,ds
&\le
cT|K|
\left\|
\sup_{0\le s\le T}\sup_{z\in K}
|\clG^\eps(s,z)-\tlE_\theta^\eps(s,z)|
\right\|_2
\to0 .
\end{split}
\]

Combining the estimates for $D_1^\eps,D_2^\eps,D_3^\eps$ gives \eqref{eq:Djgoal} for $j=1,2,3$, 
which proves the lemma.

\end{proof}

We now complete the proof of Theorem \ref{thm:maincgce}.

\begin{proof}[Proof of Theorem \ref{thm:maincgce}.]
The uniform closeness assertion in \eqref{eq:uniapprox} follows from Proposition \ref{lem:unifclose}.

In view of Corollary \ref{cor:tightcont}, for each $\theta \in (0,1)$, there exists a $C([0,\infty) \times \RR)$ valued random field
$\clG$ such that, along some subsequence, 
\begin{equation} \label{eq:jointcind}
\tlE_\theta^\eps \to \clG \text{ as } \eps \to 0,
\end{equation}
in distribution, in $C([0,\infty)\times \RR)$.

In order to prove the theorem, it suffices to show that, denoting 
the probability distribution of the subsequential limit $\clG$ 
on $C([0,\infty)\times \RR)$ as $\clQ_0$, the measure $\clQ_0$ solves the martingale problem introduced at the beginning of the section, namely  \ref{it:mpa}-\ref{it:mpc} are satisfied (with $\clQ$ replaced by $\clQ_0$).

Part (a) is immediate from Lemma \ref{lem:initlimit}. Part (b) follows from Lemma \ref{lem:Gmoments} and Fatou's lemma.
Finally we establish part (c).
Let, for $t \ge 0$, $\hat\clf_t := \sigma \{\clG(s, \cdot):0\le s \le t\}$.
Fix $\vp \in \clD(\RR)$.

To establish (c), we need to show that $M_{t,\vp}(\clG)$ and $\Lambda_{t,\vp}(\clG)$ (see \eqref{eq:mtphi}-\eqref{eq:ltphi}) are martingales with respect to the filtration $\{\hat\clf_t\}_{t\ge 0}$. 
Since $\varphi$ has compact support, we have that  $\text{supp } \vp \subset [-K,K]$ for some $K \in \ohalf$. 
This in particular implies that the functionals 
$$
M_{\cdot,\vp}, \Lambda_{\cdot,\vp} : C(\half \times \RR) \to C(\half; \RR) 
$$
are continuous. Consequently, by the distributional limit \eqref{eq:jointcind},
\begin{equation} \label{eq:indist}
(M_{\cdot,\vp}(\tlE_\theta^\eps), \Lambda_{\cdot,\vp}(\tlE_\theta^\eps), \tlE_\theta^\eps) \to (M_{\cdot,\vp}(\clG), \Lambda_{\cdot,\vp}(\clG), \clG) \text{ as } \eps \to 0,
\end{equation}
 in distribution, in $C([0,\infty); \RR)\times C([0,\infty); \RR) \times C([0,\infty)\times \RR)$. 

Next, observe that, with $\clM_t^\eps(\vp)$ and $\clW_t^\eps(\vp)$ defined as in \eqref{eq:prelim1} and \eqref{eq:prelim2}, respectively, 
\begin{equation} \label{eq:equiv}
M_{t,\vp}(\tlE_\theta^\eps) = \clM_t^\eps(\vp),
\end{equation}
and, by Lemma \ref{lem:discToCont}, 
\begin{equation} \label{eq:close}
\begin{split}
& \sup_{t \in [0,T]}|\Lambda_{t,\vp}(\tlE_\theta^\eps) - \clW_t^\eps(\vp)| \\
& \quad\quad \leq \int_0^T \left| (1 - e^{\sigma\eps^{1/4}})^2\sum_{i \in \ZZ} \clG^\eps(s,Y_i^\eps(s))^2(p_{\eps^{1/2\theta}} * \vp)(Y_i^\eps(s))^2 - \sigma^2\int_{\RR} \tlE_\theta^\eps(u)^2 \vp(u)^2 du \right| ds \\
& \quad\quad \to 0 \text{ in } L^1(\PP) \text{ as } \eps \to 0.
\end{split}
\end{equation}
Combining \eqref{eq:indist}, \eqref{eq:equiv}, and \eqref{eq:close}, we conclude that \[
(\clM_\cdot^\eps(\vp), \clW_\cdot^\eps(\vp), \tlE_\theta^\eps) \to (M_{\cdot,\vp}(\clG),  \Lambda_{\cdot,\vp}(\clG), \clG) \text{ as } \eps \to 0,
\]
 in distribution, in $C([0,\infty); \RR)\times C([0,\infty); \RR) \times C([0,\infty)\times \RR)$.
 Let $$\hat \clf^{\eps}_t := \sigma\{X^{\eps}_i(0), i \in \ZZ\}
 \vee \sigma\{B^{\eps}_i(s): 0 \le s \le t\}, \;\; t \ge 0.$$
 By strong existence and pathwise uniqueness (Theorems \ref{thm:exist} and \ref{thm:unique}), 
 $$\hat \clf^{X,\eps}_t:= \sigma\{X^{\eps}_i(s),\; 0 \le s \le t,\; i \in \ZZ\} \subset
 \hat \clf^{\eps}_t,$$
 and by definition of $N_{\eps}$, $\clG^{\eps}$, $\tilde \cle^{\eps}_{\theta}$ in \eqref{eq:rescaled_count}, \eqref{eq:Gdef0}, and \eqref{eq:tEdef0}, we see that
 $$\hat \clf^{\clE,\eps}_t:= \sigma\{\tilde \cle^{\eps}_{\theta}(s,x): 0 \le s \le t, \; x \in \RR\}
 \subset \hat \clf^{X,\eps}_t \subset \hat \clf^{\eps}_t, \; t\ge 0.$$

By \eqref{eq:equiv}, $\clM_t^\eps(\vp)=M_{t,\vp}(\tlE_\theta^\eps)$ is
$\hat\clf^{\clE,\eps}_t$-measurable. Lemma \ref{lem:martprelim} shows that $\clM_\cdot^\eps(\vp)$ is a continuous  martingale with respect to the filtration $\{\hat \clf^{\eps}_t: t\ge 0\}$. Hence, by the tower property, we conclude that $\clM_\cdot^\eps(\vp)
=M_{\cdot,\vp}(\tlE_\theta^\eps)$ is a continuous  martingale with
respect to $\{\hat\clf^{\clE,\eps}_t:t\ge 0\}$. Passing to the limit as $\eps \to 0$, a standard argument shows that 
$M_{\cdot,\vp}(\clG)$ is a continuous  martingale with respect to the filtration $\{\hat \clf_t : t \ge 0\}$.

For the quadratic variation martingale, the process
$\clW_\cdot^\eps(\vp)$ is a continuous  martingale with respect to
$\{\hat\clf^\eps_t:t\ge 0\}$ by Lemma~\ref{lem:martprelim}. The process
$\Lambda_{\cdot,\vp}(\tlE_\theta^\eps)$ is adapted to
$\{\hat\clf^{\clE,\eps}_t:t\ge 0\}$, and \eqref{eq:close} shows that it is
asymptotically equivalent to $\clW_\cdot^\eps(\vp)$ in $C([0,T];\RR)$, for
each $T<\infty$. Using these two observations, we conclude by taking a limit as $\eps \to 0$ that $\Lambda_{\cdot,\vp}(\clG)$ is also a continuous  martingale with respect to
$\{\hat\clf_t:t\ge 0\}$. This proves part \ref{it:mpc}, and hence $\clQ_0$
solves the martingale problem. The result follows.
\end{proof}

\appendix

\section{Proof of Proposition \ref{prop:fin1210}.}
The proof adapts arguments from the proof of 
\cite[Theorem 2.3]{banbudrud2025wp}.
Specifically, to show that $\hKs(i,[0,T])<\infty$  almost surely, one proceeds as in the proof of \cite[Theorem 2.3(i)]{banbudrud2025wp}, and to show that $\cKs(i,[0,T])<\infty$  almost surely, one  proceeds as in the proof of \cite[Theorem 2.3(iii)]{banbudrud2025wp}. The main difference is that, at the various steps in the proof, the Brownian motion $B_1$ associated with the lowest particle needs to be replaced by the sum of $B_1$ and an adapted real-valued continuous process $F$ (given either as $pL_{(0,1)}$ or $qL_{(-1,0)}$). {However, it is easy to accommodate this change for two reasons. First, much of the work in \cite{banbudrud2025wp} goes into establishing certain pathwise estimates, namely Lemma 3.4, the lemmas from Section 4, and Lemma 6.4. But as noted in the beginning of \cite[Section 4]{banbudrud2025wp}), these estimates do not depend on the distributions of the Brownian motions $\{B_i\}$; they hold even if we replace $\{B_i\}$ with an arbitrary collection of continuous, real-valued functions on $\half$. Moreover, in key calculations to prove \cite[Theorem 2.3(i) and (iii)]{banbudrud2025wp}, replacing $B_1$ with $B_1 + F$ only introduces an error of order $F/\sqrt{M}$, where we are ultimately interested in proving some almost sure property in the limit as $M \to \infty$.}

We now explain in detail how these changes are implemented.

Recall the driving process $\hV$ defined in \eqref{eq:Vhatch-def}. Following \cite[Section 4]{banbudrud2025wp}, define for $i,M \in \mathbb{N}$ and $T \in [0,\infty)$, 
\begin{equation} \label{eq:VMdef}
\widehat{\clv}_M^-(i,T) = \inf_{0 \leq s_{i + M - 2} \leq \cdots \leq s_i \leq T} \sum_{j = i}^{i+M-1} (\hV_j(s_{j-1}) - \hV_j(s_j)),
\end{equation}
where by convention $s_{i + M - 1} = 0$ and $s_{i-1}  = T$ in the sum above. 
Also, for $i,M \in \mathbb{N}$ and $T \in [0,\infty)$, define 
\begin{equation} \label{eq:defIstar}
\widehat I_M^*(i,T) := \inf_{0 \leq s \leq T} \widehat X_{i + M}(s).
\end{equation}
Then we have the following lemma. 
\begin{lemma} \label{lem:finite-hatI}
For $i \in \NN$ and $T\in [0,\infty)$,
$\liminf_{M\to \infty}\frac{\widehat I_M^*(i,T)}{\sqrt{M}} = \infty$ a.s.
\end{lemma}

\begin{proof}[Sketch of proof]
The proof is the same as that of \cite[Lemma 6.2]{banbudrud2025wp} except for the definitions of the events
$C_k$ and $D_k$ which instead of using a Brownian motion for the lowest particle, use $\widehat B_1 + p L_{(0,1)}$.  However, due to the $\limsup_{M \to \infty}$ in the definition of $C_K$ and $D_K$ the probabilities of these events remain unchanged on replacing $\widehat B_1 + p L_{(0,1)}$ with $\widehat B_1$, and as a result, the rest of the argument goes through as before.
\end{proof}

\begin{proof}[Sketch of proof of finiteness of $\hKs$] {We proceed as in the proof of \cite[Theorem 2.3(i)]{banbudrud2025wp}, using Lemma \ref{lem:finite-hatI} in place of \cite[Lemma 6.2]{banbudrud2025wp}. Recall also that the pointwise estimates from Lemma 3.4 and Section 4 of  \cite{banbudrud2025wp} do not change when we replace $B_1$ with $B_1 + pL_{(0,1)}$.}
By the same argument establishing  \cite[Equation (6.23)]{banbudrud2025wp}, one sees that, on the event 
$\hKs(i, [0,T]) = \infty$,
$$\limsup_{M\to \infty} \frac{\widehat{\clv}_M^{-}(i,T)}{\sqrt{M}} \le (2+T)\sqrt{T}.
$$
{(Note that we can replace the random variable $\Delta$ from the proof in \cite{banbudrud2025wp} with an arbitrary $T > 0$, since we are only assuming the hypotheses of Theorem 2.3(i) and not Theorem 2.3(ii) in \cite{banbudrud2025wp}.)
Using properties of Brownian last-passage percolation, as in the calculation following \cite[Equation (6.23)]{banbudrud2025wp},} one sees that
$\limsup_{M\to \infty} \frac{\widehat{\clv}_M^{-}(i,T)}{\sqrt{M}} > (2+T)\sqrt{T}$
a.s. Observe that, since we are replacing $\widehat B_1+ pL_{(0,1)}$ with $\widehat B_1$ in the proofs, both these statements depend on the fact that 
$L_{(0,1)}(T)/\sqrt{M} \to 0$, as $M \to \infty$.
Combining the two statements we have that $\widehat K^*(i, [0,T]) < \infty$ a.s.
\end{proof}

\begin{proof}[Sketch of proof of finiteness of $\cKs$] 
{We proceed as in \cite[Section 6.2]{banbudrud2025wp}. First, we comment on how quantities defined there translate to our setting. Observe that the roles of $p$ and $q$ are reversed from \cite[Section 6.2]{banbudrud2025wp}. Let $\sigma = p/q$ and note that $\sigma \in (0,1)$. Define
$\widecheck{\clv}_M^{-}(i,T)$ as in \eqref{eq:VMdef} by replacing $\hV_j$ with $\cV_j$ on the right side. This will play the role of $\clv_M^-(i,T)$ in \cite{banbudrud2025wp}; the quantities $\widecheck{\clv}_M^{-}(i,T)$ and ${\clv}_M^{-}(i,T)$ differ by an error of at most $L_{(-1,0)}(T)$. Also, for $T>0$, $M, i, j \in \NN$, let
\begin{equation} \label{eq:udef}
\widecheck\clu(j,M,T) := \sup_{\{t_\ell^{(k)}\} \in D_{M,j}(i,T)} \sum_{k = i}^{i+M} \ \sum_{\ell = k}^{j+k} \ \cV_{\ell}(t_\ell^{(k)}) - \cV_{\ell}(t_{\ell - 1}^{(k)}).
\end{equation}
Here $D_{M,j}(i,T)$ denotes the set of all doubly indexed sequences 
$$
\mathbf{t} = \{t_\ell^{(k)} : i \leq k \leq i+M, k \leq \ell \leq k + j\} \in \mathbb{R}^{(j+1)(M+1)},
$$
satisfying
$$
t_\ell^{(k)} \leq t_{\ell'}^{(k')} \text{ whenever $k \geq k'$, or $k = k'$ and $\ell \leq \ell'$}, \quad t^{(i+M)}_{i+M} = 0, \quad t^{(i)}_{i+j} = T.
$$
The quantity $\widecheck\clu(j,M,T)$ plays the role of the quantity $\clu(j,M,T)$ defined by \cite[Equation 6.28]{banbudrud2025wp}. The two quantities differ by an error of at most $L_{(-1,0)}(T)$.

Lemma 6.4 from \cite{banbudrud2025wp} gives various pathwise estimates, whose proofs do not depend on the underlying distribution of the Brownian motions. Thus, the lemma still holds in our setting (with the substitutions noted above). 


The proof of finiteness of $\cKs(i,[0,T])$ now proceeds as in the proof of \cite[Theorem 2.3(iii)]{banbudrud2025wp}. The estimates (6.46)-(6.49) from \cite{banbudrud2025wp} can be proved in our setting by the same arguments. In particular, the estimate (6.48) becomes 
\begin{equation} \label{eq:init_ubd}
\widecheck x_{i+M-1} \leq \widecheck X_i(T) - \widecheck{\clv}^-_{M+1}(i,T) + \sum_{j = 1}^\infty \sigma^j \left\{ \widecheck{\clu}(j,M,T) - \widecheck{\clv}_{M+1}^-(j+i,T) \right\}.
\end{equation}
Proceeding as in the argument establishing \cite[Equation (6.52)]{banbudrud2025wp}, we can also show that, almost surely, 
\begin{equation} \label{eq:gaussianstuff}
\limsup_{M \to \infty} \left\{ -\frac{\widecheck{\clv}_{M+1}^-(i,T)}{\sqrt{M+1}} + \sum_{j = 1}^\infty \frac{\sigma^j}{\sqrt{M+1}}\left\{ \widecheck{\clu}(j,M,T) - \widecheck{\clv}_{M+1}^-(j + i,T)\right\}  \right\} < \infty.
\end{equation}
{This argument makes use Lemmas 5.1(i) and 6.5 in their original form (i.e. without replacing $\clu(j,M,T)$ with $\widecheck\clu(j,M,T)$ and $\clv(i,T)$ with $\widecheck\clv(i,T)$) from the same paper, while also noting that in our setting the quantity in the braces differs with the corresponding quantity in \cite{banbudrud2025wp} by an error of order $L_{(-1,0)}(T)/\sqrt{M+1}$, which vanishes almost surely in the limit as $M \to \infty$.} Combining \eqref{eq:init_ubd}, \eqref{eq:gaussianstuff}, and the assumption \eqref{eq:ainit} on the initial data, we conclude that 
\[
\begin{split}
\mathbb{P}\left( \cKs(i,[0,T]) = \infty \right) & \leq \mathbb{P}\left( \text{$\forall M \in \mathbb{N}$}, \frac{\widecheck{x}_{i+M-1}}{\sqrt{M}} \leq \frac{1}{\sqrt{M}} \cdot \text{RHS of \eqref{eq:init_ubd}} \right) \\
& \leq \mathbb{P}\left( \limsup_{M \to \infty} \frac{\widecheck{x}_{i+M - 1}}{\sqrt{M}} < \infty \right) = 0,
\end{split}
\]
and the result follows.}
\end{proof}

Proposition \ref{prop:fin1210} then follows from these results.\\

\noindent {\bf Acknowledgements.}
Banerjee was partially supported by NSF-CAREER award DMS-2141621.  Budhiraja was partially supported by NSF DMS-2152577 and DMS-2506010. Banerjee, Budhiraja and Rudzis were partially funded by NSF RTG grant DMS-2134107. Part of this work was carried out while Rudzis was a postdoctoral fellow at the University of North Carolina at Chapel Hill. Part of the work was done when Budhiraja was in
residence at the Simons Laufer Mathematical Sciences Institute  in
Berkeley, California, during the Fall 2025 semester and support of National Science
Foundation  Grant No. DMS-2424139 is acknowledged.

\bibliographystyle{abbrv}
\bibliography{SHE_ref}

\end{document}